\documentclass{article}

\usepackage{latexsym,amsfonts,amsmath,yfonts}

\usepackage[titletoc,title]{appendix}

\usepackage{bbm}

\usepackage{pgfpages}

\renewcommand{\title}[1]{\centerline{\Large #1}\par\medskip}
\renewcommand{\author}[1]{\medskip\centerline{\large #1}\par \medskip}

\newtheorem{theorem}{Theorem}[section]
\newtheorem{lemma}[theorem]{Lemma}
\newtheorem{corollary}[theorem]{Corollary}

\newtheorem{remark}[theorem]{Remark}

\newtheorem{definition}[theorem]{Definition}

\newenvironment{proof}{\removelastskip\par\medskip
\noindent{\em Proof.} \rm}{\penalty-20\null\hfill$\square$\par\medbreak}

\renewcommand{\theequation}{\arabic{section}.\arabic{equation}}

\def\QED{
  \unskip\kern1pc \vadjust{\vskip-\baselineskip}\unhcopy\QEDbox 
  {\pretolerance=500\tolerance=800\parfillskip=0.4pt\smallbreak}}

\newbox\QEDbox \setbox\QEDbox=
\hbox{\penalty -200
\skip0 = 0pt plus 0.05 \hsize	  
      \hskip \skip0 \penalty    0
      \hskip \skip0 \penalty  200 \skip0 = .1\hsize
      \hskip \skip0 \penalty  283
      \hskip \skip0 \penalty  346
      \hskip \skip0 \penalty  400
      \hskip \skip0 \penalty  447
      \hskip \skip0 \penalty  490
      \hskip \skip0 \penalty  529
      \hskip \skip0 \penalty  567
      \hskip \skip0 \penalty  600
      \hskip \skip0 \penalty  632
      \hskip 5\hsize 
      \vadjust{}\hfil
      \rlap{$\sqcup$}$\sqcap$}

\newcommand{\sef}[1]{\eqref{#1}}
\newcommand{\abs}[1]{{\vert #1 \vert}}
\def\babs #1{\bigl\vert #1 \bigr\vert}
\def\gb #1{\bigl( #1 \bigr)}

\def\pa #1#2{\langle #1, #2 \rangle}

\def\DOM #1{\mathcal{D}\bigl( #1 \bigr)}

\def\mcap
{\mathop {\cap}}
\def\mcup
{\mathop {\cup}}

\def\Mcup
{\mathop {\bigcup}}

\def\what{\widehat}
\def\wbar{\overline}
\def\clo{\overline}

\def\novar{\,\cdot\,}
\def\da{\downarrow}

\def\BR{{\mathbb R}}
\def\BC{{\mathbb C}}

\def\BN{{\mathbb N}}

\def\F{\mathcal{F}}

\def\SB{\mathcal{E}}

\def\SP{\mathcal{P}}

\def\Re{\mathrm{Re}\,}

\def\({{\rm (}}
\def\){{\rm )}}

\def\Sbar{\wbar{S}}
\def\DQ{Q^{\diamond}}
\def\NBV{{\rm NBV}}

\def\al
{\alpha}
\def\be
{\beta}
\def\de
{\delta}
\def\ep{\epsilon}
\def\ga
{\gamma}

\def\la
{\lambda}

\def\om
{\omega}
\def\ph
{\varphi}
\def\si
{\sigma}
\def\th
{\theta}
\def\ze
{\zeta}

\def\De
{\Delta}
\def\Ga
{\Gamma}
\def\La
{\Lambda}
\def\Om
{\Omega}

\def\DC{C^{\diamond}}

\def\DY{Y^\diamond}

\def\DQ{Q^\diamond}
\def\dy{y^\diamond}

\def\dq{q^{\diamond}}
\def\du{u^{\diamond}}
\def\bast{\star}

\def\BX{X^{\odot\ast}}

\begin{document}

\title{Twin semigroups and delay equations}

\bigskip
\author{O. Diekmann and S.M. Verduyn Lunel}

\begin{center}
{\em Dedicated, with considerable but finite delay,  to John Mallet-Paret\\ on the occasion of his sixtieth birthday}
\end{center}

\begin{abstract}
\noindent
In the standard theory of delay equations, the fundamental solution does not `live' in the state space. To eliminate this age-old anomaly, we enlarge the state space. As a consequence, we lose the strong continuity of the solution operators and this, in turn, has as a consequence that the Riemann integral no longer suffices for giving meaning to the variation-of-constants formula. To compensate, we develop the Stieltjes-Pettis integral in the setting of a norming dual pair of spaces. Part I provides general theory, Part II deals with ``retarded'' equations, and in Part III we show how the Stieltjes integral enables incorporation of unbounded perturbations corresponding to neutral delay equations.
\end{abstract}

\section{Introduction}\label{sec:1}

A delay equation is a rule for extending a function of time towards the future on the basis of the (assumed to be) known past. The shift along the extended function (i.e., the introduction of current-time-specific past) defines a dynamical system. Delay equations come in two kinds: delay differential equations (DDE) \cite{Diek95,HVL93} and renewal equations (RE) \cite{DGMT98, DGHKMT01,DGG07}.

From a PDE oriented semigroup perspective, delay equations are eccentric: one first constructively defines the semigroup and only then determines the generator, in order to relate to an abstract ODE. Subsequently the development of the qualitative theory can, in principle, follow the well-established path of ODE theory, with the variation-of-constants formula as the key instrument to relate solution operators corresponding to (slightly) different rules for extension to each other. Concerning the function space that serves as the state space, this entails two requirements

\begin{itemize}
\item[--] the semigroup of operators defined by shifting along the extended function should be strongly continuous, in order to employ the Riemann integral when giving precise meaning to the variation-of-constants formula;

\item[--] to represent the rule for extension, one should be able to define the value in the point of extension and to change it without changing the value in nearby points.
\end{itemize}
(Incidentally, the so-called fundamental solution has as initial condition the function that is trivial, except in the point of extension where it equals one.)

Unfortunately, the obvious candidate function spaces satisfy one of these requirements, but not both. The standard approach is to sacrifice the second requirement and to make amends in one way or another. In \cite{Diek95} and \cite{DGG07} one starts with the simplest rule for extension and a Banach space $X$ on which the semigroup is strongly continuous. The representation of the rule for extension is facilitated by embedding the 'small' space $X$ into a 'big' space $\BX$, obtained as the dual of the subspace $X^\odot$ of $X^\ast$ on which the adjoint semigroup of operators is strongly continuous. 
Perturbations are bounded maps from $X$ into $\BX$ and the integral is now a weak-star Riemann integral taking values in $\BX$. Since one can show that the values belong to the image of $X$ under the embedding, they can be re-interpreted as elements of $X$. 

The framework of the four spaces $X$, $X^\ast$, $X^\odot$, $\BX$ is stable under perturbations at the generator level that are described by \emph{bounded} maps from $X$ to $\BX$. Thus sun-star calculus yields a satisfactory theory for semilinear problems (see \cite{MR09,MR18} for an alternative approach using integrated semigroups).

As far as we know, this paper is the first attempt to develop the qualitative theory when, instead of the second, we sacrifice the first requirement. Our way of making amends is to define the integral in Gelfand-Pettis spirit.

\smallskip
\noindent{\sl Note on terminology:} In the context of delay equations we call a space of functions of one real variable (time) ``small'' if translation along an (extended) element is continuous and ``big'' if it is not. So the spaces of continuous functions $C\gb{[-1,0],\BR^n}$ and integrable functions $L^1\gb{[-1,0],\BR^n}$ are small, while the spaces of bounded Borel measurable functions $B\gb{[-1,0],\BR^n}$ and  bounded variation functions $NBV\gb{[-1,0],\BR^n}$ are big.

\smallskip

The aim of the present paper is to establish the variation-of-constants formula for a semigroup of linear operators $\{S(t)\}$ on a big state space $Y$ that accommodates the fundamental solution. The motivation has four components:

\begin{itemize}

\item[--] We anticipate that such a formula should hold; indeed, an integrated version was verified in \cite[Theorem III.2.16]{Diek95},
so it seems merely a matter of making sense of the integral.

\item[--] Strong continuity is a blessing, but the need to have it can be a curse; already in 1953 Feller emphasized that measurability and integrability of (in matrix inspired notation)  $t \mapsto y^\ast S(t)y$, for $y$ belonging to $Y$ and for a sufficiently rich collection of $y^\ast$ in the dual space $Y^\ast$, might be a natural starting point for defining integrals \cite{Fel53}; more recently Kunze \cite{Kun11}, building on Feller's ideas, emphasized that it is natural to work with a norming dual pair of spaces, see the beginning of Section \ref{sec:2} below, such as $B(E)$, the space of all bounded measurable function on a measurable space $E$ and $M(E)$, the space of all bounded measures on $E$, in the theory of Markov processes; in delay equations the Markov process is trivial (just aging), but numbers change; can one incorporate the change of numbers via the variation-of-constants formula?

\item[--] For Renewal Equations corresponding to population models, the space NBV of normalized functions of bounded variation is a very natural state space (see \cite[Chapter XI]{Fel76}) with jumps capturing cohorts, cf. \cite{Huy97}.

\item[--] This is a first step towards covering neutral delay equations in Part III. Neutral delay equations correspond to \emph{unbounded} (actually relatively\hfil\break bounded) maps from $X$ to $\BX$ and as a consequence the spaces $X^\odot$ and $\BX$ depend on the particular perturbation; this undermines the strength (and beauty) of sun-star calculus.

\end{itemize}

We shall heavily exploit that the extension can be defined in terms of the solution of a finite dimensional renewal equation, for which the powerful (Lebesgue) integration theory of real valued functions provides a wealth of results. In other words, we exploit that the rule for extension is represented by an operator with finite dimensional range (so abstract delay equations are not (yet) included). But the variation-of-constants formula itself involves an abstract integral. To define it, we fine-tune the Pettis integral developed by Kunze \cite{Kun11} in the context of a norming dual pair of spaces. 

In Sections 2--4 we introduce twin semigroups defined on a norming dual pair of spaces and we show how Retarded Functional Differential Equations (RFDE), with the space of bounded measurable functions as the state space, fit into this framework. In the second part, Sections 5--7, we deal with bounded finite rank perturbations of twin semigroups and show that the theory covers both RFDE and Renewal Equations (RE) with ``smooth" kernels. In the third and final part we turn to relatively bounded (but still finite rank) perturbations. We use ``cumulative output'' \cite{DGT93} and the Stieltjes integral to extend our approach to cover Neutral Functional Differential Equations (NFDE) and RE with bounded variation kernels.

\vspace{1.5cm}
\centerline{\Large\bf Part I: Twin semigroups}
\bigskip

\section{Twin semigroups on a norming dual pair}\label{sec:2}
\setcounter{equation}{0}

Conceptually, the linear space $Y$ is the state space for the dynamical systems that we want to study and the linear space $\DY$ is an auxiliary space that helps us to perform such studies. But this difference in role is more or less hidden in the linear situation considered in this paper (it will clearly manifest itself in follow-up work on nonlinear problems that we plan to do). A related remark is that our formulation employs the field $\BR$ of real numbers, even though conceptually there is no difference with vector spaces over the field $\BC$ of complex numbers (also see the beginning of Section \ref{sec:5}).

Two Banach spaces $Y$ and $\DY$ are called a {\sl norming dual pair} (cf. \cite{Kun11}) if a bilinear map
\[\pa{\novar}{\novar} : \DY \times Y \to \BR\]
exists such that, for some $M \in [1,\infty)$,
\[\abs{\pa{\dy}{y}} \le M \|\dy\|\|y\|\]
and, moreover,
\begin{align*}
\|y\| &= \sup \Bigl\{|\pa{\dy}{y}| \mid \dy \in \DY,\ \|\dy\| \le 1\,\Bigr\}\\
\|\dy\| &= \sup \Bigl\{|\pa{\dy}{y}| \mid y \in Y,\ \|y\| \le 1\,\Bigr\}.
\end{align*}
So we can consider $Y$ as a closed subspace of $Y^{\diamond\ast}$ and $\DY$ as a closed subspace of $Y^{\ast}$ and both subspaces are necessarily weak$^\ast$ dense since they separate points. The collection of linear functionals $\DY$ defines a weak topology on $Y$, denoted by $\si(Y,\DY)$. The corresponding locally convex topological vector space is denoted by $\gb{Y,\si(Y,\DY)}$. While we denote the dual space of a Banach space $Z$ by adding a star, so by $Z^\ast$, we shall denote the dual space of such topological vector spaces by adding an acute accent. A crucial point is that the dual space $\gb{Y,\si(Y,\DY)}'$ is (isometrically isomorphic to) $\DY$ \cite[Theorem 3.10]{Rudin91}. So if a linear functional on $Y$ is continuous with respect to the topology induced by $\DY$, it can be (uniquely) represented by an element of $\DY$. And please note the symmetry: in the last five sentences one can replace $Y$ and $\DY$ by $\DY$ and $Y$!

A {\sl twin operator} $L$ on a norming dual pair $(Y,\DY)$ is a bounded bilinear map from $\DY \times Y$ to $\BR$ that defines both a bounded linear map from $Y$ to $Y$ and a bounded linear map from $\DY$ to $\DY$. More precisely,
\[L : \DY \times Y \to \BR\qquad (\dy,y) \mapsto \dy Ly\]
is such that 
\begin{itemize}
\item[(i)] for some $C > 0$ the inequality
\begin{equation}\label{eq:bnd-bilin-map}
\abs{\dy L y} \le C \|\dy\|\,\|y\|
\end{equation}
holds for all $y \in Y$ and $\dy \in \DY$;
\item[(ii)] for given $y \in Y$ the map $\dy \mapsto \dy Ly$ is continuous as a map from $\gb{\DY, \si(\DY,Y)}$ to $\BR$ and hence there exists $Ly \in Y$ such that
\begin{equation}\label{eq:bnd-bilin-map-a}
\pa{\dy}{Ly} = \dy Ly
\end{equation}
for all $\dy \in \DY$;
\item[(iii)] for given $\dy \in \DY$ the map $y \mapsto \dy Ly$ is continuous as a map from $\gb{Y,\si(Y,\DY)}$ to $\BR$ and hence there exists $\dy L \in \DY$ such that
\begin{equation}\label{eq:bnd-bilin-map-b}
\pa{\dy L}{y} = \dy Ly
\end{equation}
for all $y \in Y$.
\end{itemize}
So all three maps are denoted by the symbol $L$, but to indicate on which space $L$ acts we write, inspired by \cite{Fel53} which, in turn, is inspired by matrix notation, either $\dy Ly$, $Ly$ or $\dy L$. As a concrete example, consider the identity operator. It maps $(\dy,y)$ to $\pa{\dy}{y}$, $y$ to $y$ and $\dy$ to $\dy$.

If our starting point is a bounded linear operator $L : Y \to Y$ then there exists an associated twin operator if and only if the adjoint of $L$ leaves the embedding of $\DY$ into $Y^\ast$ invariant. We express this in words by saying that $L$ extends to a twin operator. Likewise, if our starting point is an operator $L : \DY \to \DY$ then $L$ extends to a twin operator if and only if the adjoint of $L$ leaves the embedding of $Y$ into $Y^{\diamond\ast}$ invariant. So a twin operator on a norming dual pair is reminiscent of the combination of a bounded linear operator on a reflexive Banach space and its adjoint, whence the adjective ``twin''.

The composition of bounded bilinear maps is, in general, not defined. But for twin operators it is! Indeed, if $L_1$ and $L_2$ are both twin operators on the norming dual pair $(Y,\DY)$, we define the composition $L_1L_2$ by
\begin{equation}\label{eq:bnd-bilin-map-comp}
\dy L_1L_2y := \pa{\dy L_1}{L_2y}.
\end{equation}
Note that this definition entails that $L_1L_2$ acts on $Y$ by first applying $L_2$ and next $L_1$, whereas $L_1L_2$ acts on $\DY$ by first applying $L_1$ and next $L_2$.

\begin{definition}\label{def:2.1}
A family $\{S(t)\}_{t \ge 0}$ of twin operators on a norming dual pair $(Y,\DY)$ is called a {\sl twin semigroup} if
\begin{itemize}
\item[i\)] $S(0) = I$, and\ \ $S(t+s) = S(t)S(s)\quad$ for $t,s \ge 0$;
\item[ii\)] there exist constants $M \ge 1$ and $\om \in \BR$ such that
\[\abs{\dy S(t)y} \le Me^{\om t} \|y\|\|\,\dy\|;\]
\item[iii\)] for all $y \in Y$, $\dy \in \DY$ the function
\[t \mapsto \dy S(t)y\]
is measurable;
\item[iv\)] for $\Re \la > \om$ \(with $\om$ as introduced in ii\)\) there exists a twin operator $\Sbar(\la)$ such that
\begin{equation}\label{eq:2.2}
\dy \Sbar(\la)y = \int_0^\infty e^{-\la t} \dy S(t)y\,dt.
\end{equation}
\end{itemize}
\end{definition}

Note that the combination of $ii)$ and $iii)$ allows us to conclude that the right hand side of $\sef{eq:2.2}$ defines a bounded bilinear map, but not that it defines a twin operator. Hence $iv)$ is indeed an additional assumption.

We call $\Sbar(\la)$ defined on $\{\la \mid \Re \la > \om\}$ the {\sl Laplace transform} of $\{S(t)\}$. It actually suffices to assume that the assertion of iv) holds for $\la = \la_0$ with $\Re \la_0 > \om$. This assumption allows us to introduce the multi-valued operator
\begin{equation}\label{eq:2.3a}
C = \la_0 I - \Sbar(\la_0)^{-1}
\end{equation}
on $Y$ and next define the function $\la \mapsto \Sbar(\la)$ by
\begin{equation}\label{eq:2.3}
\Sbar(\la) = (\la I - C)^{-1}
\end{equation}
on an open neighbourhood of $\la_0$. As Proposition A.2.3 of \cite{Haa06} shows, the function $R$ is holomorphic with Taylor series given by
\[\Sbar(\la) = \sum_{k=0}^\infty (\mu - \la)^k \Sbar(\mu)^{k+1}\]
and the resolvent identity
\[\Sbar(\la) - \Sbar(\mu) = (\mu - \la)\Sbar(\la)\Sbar(\mu)\]
holds. In Proposition 5.2 of \cite{Kun11} these facts are used to prove that
\[\Om_0 := \bigl\{ \la \mid \Sbar(\la) \hbox{ is a twin operator and } \sef{eq:2.2} \hbox{ holds } \bigr\}\]
contains the half plane $\{ \la \mid \Re \la > \om \}$.

In Definition 2.6 of \cite{Kun09} Kunze calls $C$ the \emph{generator} of the semigroup provided the Laplace transform is injective and hence $C$ is single-valued. Here we adopt a more pliant position and call $C$ the generator even when it is multi-valued. Note that we might equally well call the operator $\DC$, defined on $\DY$ as the inverse of the Laplace transform, but now considered as an operator mapping $\DY$ into $\DY$, the generator. As long as one realises that the two have the same twin operator as their resolvent, this cannot lead to confusion. By combining \cite[Prop. 2.7]{Kun09} and \cite[Thm. 5.4]{Kun11} one obtains that the twin semigroup is uniquely determined by the generator if both $C$ and $\DC$ are single-valued.

\medskip\noindent

Focusing on $\{S(t)\}_{t \ge 0}$ as a semigroup of bounded linear operators on Y, we now list some basic results from \cite{Kun11}. For completeness we provide proofs, even though these are, in essence, copied from \cite{Kun11}.

\begin{lemma}\label{lem:2.2}
The following statements are equivalent
\begin{itemize}
\item[1.] $y \in \DOM{C}$ and $z \in Cy$;
\item[2.] there exist $\la \in \BC$ with $\Re \la > \om$ and $\om$ as introduced in ii) and $y,z \in Y$ such that
\begin{equation}\label{eq:2.4}
y = \Sbar(\la)(\la y - z)
\end{equation}
\item[3.] $y,z \in Y$ and for all $t > 0$
\begin{equation}\label{eq:2.5}
\int_0^t S(\tau)z\,d\tau = S(t)y - y.
\end{equation}
\end{itemize}
Here it should be noted that item 3. includes the assertions
\begin{itemize}
\item[--] the integral $\int_0^t S(\tau)z\,d\tau$ defines an element of $Y$ (even though at first it only defines an element of $Y^{\diamond\ast}$);
\item[--] the integral $\int_0^t S(\tau)z\,d\tau$ does not depend on the choice of $z \in Cy$ in case $C$ is multi-valued.
\end{itemize}
\end{lemma} 

\begin{proof}
The observation $y \in \DOM{C}$ if and only if $y = \Sbar(\la)\tilde y$ and in that case $(\la I - C)y = \tilde y$, establishes the equivalence of the items 1. and 2.

The integrals below derive their meaning by pairing the integrand with arbitrary $\dy \in \DY$. But in order to enhance readability, we do not actually write these pairings. Let $\Re \la > \om$. The identity
\begin{equation}\label{eq:2.6}
\int_0^t e^{-\la \tau}S(\tau)y\,d\tau = \Sbar(\la)\gb{y - e^{-\la t}S(t)y}
\end{equation}
follows straightforwardly by considering $\int_0^t = \int_0^\infty - \int_t^\infty$ and next shifting the integration variable in the second integral over $t$. If we multiply \sef{eq:2.6} by $\la$, assume that 2. holds, and use \sef{eq:2.4} to rewrite $\la R(\la)y$, we obtain
\[\la \int_0^t e^{-\la \tau}S(\tau)y\,d\tau = y + \Sbar(\la)\gb{z - \la e^{-\la t}S(t)y}.\]
Next use \sef{eq:2.6} with $y$ replaced by $z$, as well as the fact that $S(t)$ and $\Sbar(\la)$ commute, to arrive at
\begin{equation*}
\la \int_0^t e^{-\la \tau}S(\tau)y\,d\tau = y + \int_0^t e^{-\la \tau} S(\tau)z\,d\tau - e^{-\la t}S(t)\Sbar(\la)(\la y - z)
\end{equation*}
or, on account of \sef{eq:2.4}
\begin{equation}\label{eq:2.7}
\la \int_0^t e^{-\la \tau}S(\tau)y\,d\tau = y + \int_0^t e^{-\la \tau} S(\tau)z\,d\tau - e^{-\la t}S(t)y.
\end{equation}
The identity \sef{eq:2.7} does not involve any improper integral, so we can extend by analytic continuation and, in particular, take $\la = 0$. This yields \sef{eq:2.5}. Thus we have proved that 2. implies 3.

Finally, assume that 3. holds. Then
\begin{align*}
\la \Sbar(\la)y - y &= \int_0^\infty \la e^{-\la \tau}\gb{S(\tau)y - y}\,d\tau\\
&= \int_0^\infty \la e^{-\la \tau}\int_0^\tau S(\si)z\,d\si\,d\tau\\
&= \int_0^\infty \int_\si^\infty \la e^{-\la \tau}\,d\tau\,S(\si)z\,d\si\\
&= \int_0^\infty e^{-\la \si}S(\si)z\,d\si = \Sbar(\la)z
\end{align*}
which amounts to \sef{eq:2.4}
\end{proof}

\begin{lemma}\label{lem:2.3} 
For all $t > 0$ and $y \in Y$, we have $\int_0^t S(\tau)y\,d\tau \in \DOM{C}$ and 
\begin{equation}\label{eq:2.8}
S(t)y - y \in C\int_0^t S(\tau)y\,d\tau.
\end{equation}
\end{lemma}

\begin{proof}
Again we omit the pairing with $\dy$. Yet, we keep in mind that the integrals define elements in $Y^{\diamond\ast}$ for which we subsequently check that they are represented by elements in $Y$. Since
$y \in (\la I - C)\Sbar(\la)y$ we have
\[\int_0^t S(\tau)y\,d\tau \in \int_0^t S(\tau)(\la I - C)\Sbar(\la)y\,d\tau\]
and
\begin{align*}
\int_0^t S(\tau)(\la I - C)\Sbar(\la)y\,d\tau &= \la \int_0^t S(\tau)\Sbar(\la)y\,d\tau - \int_0^t S(\tau)C\Sbar(\la)y\,d\tau\\
&= \la \int_0^t S(\tau)\Sbar(\la)y\,d\tau - S(t)\Sbar(\la)y + \Sbar(\la)y,
\end{align*}
where we have used \sef{eq:2.5}. Note that the right hand side is single valued. We claim that the right hand side belongs to $Y$. This is clear for the last two terms. Concerning the first, observe that \sef{eq:2.5} implies that $t \mapsto S(t)y$ is continuous if $y \in \DOM{C}$. Hence we can interpret the integral $\int_0^t S(\tau)\Sbar(\la)y\,d\tau$ as a Bochner integral of a continuous $Y$-valued function.

Since $S(\tau)\Sbar(\la) = \Sbar(\la)S(\tau)$ and $\Sbar(\la)$ is a twin operator we have
\[\int_0^t S(\tau)\Sbar(\la)y\,d\tau = \Sbar(\la)\int_0^t S(\tau)y\,d\tau.\]
So the identity above can be written in the form
\[\int_0^t S(\tau)y\,d\tau = \Sbar(\la)\gb{\la\int_0^t S(\tau)y\,d\tau + y - S(t)y}.\]
Comparing this to \sef{eq:2.4} we conclude that $\int_0^t S(\tau)y\,d\tau \in \DOM{C}$ and that \sef{eq:2.8} holds.
\end{proof}

In the proof of Lemma \ref{lem:2.3} we used the assumption that $\Sbar(\la)$ is a twin operator (cf. Definition \ref{def:2.1}, iv) to prove that the same is true for local integrals of the orbit $t \mapsto S(t)y$ for arbitrary $y \in Y$. In Theorem 5.8 of \cite{Kun11} Kunze proves that these two properties are equivalent.

\medskip

In order to obtain information about the asymptotic behaviour of the twin semigroup $S(t)$, we adapt a result  for strongly continuous semigroups from \cite{AreBat88}. It was observed by Batty in \cite{Batty94} that in case $\si(C) \mcap i\BR = \emptyset$, the asymptotic behaviour actually follows from Korevaar's proof of the Ingham theorem \cite{Korevaar}. Here we adapt this argument from \cite{Batty94} to the case of twin semigroups.

\begin{theorem}\label{thm:Tauber}
Let $S(t)$ be a twin semigroup on a norming dual pair $(Y,\DY)$ and assume that $S(t)$ is bounded. If $\si(C) \mcap i\BR = \emptyset$, then
\begin{equation}\label{eq:Tauber}
\|S(t)C^{-1}\| \to 0 \qquad \hbox{as}\quad t \to \infty.
\end{equation}
As a consequence we have that $S(t)y \to 0$ as $t \to \infty$ for every $y$ in the norm-closure of $\DOM{C}$.
\end{theorem}

\begin{proof}
Let $\Ga_R := \{ z \in \BC \mid |z| = R \}$ and $\Ga_R^{-}$ and $\Ga_R^+$ denote the part of $\Ga_R$ in the, respectively, left and right closed half plane of $\BC$. Define $\Ga_0$ to be a path in the intersection of $\rho(C)$ and the open left half plane connecting $iR$ and $-iR$ such that the closed contour $\Ga$ given by the union of $\Ga_R^+$ and $\Ga_0$ does not encircle any pole of $(zI - C)^{-1}$. 

From Cauchy's Residue Theorem it follows that we can write
\begin{equation}\label{eq:cauchy-id}
\dy S(t)C^{-1}y = -\frac{1}{2\pi i} \int_{\Ga} \gb{1 + \frac{z^2}{R^2}}\dy (zI - C)^{-1}S(t)y\,\frac{dz}{z},
\end{equation}
where the factor $\gb{1 + \frac{z^2}{R^2}}$ is chosen because for $z \in \Ga_R$ the identity
\begin{equation}\label{eq:fudge}
\bigl| 1 + \frac{z^2}{R^2} \bigr| = \frac{2\bigl| \Re z\bigr|}{R}
\end{equation}
holds. Fix $t \ge 0$ and observe that from the identity \sef{eq:2.6} we have for $\Re z \ge 0$
\begin{equation}\label{eq:2.6new}
e^{zt}\int_0^t e^{-z\tau} \dy S(\tau)y\,d\tau = \dy (zI - C)^{-1}\gb{e^{zt}y - S(t)y}.
\end{equation}
Define the entire function $g_t : \BC \to Y$ by
\[g_t(z) := \int_0^t e^{-z\tau} S(\tau)y\,d\tau\]
and use \sef{eq:2.6new} to deduce the identity
\begin{align}\label{eq:cauchy-id2}
&\frac{1}{2\pi i} \int_{\Ga_0}\gb{1 + \frac{z^2}{R^2}}\dy (zI - C)^{-1}S(t)y\,\frac{dz}{z}\nonumber\\
&\qquad\qquad\qquad= \frac{1}{2\pi i} \int_{\Ga_0}\gb{1 + \frac{z^2}{R^2}}e^{zt}\dy (zI - C)^{-1}y\,\frac{dz}{z}\nonumber\\
&\qquad\qquad\qquad\qquad\qquad\qquad -\frac{1}{2\pi i} \int_{\Ga_R^{-}}\gb{1 + \frac{z^2}{R^2}} e^{zt} \dy g_t(z)\,\frac{dz}{z}.
\end{align}
Since along $\Ga_0$ we have $\Re z < 0$, it follows from the dominated convergence theorem that the first integral on the right hand side of \sef{eq:cauchy-id2} tends to zero as $t \to \infty$. 

Using the fact that $|\dy S(t)y| \le M\|\dy\|\,\|y\|$, we have for $z \in \Ga_R^-$,
\[\bigl|e^{zt}\dy g_t(z)\bigr| = \bigl| \int_0^t e^{z(t-\tau)} \dy S(\tau)y\,d\tau \bigr| \le \frac{M}{\bigl|\Re z\bigr|}\|\dy\|\,\|y\|.\]
Similarly, for $z \in \Ga_R^+$
\[\bigl| \dy (z I - C)^{-1}S(t)y \bigr| = \bigl| \int_0^\infty e^{-z\tau} \dy S(t+\tau)y\,d\tau \bigr| \le \frac{M}{\Re z}\|\dy\|\,\|y\|,\]
From the property \sef{eq:fudge} it follows that both the integral over $\Ga_R^-$ in \sef{eq:cauchy-id2} and the integral over $\Ga_R^+$ in \sef{eq:cauchy-id} are bounded. Using these estimates in combination with the identities \sef{eq:cauchy-id} and \sef{eq:cauchy-id2} yields
\begin{equation}\label{eq:cauchy-id3}
\limsup_{t \to \infty}\ \bigl|\dy S(t)C^{-1}y \bigr| \le \frac{2M}{R}\|\dy\|\,\|y\|.
\end{equation}
By letting $R \to \infty$, we conclude \sef{eq:Tauber}.

Since $C^{-1}$ has dense range in the norm-closure of $\DOM{C}$, the final observation follows from \sef{eq:Tauber} and the fact that $S(t)$ is bounded.
\end{proof}

In this paper we will see that our perturbation results are well suited to verify the conditions of Theorem \ref{thm:Tauber} in terms of the given data.

\section{The subspace of strong continuity}\label{sec:3}
\setcounter{equation}{0}

We define the subspace $X$ of $Y$ by
\begin{equation}\label{eq:3.1}
X := \bigl\{y \in Y \mid t \mapsto S(t)y\hbox{ is continuous}\, \bigr\}
\end{equation}
and note, first of all, that the semigroup property of $\{S(t)\}_{t \ge 0}$ yields as an equivalent characterization
\begin{equation}\label{eq:3.2}
X := \bigl\{y \in Y \mid \lim_{t \da 0} \| S(t)y - y \| = 0 \bigr\}.
\end{equation}
As $S(t)$ maps $X$ into $X$, the restriction
\begin{equation}\label{eq:3.3}
T(t) = S(t)\big\vert_{X}
\end{equation}
defines a strongly continuous semigroup $\{T(t)\}_{t \ge 0}$ on the Banach space $X$ ($X$ is norm-closed in $Y$, see Theorem \ref{thm:3.1}). 

\smallskip

The main results of this section are the following theorems.

\begin{theorem}\label{thm:3.1}
The subspace $X$ of strong continuity equals the norm closure of $\DOM{C}$
\[ X = \clo{\DOM{C}}.\]
\end{theorem}

\vspace{.25cm}

\begin{theorem}\label{thm:3.2}
The generator $A$ of the strongly continuous semigroup $\{T(t)\}_{t \ge 0}$ on $X$ is the part of $C$ in $X$.
\end{theorem}

It should be noted here that, as we shall prove below, the generator $A$ is single-valued even if $C$ is a multi-valued map. 

\smallskip
In order to prove Theorems \ref{thm:3.1} and \ref{thm:3.2} we first provide an auxiliary result that is of independent interest, cf. \cite{Cran73}.

\begin{lemma}\label{lem:3.3}
If $y \in \DOM{C}$ then
\[\limsup_{h \da 0} \frac{1}{h}\|S(h)y - y\| < \infty.\]
\end{lemma}

\begin{proof}
By Lemma \ref{lem:2.2} we have for $z \in Cy$ the identity
\[\frac{1}{h}\gb{\dy S(h)y - \pa{\dy}{y}} = \frac{1}{h}\int_0^h \dy S(\tau)z\,d\tau\]
and consequently
\begin{align*}
\frac{1}{h}\bigl|\dy S(h)y - \pa{\dy}{y}\bigr| &\le \frac{1}{h}\int_0^h Me^{\om\tau}\|\dy\|\|z\|\,d\tau\\
&= M\frac{e^{\om h}-1}{\om h}\|\dy\|\|z\|.
\end{align*}
It follows that
\[\frac{1}{h}\|S(h)y - y\| \le M\frac{e^{\om h}-1}{\om h}\|z\|\]
and so
\[\limsup_{h \da 0} \frac{1}{h}\|S(h)y - y\| \le M\|z\|.\]
\end{proof}

\begin{corollary}\label{col:3.4}
The domain of the generator $C$ of the semigroup $\{S(t)\}_{t \ge 0}$ satisfies
\[ \DOM{C} \subset X.\]
\end{corollary}

\begin{lemma}\label{lem:3.5}
For $X$ defined by $\sef{eq:3.1}$ we have
\[X \subset \clo{\DOM{C}}.\]
\end{lemma}

\begin{proof}
For arbitrary $y \in Y$
\begin{align*}
\bigl\|\frac{1}{t}\int_0^t S(\tau)y\,d\tau - y \bigr\| &= \sup_{\|\dy\|\le 1}\bigl|\frac{1}{t}\int_0^t \gb{\dy S(\tau)y - \pa{\dy}{y}}\,d\tau\bigr|\\
&\le \frac{1}{t}\int_0^t \| S(\tau)y - y\|\,d\tau.
\end{align*}
If $y \in X$, then the integrand at the right hand side is a continuous function of $\tau$ vanishing at $\tau = 0$. It follows that in that case the right hand side converges to zero for $t \da 0$. Since
\[\int_0^t S(\tau)y\,d\tau \in \DOM{C},\]
cf. Lemma \ref{lem:2.3}, we conclude that in any $\ep$-neighbourhood of $y \in X$, there is an element of $\DOM{C}$.
\end{proof}

By combining Corollary \ref{col:3.4} and Lemma \ref{lem:3.5} we obtain a proof of Theorem \ref{thm:3.1}. Note that the semigroups $\{S(t)\}_{t \ge 0}$ and $\{T(t)\}_{t \ge 0}$ are {\it intertwined} in the sense that
\begin{equation}\label{eq:3.4}
S(t)y \in (\la I - C)T(t)\Sbar(\la)y.
\end{equation}

\begin{remark}\textup{
It is unclear whether the converse of Lemma \ref{lem:3.3} holds:
\[\limsup_{h \da 0} \frac{1}{h}\|S(h)y - y\| < \infty \implies y \in \DOM{C}?\]
In the rather special case that i) $Y = Y^{\diamond\ast}$ and ii) $\{S(t)\}_{t \ge 0}$ as a semigroup of bounded linear operators on $\DY$ is strongly continuous, this does hold, see e.g. Theorem 3.19 in Appendix II of \cite{Diek95}.}
\end{remark}

\vspace{1cm}
\noindent{\it Proof of Theorem $3.2$.} If $y \in \DOM{C}$ and $z \in Cy \cap X$ then, by Lemma \ref{lem:2.2},
\[T(t)y - y = \int_0^t T(\tau)z\,d\tau\]
and it follows that $t^{-1}(T(t)y - y) \to z$ for $t \da 0$. In particular this shows that $Cy \cap X$ is, when non-empty, a singleton. Moreover, $Ay \in Cy$.

Now assume that $y \in \DOM{A}$ and $Ay = z \in X$. Then
\[T(t)y - y = \int_0^t T(\tau)z\,d\tau\]
and we conclude from Lemma \ref{lem:2.2} that $y \in \DOM{C}$ and $z \in Cy$.
\QED

Note on notation: the analogue of $X$ at the $\diamond$ side we shall denote by $X^{\odot}$. So in this paper
\begin{equation}\label{eq:3.5}
X^{\odot} := \bigl\{ \dy \in \DY \mid \lim_{t \da 0} \| \dy S(t) - \dy\| = 0\, \bigr\}.
\end{equation}

\newpage

\section{RFDE -- Retarded Functional Differential\\ Equations}\label{sec:4}
\setcounter{equation}{0}

We adopt the standard notation $x_t(\th) = x(t+\th)$ and the only slightly less standard notation
\[\pa{\ze}{\ph} := \int_{[0,1]} d\ze(\si)\ph(-\si)\]
for $\ze \in NBV\gb{[0,1],\BR^{n \times n}}$ and $\ph \in B\gb{[-1,0],\BR^n}$.
An equation of the form
\begin{equation}\label{eq:6.1}
\dot x(t) = \pa{\ze}{x_t} = \int_{[0,1]} d\ze(\si)\,x(t-\si)
\end{equation}
is called a RFDE. If we pose an initial value problem, we require \sef{eq:6.1} to hold for $t \ge 0$ and supplement the equation by the initial condition
\begin{equation}\label{eq:6.2}
x(\th) = \ph(\th),\qquad -1 \le \th \le 0,
\end{equation}
for a given function $\ph$. The standard theory assumes that $\ph \in X$ with $X = C\gb{[-1,0],\BR^n}$, but here we allow
\begin{equation}\label{eq:6.3}
\ph \in Y = B\gb{[-1,0],\BR^n}.
\end{equation}
Concerning the given kernel $\ze$ we assume that for $i = 1,\ldots,n$
\begin{equation}\label{eq:6.4}
\ze_i \in \DY = NBV\gb{[0,1],\BR^n},
\end{equation}
where $\ze_i$ is the $i$-th row of the matrix $\ze$.

In Appendix B, it is shown that $Y$ and $\DY$ given by \sef{eq:6.3} and \sef{eq:6.4} form a norming dual pair.

Once we solve \sef{eq:6.1}--\sef{eq:6.2}, we can define a $Y$-valued function $u : [0,\infty) \to Y$ by
\begin{equation}\label{eq:6.5}
u(t)(\th) = x(t+\th; \ph),\qquad -1 \le \th \le 0,\ t \ge 0
\end{equation}
and bounded linear operators $S(t) : Y \to Y$ by
\begin{equation}\label{eq:6.5a}
S(t)\ph = u(t;\ph) = x(t + \novar;\ph).
\end{equation}
The initial condition \sef{eq:6.2} translates into
\begin{equation}\label{eq:6.6}
S(0)\ph = u(0;\ph) = \ph
\end{equation}
and \sef{eq:6.5a} reflects that we define a dynamical system on $Y$ by translating along the function $\ph$ extended according to \sef{eq:6.1}.
Below we show that $\{S(t)\}$ is a twin semigroup and we characterize its generator $C$. But first we present some heuristics.

In order to motivate an abstract ODE for the $Y$-valued function $u$, we first observe that the infinitesimal formulation of the translation rule \sef{eq:6.5} amounts to the PDE
\begin{equation}\label{eq:6.7}
\frac{\partial u}{\partial t} - \frac{\partial u}{\partial \th} = 0.
\end{equation}
We need to combine this with \sef{eq:6.1}, in terms of $u(t)(0) = x(t)$, and we have to specify the domain of definition of the derivative with respect to $\th$. The latter is actually rather subtle. An absolutely continuous function has almost everywhere a derivative and when the function is Lipschitz continuous this derivative is bounded. Thus a Lipschitz function specifies a unique $L^\infty$-equivalence class by the process of differentiation. But {\sl not} a unique element of $Y$. In fact the set
\begin{equation}\label{eq:6.8}
C\psi = \bigl\{ \psi' \in Y \mid \psi(\th) = \psi(-1) + \int_{-1}^\th \psi'(\si)\,d\si,\ \psi'(0) = \pa{\ze}{\psi} \bigr\}
\end{equation}
is, for a given Lipschitz continuous function $\psi$, very large indeed. Nota bene that the condition $\psi'(0) = \pa{\ze}{\psi}$ takes care of \sef{eq:6.1} and that, in the context of the space $Y$, we can simply take this as the definition of $\psi'(0)$ without having to worry about an influence of this choice on $\psi'(\si)$ for $\si$ near zero (such in sharp contrast to the space $X$ of continuous functions). Anyhow, we define $C$ as a multi-valued, unbounded, operator on $Y$ by
\begin{equation}\label{eq:6.9}
\DOM{C} = Lip\gb{[-1,0],\BR^n},\qquad C\psi \hbox{ given by } \sef{eq:6.8}.
\end{equation}
We claim that \sef{eq:6.1}--\sef{eq:6.2} and \sef{eq:6.5} correspond to 
\begin{equation}\label{eq:6.10}
\frac{du}{dt} \in Cu.
\end{equation}
To substantiate this claim, we shall first derive (following essentially Section I.2 of \cite{Diek95}) a representation of the solution of \sef{eq:6.1}--\sef{eq:6.2} in terms of $\ph$, $\ze$ and the resolvent $\rho$ of $\ze$, next verify that $\{S(t)\}_{t \ge 0}$ defined by \sef{eq:6.5a} is a twin semigroup and, finally, that $C$ is the corresponding generator in the sense of \sef{eq:2.3}--\sef{eq:2.2}.

\begin{lemma}\label{lem:6.1}
The solution of \sef{eq:6.1}--\sef{eq:6.2} is given explicitly by
\begin{align}\label{eq:6.12}
x(t;\ph) &= \bigl(1 + \int_0^t \rho(\si)\,d\si\bigr)\ph(0) + \int_0^1 \Bigl\{ \ze(t+\si) - \ze(\si) +\nonumber\\ &\qquad\qquad\qquad\qquad\int_0^t \rho(\tau)\gb{\ze(t-\tau+\si) - \ze(\si)}\,d\tau\Bigr\}\ph(-\si)\,d\si,
\end{align}
where the {\sl resolvent} $\rho$ of the kernel $\ze$ is the unique solution of
\begin{equation}\label{eq:6.13}
\rho \ast \ze + \ze = \rho = \ze \ast \rho + \ze
\end{equation}
and hence given by
\begin{equation}\label{eq:6.14}
\rho = \sum_{l=1}^\infty \ze^{l\ast}.
\end{equation}
\end{lemma}

\begin{proof}
(See Section I.2 of \cite{Diek95} for more detail). We integrate \sef{eq:6.1} from $0$ to $t$ and interchange the order of the two integrals at the right hand side. This yields
\begin{equation}\label{eq:6.15}
x = \ze \ast x + f
\end{equation}
with
\begin{align}\label{eq:6.16}
f(t) &= \ph(0) + \int_0^t\gb{\int_s^1 d\ze(\th)\,\ph(s-\th)}\,ds\nonumber\\
     &= \ph(0) + \int_0^1 \gb{\ze(t+\si)-\ze(\si)}\ph(-\si)\,d\si.
\end{align}
The solution of \sef{eq:6.15} is given by
\begin{equation}\label{eq:6.17}
x = f + \rho * f
\end{equation}
which leads, after another change of integration order, to \sef{eq:6.12}.
\end{proof}

Please observe that $x$ depends on the value of $\ph$ in $\th = 0$ and the $L^\infty$-equivalence class to which $\ph$ belongs, but not on the precise point values of $\ph$ in points $\th<0$.

\begin{corollary}\label{col:6.2}
The definition \sef{eq:6.5a} amounts to
\begin{equation}\label{eq:6.18}
\gb{S(t)\ph}(\th) = \int_0^1 K_t(\th,d\si)\,\ph(-\si)
\end{equation}
with for $\si > 0$
\begin{align}\label{eq:6.19}
K_t(\th,\si) &= H(\si+t+\th) + H(t+\th)\Bigl\{\int_0^{t+\th} \rho(\tau)\,d\tau + \int_0^\si \Bigl[\ze(t+\th+\tau)\nonumber\\
&\qquad\qquad - \ze(\tau) + \int_0^{t+\th} \rho(\xi)\gb{\ze(t+\th+\tau-\xi) - \ze(\tau)}\,d\xi\Bigr]\,d\tau\Bigr\}
\end{align}
and $K_t(\th,0) = 0$. \(Here $H$ is the standard Heaviside function.\)
\end{corollary}

\begin{proof}
For $t+\th < 0$ the second term in the expression for $K$ does not contribute and the first term yields
\[\gb{S(t)\ph}(\th) = \ph(t+\th)\]
which is in accordance with \sef{eq:6.5a} because of \sef{eq:6.2}. Now assume that $t+\th \ge 0$. Clearly the first term contributes a unit jump at $\si = 0$ and $H(t+\th) = 1$. The second factor has, as a function of $\si$, a jump of magnitude $\int_0^{t+\th} \rho(\tau)\,d\tau$ at $\si = 0$, but is otherwise absolutely continuous with derivative
\[\ze(t+\th+\si) - \ze(\si) + \int_0^{t+\th} \rho(\xi)\gb{\ze(t+\th+\si-\xi) - \ze(\si)}\,d\xi.\]
The jumps yield the first term at the right hand side of \sef{eq:6.12} evaluated at $t+\th$ and the absolutely continuous part yields the second term.
\end{proof}

Note that $K_t$ is a \emph{bounded} in the sense (cf. \cite[Definition 3.2]{Kun09}) that for fixed $\th$ in $[-1,0]$ the function $\si \mapsto K_t(\th,\si)$ is of normalized bounded variation, while for fixed $\si \in [0,1]$ the function $\th \mapsto K_t(\th,\si)$ is bounded and measurable.

\begin{corollary}\label{col:6.3}
The operator $S(t)$ extends to a twin operator.
\end{corollary}

\begin{proof}
This is a general property of kernel operators. Explicitly we have
\begin{equation}\label{eq:6.20}
\gb{\dy S(t)}(\si) = \int_0^1 \dy(d\tau)\,K_t(-\tau,\si).
\end{equation}
\end{proof}

\begin{theorem}\label{thm:6.4} 
The semigroup $\{S(t)\}_{t \ge 0}$ defined by \sef{eq:6.18} is a twin semigroup.
\end{theorem}

\begin{proof}
With reference to Definition \ref{def:2.1} we note that $S(0)=I$ follows directly from \sef{eq:6.18}--\sef{eq:6.19}, while the semigroup property follows from the uniqueness of solutions to \sef{eq:6.1}--\sef{eq:6.2} and the fact that $S(t)$ corresponds to translation along the solution (so essentially it follows from the corresponding property for translation, and uniqueness of extension). 

The exponential estimates ii) are well-established in the theory of RFDE, for instance Sections I.5, IV.2 and IV.3 of \cite{Diek95} or the proof of Theorem \ref{thm:5.1}. 

Property iii), the measurability of $t \mapsto \dy S(t)y$, is a direct consequence of the way $K_t(\th,\si)$ defined in \sef{eq:6.19} depends on $t$. 

It remains to verify that the Laplace transform defines a twin operator. By Fubini's Theorem, the Laplace transform is a kernel operator with kernel
\[\int_0^\infty e^{-\la t} K_t(\th,\si)\,dt.\]
\end{proof}

\begin{theorem}\label{thm:6.5}
The operator $C$ defined by \sef{eq:6.8}--\sef{eq:6.9} is the generator \(in the sense of \sef{eq:2.3}\) of $\{S(t)\}_{t \ge 0}$ defined by \sef{eq:6.18}.
\end{theorem}

\begin{proof}
Assume $\ph \in (\la I - C)\psi$. Then there exists $\psi' \in Y$ which is a.e. derivative of $\psi$ such that
\begin{align*}
\la\psi - \psi' &= \ph,\qquad -1 \le \th < 0\\
\la\psi(0) - \pa{\ze}{\psi} &= \ph(0).
\end{align*}
Solving the differential equation yields that
\begin{equation}\label{eq:6.21}
\psi(\th) = e^{\la\th}\bigl\{\int_\th^0 e^{-\la\si}\ph(\si)\,d\si + \psi(0)\bigr\}
\end{equation}
and accordingly the boundary condition for $\th = 0$ boils down to
\begin{equation}\label{eq:6.22}
\psi(0) = \De(\la)^{-1}\bigl[\ph(0) + \int_0^1 d\ze(\si)e^{-\la\si}\int_{-\si}^0 e^{-\la\tau}\ph(\tau)\,d\tau\bigr]
\end{equation}
which requires that $\det \De(\la) \not= 0$ with
\[\De(\la) = \la I - \int_0^1 d\ze(\si)e^{-\la\si}.\]

Our claim is that the identity
\[\psi(\th) = \int_0^\infty e^{-\la t}\gb{S(t)\ph}(\th)\,dt\]
holds. To verify this, we first note that
\[\int_0^\infty e^{-\la t}\gb{S(t)\ph}(\th)\,dt = e^{\la\th}\bigl\{\int_\th^0 e^{-\la\si}\ph(\si)\,d\si + \bar x(\la;\ph)\bigr\}\]
(where $\bar x(\la;\ph) := \int_0^\infty e^{-\la t}x(t;\ph)\,dt$, with $x(t;\ph)$ the solution of \sef{eq:6.1}--\sef{eq:6.2} given by \sef{eq:6.17}) since
\begin{align*}
\int_0^\infty e^{-\la t}x(t+\th;\ph)\,dt &= \int_0^{-\th} e^{-\la t}\ph(t+\th)\,dt + \int_{-\th}^\infty  e^{-\la t}x(t+\th)\,dt\\
&= e^{\la\th}\bigl\{\int_\th^0 e^{-\la\si}\ph(\si)\,d\si + \bar x(\la;\ph)\bigr\}.
\end{align*}
So, since \sef{eq:6.21} holds, we need to check that $\psi(0) = \bar x(\la;\ph)$. From \sef{eq:6.15} we deduce that
\[\bar x = (1 - \bar \ze)^{-1}\bar f.\]
Therefore, using the first representation of $f$ in \sef{eq:6.16}, it follows that
\begin{align*}
\la \bar f(\la) &= \ph(0) + \int_0^\infty \la e^{-\la t}\int_0^t\gb{\int_s^1 d\ze(\th)\ph(s-\th)}\,ds dt\\
&= \ph(0) + \int_0^\infty e^{-\la t}\int_t^1 d\ze(\th)\ph(t-\th)\, dt\\
&= \ph(0) + \int_0^1 d\ze(\th)\int_0^\th e^{-\la t}\ph(t-\th)\,dt\\
&= \ph(0) + \int_0^1 d\ze(\th)e^{-\la\th}\int_{-\th}^0 e^{-\la\si}\ph(\si)\,d\si
\end{align*}
which equals the vector at the right hand side of \sef{eq:6.22} on which the matrix $\De(\la)^{-1}$ acts. Since
\[\la \bar\ze(\la) = \int_0^1 d\ze(\th)e^{-\la\th},\]
we arrive at the conclusion that indeed $\psi(0) = \bar x(\la;\ph)$.
\end{proof}

It is a direct consequence of \sef{eq:6.9} that
\begin{equation}\label{eq:6.23}
X = \clo{\DOM{C}} = C\gb{[-1,0],\BR^n}.
\end{equation}
Clearly $C\psi \cap X$ is either empty or a singleton, cf. \sef{eq:6.8}, and for the set to be nonempty we need that $\psi \in C^1$ and $\psi'(0) = \pa{\ze}{\psi}$. So the generator $A$ of the restriction $\{T(t)\}_{t \ge 0}$ of $\{S(t)\}_{t \ge 0}$ to $X$ is given by
\begin{align}\label{eq:6.24}
\begin{split}
\DOM{A} &= \bigl\{ \psi \in C^1 \mid \psi'(0) = \pa{\ze}{\psi} \bigr\}\\
A\psi &= \psi'
\end{split}
\end{align}
in complete agreement with the standard theory.

As $S(t)$ maps $Y$ into $X$ for $t \ge 1$, one might wonder whether we gained anything at all by the extension from $X$ to $Y$? Already in the pioneering first version of his book \cite{Hale71}, Jack Hale emphasized that if one adds a forcing term to \sef{eq:6.1}, one needs
\begin{equation}\label{eq:6.25}
q(\th) = \begin{cases}
1 &\th = 0\\
0 &-1 \le \th < 0
\end{cases}
\end{equation}
to describe the solution by way of the variation-of-constants formula. Indeed, the solution of
\begin{align}\label{eq:6.26}
\begin{split}
\dot x(t) &= \pa{\ze}{x_t} + f(t),\qquad t \ge 0\\
x(\th) &= \ph(\th),\qquad -1 \le \th \le 0
\end{split}
\end{align}
is explicitly given by
\begin{equation}\label{eq:6.27}
x_t = S(t)\ph + \int_0^t S(t-\tau)q f(\tau)\,d\tau
\end{equation}
since \sef{eq:6.26} corresponds to the initial value problem
\begin{equation}\label{eq:6.28}
\frac{du}{dt} \in Cu + qf,\qquad u(0) = \ph,
\end{equation}
where $u(t) = x_t$. (Incidentally, please note that the solution with initial condition $q$ is the so-called \emph{fundamental solution}, cf. \cite[Section I.2]{Diek95}.)

The integration theory of Section \ref{sec:5} provides a precise underpinning of the integral in \sef{eq:6.27}. In the original approach of Hale, the hidden argument $\th$ in \sef{eq:6.27} is inserted and thus the integral reduces to the integration of an $\BR^n$-valued function. Note that evaluation in a point corresponds to the application of a Dirac functional, so our approach yields, in a sense, a rather late theoretical underpinning of Hale's approach. The $\odot\ast$-calculus approach of \cite{Diek95} amounts, for RFDE, to the observation just before Corollary \ref{eq:6.2} and its consequences.

More precisely, one embeds $X$ into $\BR^n \times L^\infty\gb{[-1,0],\BR^n}$, interprets $q$ as $(1,0)$, considers $\BR^n \times L^\infty\gb{[-1,0],\BR^n}$ as the dual space of $\BR^n \times L^1\gb{[0,1],\BR^n}$, interprets the integral as a weak$^\ast$-integral and checks that the integral belongs to the range of the embedding, so defines an element of $X$. As long as one restricts attention to RFDE, the current approach has its more sophisticated integration theory as a drawback and no clear advantage to compensate. However, this changes when one extends the theory, as we shall do in Section \ref{sec:11}, to neutral equations. Neutral equations correspond to an unbounded change in the rule for extension, even within a functional analytic framework where $q$ is well-defined. In the $\odot\ast$-setting this manifests itself in dependence of $X^\odot$, and hence $X^{\odot\ast}$, on the particular perturbation thus obstructing a satisfactory sun-star perturbation theory for neutral equations. In contrast, the present approach allows us to keep working with the norming dual pair $Y$ and $\DY$ and to develop a variation-of-constants formula.

At the end of Section \ref{sec:6} we shall briefly indicate how, alternatively, one can use a perturbation approach to derive the results presented above.

As a final remark, we emphasize that the variation-of-constants formula \sef{eq:6.27} is the key first step towards a local stability and bifurcation theory for nonlinear problems, as shown in detail in \cite{Diek95}.

\vspace{1.5cm}
\centerline{\Large\bf Part II: Bounded perturbations}
\smallskip
\centerline{\Large\bf describing retarded equations}
\bigskip

\section{The variation-of-constants formula for forcing functions with finite dimensional range}\label{sec:5}
\setcounter{equation}{0}

When the ultimate aim is to study nonlinear problems, one usually focuses on real-valued functions and functionals. Spectral theory, on the other hand, benefits from complexification. The formulation below considers real functionals acting on a real vector space, but when $Y$, $\DY$ is a norming dual pair, the same holds for their complexifications.\footnote{Complexification of a norming dual pair entails some subtle difficulties regarding the choice of norms. These subtleties are explained in \cite[Section III.7]{Diek95}. But when we deal with function spaces, complexification can be represented by allowing the functions to take values in $\BC$ or $\BC^n$ and subsequently the norm can be defined by copying the definition for the real functions, while replacing the real absolute value by the complex modulus. In the present paper the two relevant norms are the supremum norm and the total variation norm, see Appendix B.}

\medskip\noindent

Motivated by RFDE, in particular \sef{eq:6.27}, we want to define an element $u(t)$ of $Y$ by way of the action on $\DY$ expressed in the formula
\begin{equation}\label{eq:4.2}
\pa{\dy}{u(t)} = \dy S(t)u_0 + \int_0^t \dy S(t-\tau)q\,f(\tau)d\tau,
\end{equation}
where
\begin{itemize}
\item[(i)] $(Y,\DY)$ is a norming dual pair;
\item[(ii)] $q \in Y$;
\item[(iii)] $f : [0,T] \to \BR$ is bounded and measurable;
\item[(iv)] $\bigl\{S(t)\bigr\}$ is a twin semigroup,
\end{itemize}
and where $u_0$ (corresponding to $\ph$ in \sef{eq:6.27}) is an arbitrary element of $Y$. The first term at the right hand side of \sef{eq:4.2} is no problem at all, it contributes $S(t)u_0$ to $u(t)$. The second term defines an element of $Y^{\diamond\ast}$, but it is not clear that this element is, without additional assumptions, represented by an element of $Y$.

\begin{lemma}\label{lem:4.1}
In addition to (i)-(iv) assume that
\begin{equation}\label{eq:4.1}
\gb{Y,\si(Y,\DY)}\quad\hbox{is sequentially complete.}
\end{equation}
Then
\begin{equation}\label{eq:4.3}
\dy \mapsto \int_0^t \dy S(t-\tau)q\,f(\tau)d\tau
\end{equation}
is represented by an element of $Y$, to be denoted as
\begin{equation}\label{eq:4.4}
\int_0^t S(t-\tau)q\,f(\tau)\,d\tau
\end{equation}
\end{lemma}

\begin{proof}
There exists a sequence of step functions $f_m$ such that $|f_m| \le |f|$ and $f_m \to f$ pointwise. Lemma \ref{lem:2.3} shows that
\begin{equation*}
\int_0^t S(t-\tau)q\,f_m(\tau)d\tau
\end{equation*}
belongs to $Y$ (in fact even to $\DOM{C}$). Since (see Definition \ref{def:2.1}(ii))
\begin{equation*}
\babs{\dy S(t-\tau)q f_m(\tau)} \le M e^{\om(t-\tau)}\|q\|\,\|\dy\|\,\sup_{\si} |f(\si)|,
\end{equation*}
the dominated convergence theorem implies that for every $\dy \in \DY$
\begin{equation*}
\lim_{m \to \infty} \int_0^t \dy S(t-\tau)q\,f_m(\tau)d\tau = \int_0^t \dy S(t-\tau)q\,f(\tau)d\tau.
\end{equation*}
The sequential completeness next guarantees that the limit too is represented by an element of $Y$.
\end{proof}

In Section \ref{sec:8} we shall, as a step towards treating neutral equations, replace $f(\tau)\,d\tau$ by $F(d\tau)$ with $F$ of bounded variation. Then approximation by step functions no longer works. This observation motivates to look for an alternative sufficient condition.

\begin{lemma}\label{lem:4.2}
In addition to (i)-(iv) assume that
\begin{align}\label{eq:4.5}
&\hbox{a linear map  } \gb{\DY, \si(\DY,Y)} \to \BR\hbox{ is continuous}\nonumber\\
&\hbox{if it is sequentially continuous}.
\end{align}
Then the assertion of Lemma \ref{lem:4.1} holds.
\end{lemma}

\begin{proof}
Again we are going to make use of the dominated convergence theorem. Consider a sequence $\{\dy_m\}$ in $\DY$ such that for every $y \in Y$ the sequence $\pa{\dy_m}{y}$ converges to zero in $\BR$. Then for all relevant $t$ and $\tau$ we have
\[\lim_{m \to \infty} \dy_m S(t-\tau)q = 0\]
and consequently
\[\lim_{m \to \infty} \int_0^t \dy_m S(t-\tau)q \, f(\tau)\,d\tau = 0.\]
So the linear map \sef{eq:4.3} is, in the sense described in \sef{eq:4.5}, sequentially continuous and therefore, by the assumption, continuous. Since
\[\gb{\DY,\si(\DY,Y)}' = Y,\]
we conclude that \sef{eq:4.3} is represented by an element of $Y$.
\end{proof}

In the next section we are going to use these results to show that a certain type of perturbation of a twin semigroup yields again a {\sl twin} semigroup and then we will also need that with (ii) replaced by
\begin{itemize}
\item[$\hbox{(ii)}^\prime$] $\dq \in \DY$,
\end{itemize}
we have that
\begin{equation}\label{eq:4.6}
y \mapsto \int_0^t \dq S(t-\tau)y\, f(\tau)\,d\tau
\end{equation}
is represented by an element of $\DY$, to be denoted as
\begin{equation}\label{eq:4.7}
\int_0^t \dq S(t-\tau)\, f(\tau)\,d\tau.
\end{equation}
Applying the two lemmas above, with the role of $Y$ and $\DY$ interchanged, we find that this is indeed the case if either
\begin{equation}\label{eq:4.8}
\gb{\DY,\si(\DY,Y)}\quad\hbox{is sequentially complete.}
\end{equation}
or
\begin{align}\label{eq:4.9}
&\hbox{a linear map  } \gb{Y, \si(Y,\DY)} \to \BR\hbox{ is continuous}\nonumber\\
&\hbox{if it is sequentially continuous}.
\end{align}
In our treatment of delay differential equations we shall assume \sef{eq:4.1} and \sef{eq:4.9}, but in our treatment of renewal equations we shall assume \sef{eq:4.8} and \sef{eq:4.5}. 

This difference is a consequence of what we stated at the start of Section \ref{sec:2}: we want that $Y$ is the state space and $\DY$ is an auxiliary space. For delay differential equations we take $Y = B([-1,0])$ and $\DY = NBV([0,1])$, while for renewal equations we take $Y = NBV([0,1])$ and $\DY = B([-1,0])$. So in terms of the two function spaces involved, the assumptions are identical (and these assumptions are substantiated in Appendix B), but because their roles are interchanged the formulations are a mirror image of each other.

As in the next section we shall use both properties, we state

\begin{definition}\label{def:4.3}
We say that a norming dual pair $(Y,\DY)$ is suitable for twin perturbation if 
\begin{itemize}
\item[(a)] at least one of \sef{eq:4.1} and \sef{eq:4.5} holds; and
\item[(b)] at least one of \sef{eq:4.8} and \sef{eq:4.9} holds
\end{itemize}
\end{definition}

\section{Finite dimensional range perturbation of twin semigroups}\label{sec:6}
\setcounter{equation}{0}

In this section we consider the following situation:
\begin{itemize}
\item[--] $(Y,\DY)$ is a norming dual pair that is suitable for twin perturbation, cf. Definition \ref{def:4.3};
\item[--] $\{S_0(t)\}$ is a twin semigroup on $(Y,\DY)$ with generator $C_0$;
\item[--] For $j=1,\ldots,n$ the elements $q_j \in Y$ and $\dq_j \in \DY$ are given. 
\end{itemize}

Our aim is to define constructively a twin semigroup $\{S(t)\}$ with generator $C$ defined by
\begin{equation}\label{eq:5.1}
\DOM{C} = \DOM{C_0}.\qquad Cy = C_0y + \sum_{j=1}^n \pa{\dq_j}{y}q_j.
\end{equation}
The first step is to introduce a $n \times n$-matrix valued function $k$ on $[0,\infty)$ via
\begin{equation}\label{eq:5.2}
k_{ij}(t) = \dq_i S_0(t) q_j.
\end{equation}
Note that, by assumption, $t \mapsto k(t)$ is locally bounded and measurable. With the kernel $k$ we associate its resolvent $r$. This is by definition the unique solution of the matrix renewal equation
\begin{equation}\label{eq:5.3}
k + k \ast r = r = k + r \ast k
\end{equation}
or, equivalently,
\begin{equation}\label{eq:5.4}
r = \sum_{j=1}^\infty k^{j\ast},
\end{equation}
where $k^{1\ast} := k$ and $k^{m\ast} := k \ast k^{(m-1)\ast}$ for $m \ge 2$. Here $\ast$ denotes the usual convolution product of functions.

In variation-of-constants spirit,  \sef{eq:5.1} motivates us to presuppose that $S(t)$ and $S_0(t)$ should be related to each other by the equation
\begin{equation}\label{eq:5.5}
S(t) = S_0(t) + \int_0^t S_0(t-\tau)BS(\tau)\,d\tau,
\end{equation}
where
\begin{equation}\label{eq:5.6}
By := \sum_{j=1}^n \pa{\dq_j}{y}q_j.
\end{equation}
By letting $B$ act on \sef{eq:5.5} we obtain, for given initial point $y \in Y$, a finite dimensional renewal equation. To formulate this equation, we first write \sef{eq:5.6} as
\begin{equation}\label{eq:5.7}
By = \pa{\dq}{y}\cdot q
\end{equation}
where $\dq$ is the $n$-vector with $\DY$-valued components $\dq_j$ and similarly $q$ is the $n$-vector with $Y$-valued components $q_j$ and where $\cdot$ denotes the inner product in $\BR^n$. We can factor (a rank factorization) $B$ as $B= B_2B_1$ with $B_1 : Y \to \BR^n$ and $B_2 : \BR^n \to Y$ defined by
\begin{equation}\label{eq:5.7a}
B_1y = \pa{\dq}{y},\qquad B_2x = \sum_{j=1}^n x_jq_j
\end{equation}

Now let \sef{eq:5.5} act on $y \in Y$ and next act on the resulting identity with the vector $\dq$. This yields the equation
\begin{equation}\label{eq:5.8}
v(t)y = \dq S_0(t)y + \int_0^t k(t-\tau)v(\tau)y\,d\tau,
\end{equation}
where $v(t)y$ corresponds to $\dq S(t)y = B_1S(t)y$. The solution of \sef{eq:5.8} can be expressed in terms of the resolvent $r$ of the kernel $k$ and the forcing function $t \mapsto \dq S_0(t)y$ by the formula
\begin{equation}\label{eq:5.9}
v(t)y = \dq S_0(t)y + \int_0^t r(t-\tau)\dq S_0(\tau)y\,d\tau.
\end{equation}
And now that $v(\novar)y$, representing $\dq S(\novar)y$, can be considered as known, \sef{eq:5.5} becomes an explicit formula
\begin{equation}\label{eq:5.10}
S(t) = S_0(t) + \int _0^t S_0(t-\tau)q\cdot v(\tau)\,d\tau.
\end{equation}

Please note that, with this definition of $S(t)$, we do indeed have that
\[v(t)y = \dq S(t)y\]
(compare \sef{eq:5.10} to \sef{eq:5.8}).

Formula \sef{eq:5.10} is well suited for proving, on the basis of Lemma \ref{lem:4.1} or Lemma \ref{lem:4.2}, that $S(t)$ maps $Y$ into $Y$. But not for proving that $S(t)$ maps $\DY$ into $\DY$. So even though this may seem superfluous, we now provide an alternative dual constructive definition starting from the following equation
\begin{equation}\label{eq:5.12}
S(t) = S_0(t) + \int_0^t S(t-\tau) BS_0(\tau)\,d\tau
\end{equation}
which is the variant of \sef{eq:5.5} in which the roles of $S(t)$ and $S_0(t)$ are interchanged. Let \sef{eq:5.12} act (from the right) on $\dy \in \DY$ and next let the resulting identity act on the vector $q$. This yields the equation
\begin{equation}\label{eq:5.13}
\dy w(t) = \dy S_0(t)q + \int_0^t \dy w(t-\tau) k(\tau)\,d\tau,
\end{equation}
where $\dy w(t)$ corresponds to $\dy S(t)q$. The formula
\begin{equation}\label{eq:5.14}
\dy w(t) = \dy S_0(t)q + \int_0^t \dy S_0(t-\tau)q\,r(\tau)\,d\tau
\end{equation}
expresses the solution of \sef{eq:5.13} in terms of the forcing function $\dy S_0(t)q$ and the resolvent $r$ of the kernel $k$. Next we rewrite \sef{eq:5.12} in the form
\begin{equation}\label{eq:5.15}
S(t) = S_0(t) + \int_0^t w(t-\tau) \cdot \dq S_0(\tau)\,d\tau.
\end{equation}
Please note that indeed $\dy w(t) = \dy S(t)q$ (compare \sef{eq:5.15} to \sef{eq:5.13}).

Of course we should now verify that the integrals in \sef{eq:5.10} and \sef{eq:5.15} do indeed define the same object. Writing the integral in \sef{eq:5.10} as $w_0 \ast v$ and the integral in \sef{eq:5.15} as $w \ast v_0$, equality follows from \sef{eq:5.9} written in the form
\[v = v_0 + r \ast v_0\]
and \sef{eq:5.14} written in the form 
\[w = w_0 + w_0 \ast r\]
since
\begin{align*}
w_0 \ast v &= w_0 \ast (v_0 + r \ast v_0) = w_0 \ast v_0 + w_0 \ast r \ast v_0\\
&= (w_0 + w_0 \ast r) \ast v_0 = w \ast v_0.
\end{align*}

\begin{theorem}\label{thm:5.1} 
The combination \sef{eq:5.9}--\sef{eq:5.10} or, equivalently, the combination of \sef{eq:5.14}--\sef{eq:5.15}, defines a twin semigroup $\bigl\{S(t)\bigr\}$ with generator $C$ defined in \sef{eq:5.1}.
\end{theorem}

\begin{proof}
Since $(Y,\DY)$ is suitable for twin perturbation, we can use \sef{eq:5.10} and either Lemma \ref{lem:4.1} or Lemma \ref{lem:4.2} to deduce that $S(t)$ maps $Y$ into $Y$. Similarly we can use \sef{eq:5.15} and the observation concerning \sef{eq:4.6} to deduce that $S(t)$ maps $\DY$ into $\DY$. So $\bigl\{S(t)\bigr\}$ is a twin operator.

\medskip\noindent
With a view to deriving the semigroup property
\begin{equation}\label{eq:5.16}
S(t+s) = S(t)S(s),\qquad t,s \ge 0,
\end{equation}
we first formulate the auxiliary result

\begin{lemma}\label{lem:5.2}
The solution $v(\novar)y$ of \sef{eq:5.8} has the property
\begin{equation}\label{eq:5.17}
v(t+s)y = v(t)S(s)y
\end{equation}
\end{lemma}

\begin{proof}
From \sef{eq:5.8} it follows that
\begin{align*}
v(t+s)y &= \dq S_0(t)S_0(s)y + \int_0^s k(t+s-\tau)v(\tau)y\,d\tau\\
&\qquad + \int_0^t k(t-\si)v(s+\si)y\,d\si
\end{align*}
and by uniqueness \sef{eq:5.17} follows provided
\[\dq S_0(t)S_0(s)y + \int_0^s k(t+s-\tau)v(\tau)y\,d\tau = \dq S_0(t)S(s)y.\]
Noting that
\[k(t+s-\tau) = \dq S_0(t+s-\tau)q = \dq S_0(t)S_0(s-\tau)q,\]
we conclude from \sef{eq:5.10} that this identity does indeed hold.
\end{proof}

To verify \sef{eq:5.16}, we start from \sef{eq:5.10} and write
\begin{align*}
S(t+s)y &= S_0(t)S_0(s)y + \int_0^s S_0(t+s-\tau)q \cdot v(\tau)y\,d\tau\\
&\qquad + \int_0^t S_0(t-\si)q \cdot v(\si+s)y\,d\si\\
&= S_0(t)S(s)y + \int_0^t S_0(t-\si)q \cdot v(\si)S(s)y\,d\si\\
&= S(t)S(s)y.
\end{align*}

Both the property $S(0)=I$ and the measurability, for all $y \in Y$, $\dy \in \DY$, of $t \mapsto \dy S(t)y$ follow from \sef{eq:5.10} and the corresponding properties of $\{S_0(t)\}$. 

The exponential estimate for $\dy S_0(t) y$ yields exponential estimates for both the kernel $k$ and the forcing function $\dq S_0(\novar)y$ in the renewal equation \sef{eq:5.8}. Therefore, see Theorem \ref{thm:resolvent} with $\mu(dt)=k(t)dt$, we obtain an exponential estimate for the resolvent $\rho(dt)=r(t)dt$, and hence via \sef{eq:5.9} an exponential bound for $v(t)y$. Finally, using \sef{eq:5.10} we obtain an exponential bound for $\dy S(t)y$.

It remains to compute the Laplace transform, cf. \sef{eq:2.2}. Since 
\begin{equation}\label{eq:5.18}
\int_0^\infty e^{-\la t} \dy S_0(t)y\,dt = \dy (\la I - C_0)^{-1}y,
\end{equation}
we obtain by Laplace transformation of \sef{eq:5.10} the identity
\[\int_0^\infty e^{-\la t} \dy S(t)y\,dt = \dy (\la I - C_0)^{-1}y + \dy (\la I - C_0)^{-1}q \novar \bar{v}(\la)y.\]
Laplace transformation of either \sef{eq:5.8} or \sef{eq:5.9} and \sef{eq:5.3} yields
\begin{align}\label{eq:5.18a}
\bar{v}(\la)y &= \bigl[ I - \dq(\la I - C_0)^{-1}q\bigr]^{-1}\,\dq(\la I - C_0)^{-1}y\nonumber\\
&= \bigl[ I - \bar k(\la) \bigr]^{-1}\,\dq(\la I - C_0)^{-1}y
\end{align}
By combining the last two identities we arrive at
\begin{align}\label{eq:5.19}
\int_0^\infty e^{-\la t}\dy S(t)y\,dt
&= \dy (\la I - C_0)^{-1}y + \dy (\la I - C_0)^{-1}q\, \cdot\nonumber \\
&\qquad \cdot\, \bigl[ I - \dq(\la I - C_0)^{-1}q\bigr]^{-1}\,\dq(\la I - C_0)^{-1}y.
\end{align}
It remains to check that the right hand side of \sef{eq:5.19} is exactly $\dy (\la I - C)^{-1}y$ when $C$ is defined by \sef{eq:5.1}. So consider the equation
\[(\la I - C)\eta = y.\]
By \sef{eq:5.1} this is equivalent to
\[(\la I - C_0)\eta = y + \pa{\dq}{\eta}\cdot q\]
and hence to
\[\eta = (\la I - C_0)^{-1}y + \sum_{j=1}^n \pa{\dq_j}{\eta}(\la I - C_0)^{-1}q_j.\]
In particular,
\[\pa{\dq_k}{\eta} = \dq_k(\la I - C_0)^{-1}y + \sum_{j=1}^n \dq_k (\la I - C_0)^{-1}q_j\,\pa{\dq_j}{\eta}\]
or, in vector form,
\[\pa{\dq}{\eta} = \dq (\la I - C_0)^{-1}y + \dq (\la I - C_0)^{-1}q \pa{\dq}{\eta}.\]
Hence
\begin{align}\label{eq:5.19a}
\eta = (\la I - C)^{-1}y &= (\la I - C_0)^{-1}y + (\la I - C_0)^{-1}q \nonumber\\
&\hspace{1.75cm} \times\gb{ I - \dq (\la I - C_0)^{-1}q}^{-1}\,\dq (\la I - C_0)^{-1}y
\end{align}
and comparison with \sef{eq:5.19} shows that indeed
\begin{equation}\label{eq:5.20}
\int_0^\infty e^{-\la t} \dy S(t)y\,dt = \dy (\la I - C)^{-1}y.
\end{equation}
This completes the proof of Theorem \ref{thm:5.1}.
\end{proof}

\smallskip

By combining Theorems \ref{thm:5.1} and \ref{thm:3.1} we see that perturbations of the form \sef{eq:5.1} do not alter the subspaces of strong continuity.

\begin{corollary}\label{col:5.3}
The subspaces $X$ and $X^\odot$ of strong continuity are the same for $\{S_0(t)\}$ and $\{S(t)\}$.
\end{corollary}

The special representation of the perturbed semigroup $S(t)$ given in respectively \sef{eq:5.9}--\sef{eq:5.10} and \sef{eq:5.14}--\sef{eq:5.15} allows us to use Theorem \ref{thm:Tauber} to derive a result about the asymptotic behaviour of $S(t)$ without using a spectral mapping theorem, eventual compactness or eventual norm continuity of the semigroup $S(t)$. 

\begin{theorem}\label{thm:5.4}
Under the assumptions of this section let $k$, given by \sef{eq:5.2}, be integrable. Suppose that $S_0(t)$ is bounded and that $(\la I - C_0)^{-1}$ is bounded for $\Re \la \ge 0$. If
\begin{equation}\label{eq:5.22}
\det\gb{I - \bar k(\la)}\quad\hbox{has no zeros for } \Re \la \ge 0,
\end{equation}
then
\begin{equation}\label{eq:5.23}
\|S(t)C^{-1}\| \to 0\qquad\hbox{as}\quad t \to \infty.
\end{equation}
As a consequence we have that $S(t)y \to 0$ as $t \to \infty$ for every $y$ in the norm-closure of $\DOM{C}$.
\end{theorem}

\begin{proof}
We first show that $S(t)$ is bounded. From the half-line Gel'fand theorem, see Theorem \ref{thm:Gelfand}, applied to the absolutely continuous measure $\mu(dt) = k\,dt$, it follows that the resolvent $\rho(dt) = r\,dt$ is an absolutely continuous bounded measure. Fix $y$ in $Y$ and $\dy \in \DY$. From Theorem \ref{thm:convo-meas-Borel-fun} and \sef{eq:5.9} it follows that $v\,dt$ is a bounded measure and another application of Theorem \ref{thm:convo-meas-Borel-fun} shows that $\dy S(t) y$ as defined via \sef{eq:5.10}, is a bounded Borel function and hence $S(t)$ is a bounded twin semigroup.

If $(\la I - C_0)^{-1}$ is bounded for $\Re \la \ge 0$ and \sef{eq:5.22} holds, then it follows from \sef{eq:5.18a} and \sef{eq:5.19a} that $(\la I - C)^{-1}$ is bounded for $\Re \la \ge 0$.

This completes the proof that $S(t)$ is bounded and that $\si(C) \mcap i\BR = \emptyset$. So an application of Theorem \ref{thm:Tauber} yields the proof.
\end{proof}

The following variant of Theorem \ref{thm:5.4} is motivated by RFDE and various boundary value problems. See \cite{KVL20} for more information.

\begin{theorem}\label{thm:5.5}
Under the assumptions of this section let $k$, given by \sef{eq:5.2}, be of bounded variation with $k(0) = 0$. Suppose that $S_0(t)$ is bounded and that $(\la I - C_0)^{-1}$ has a simple pole at $\la = 0$ but is otherwise bounded for $\Re \la \ge 0$. Assume that
\begin{equation}\label{eq:5.24}
\pa{\dq}{P_0q}P_0y = \pa{\dq}{P_0y}P_0q,
\end{equation}
where $P_0 : Y \to Y$ denotes the spectral projection onto the eigenspace of $C_0$ at $\la = 0$. If
\begin{equation}\label{eq:5.25}
\det\gb{\la I - \hat k(\la)}\quad\hbox{has no zeros for } \Re \la \ge 0,
\end{equation}
then
\begin{equation}\label{eq:5.26}
\|S(t)C^{-1}\| \to 0\qquad\hbox{as}\quad t \to \infty.
\end{equation}
As a consequence we have that $S(t)y \to 0$ as $t \to \infty$ for every $y$ in the norm-closure of $\DOM{C}$.
\end{theorem}

\begin{proof}
We first show that $S(t)$ is bounded. Let $\mu(dt) = k\,dt$ and observe that
\[\what\mu(\la) = \bar k(\la) = \frac{1}{\la}\hat k(\la).\]
So it follows from \sef{eq:5.22} and the half-line Gel'fand theorem, see Theorem \ref{thm:Gelfand}, applied to $\mu$, that the resolvent $\rho$ is an absolutely continuous bounded measure $\rho = r\,dt$. Fix $y$ in $Y$ and $\dy \in \DY$. From Theorem \ref{thm:convo-meas-Borel-fun} and \sef{eq:5.9} it follows that $v\,dt$ is a bounded measure and another application of Theorem \ref{thm:convo-meas-Borel-fun} shows that $\dy S(t) y$ as defined via \sef{eq:5.10}, is a bounded Borel function and hence $S(t)$ is a bounded twin semigroup.

To show that $(\la I - C)^{-1}$ is bounded for $\Re \la \ge 0$ first observe that
\begin{align}\label{eq:5.27}
(\la I - C)^{-1}y &= \la \bigl[\la I - \hat k(\la)\bigr]^{-1}\gb{(\la I - C_0)^{-1}y - \dq (\la I - C_0)^{-1}q\,(\la I - C_0)^{-1}y\nonumber\\
&\qquad\qquad\qquad + \dq (\la I - C_0)^{-1}y\, (\la I - C_0)^{-1}q}.
\end{align}
Using the assumption on $(\la I - C_0)^{-1}$ we can write
\[(\la I - C_0)^{-1}y = \frac{1}{\la}P_0y + H(\la)y,\]
where $P_0 : Y \to Y$ denotes the spectral projection onto the eigenspace of $C_0$ at $\la = 0$ and $H(\la) : Y \to Y$ is a bounded linear operator for $\Re \la \ge 0$. Using this we can expand
\begin{align*}
&-\dq (\la I - C_0)^{-1}q\,(\la I - C_0)^{-1}y + \dq (\la I - C_0)^{-1}y\, (\la I - C_0)^{-1}q\\
&\qquad=\\
&\ \frac{1}{\la^2}\gb{\dq P_0y\, P_0q-\dq P_0q\,P_0y }\\
&\quad +\frac{1}{\la}\gb{\dq\, P_0y H(\la)q + \dq H(\la)y P_0q -\dq P_0q\,H(y)y - \dq H(\la)q\, P_0y}\\
&\quad + \dq H(\la) y \,H(\la)q - \dq H(\la)q\,H(\la)y.
\end{align*}
Condition \sef{eq:5.24} shows that the term $\la^{-2}$ vanishes and since $(\la I - C_0)^{-1}y$ has a simple pole at $\la = 0$ and $H(\la)y$ is bounded for $\Re \la \ge 0$, it follows from \sef{eq:5.27} that $(\la I - C)^{-1}$ is bounded for $\Re \la \ge 0$.

This completes the proof that $S(t)$ is bounded and $\si(C) \mcap i\BR = \emptyset$. So an application of Theorem \ref{thm:Tauber} yields the proof.
\end{proof}

\smallskip
Note that condition \sef{eq:5.24} is automatically satisfied if the null space of $C_0$ is one-dimensional. In general, the condition that $\la = 0$ is a simple pole of $(\la I - C_0)^{-1}$ implies that the generalized null space of $C_0$ equals the null space of $C_0$, but does not give any information about the dimension of the null space of $C_0$. In the example of RFDE, the null space of $C_0$ is, as we show soon, $n$-dimensional.

\medskip

In Section \ref{sec:4} there was no need to use the perturbation approach developed in this section. Yet, alternatively, we can first concentrate on the special case $\ze = 0$, calling the corresponding twin semigroup $\{S_0(t)\}$ and its generator $C_0$. Next we define for $i = 1,\ldots,n$  elements $q_i \in Y$ and $\dq_i \in \DY$ by
\begin{equation}\label{eq:6.29}
q_i(\th) =
\begin{cases}
0   &\th < 0\\
e_i &\th = 0,
\end{cases}
\end{equation}
where $e_i$ is the $i$-th unit vector in $\BR^n$ and
\begin{equation}\label{eq:6.30}
\dq_i(\th) = \ze_i(\th),
\end{equation}
where $\ze_i$ is the $i$-th row of the matrix valued function $\ze$. For the matrix $k$ introduced in \sef{eq:5.2} we find
\begin{equation}\label{eq:6.31}
k_{ij}(t) = \dq_i S_0(t) q_j = \int_0^1 \ze_i(d\tau) \chi_{t-\tau \ge 0} e_j = \ze_{ij}(t).
\end{equation}
With the convention that $\ze(\tau) = \ze(1)$ for $\tau \ge 1$, we can also write (with $y$ corresponding to $\ph$)
\begin{align}\label{eq:6.32}
\dq S_0(t)y &= \int_0^t \ze(d\tau)y(0) + \int_t^1 \ze(d\tau) y(t-\tau)\nonumber\\
                &= \ze(t)y(0) + \int_t^1 \ze(d\tau) y(t-\tau).
\end{align}
Comparing the right hand side to the right hand side of (2.5) on page 16 of \cite{Diek95}, we see that the RE \sef{eq:5.8} of the present paper is identical to equation (2.4a) on page 16 of \cite{Diek95}:
\[\dot x(t) = \int_0^t \ze(\th) \dot x(t-\th)\,d\th + \ze(t)y(0) + \int_t^1 \ze(d\tau) y(t-\tau).\]
We conclude that $v(\cdot)y$ in \sef{eq:5.8} corresponds to $\dot x$ in this equation. 

Next apply to \sef{eq:5.10} the element of $\DY$ that corresponds to the Dirac measure in $-\th \in [0,1]$. This yields
\begin{equation}\label{eq:6.33}
\gb{S(t)y}(\th) = y(t+\th) + \int_0^t r(t-\tau+\th)\cdot v(\tau)y\,d\tau,
\end{equation}
where we adopted the convention that both $y$ and $q$ are extended by their value in zero. It follows that
\begin{equation}\label{eq:6.34}
\gb{S(t)y}(\th) = 
\begin{cases}
y(t+\th), &t+\th \le 0\\
y(0) + \int_0^{t+\th} v(\tau)y\,d\tau, &t+\th \ge 0.
\end{cases}
\end{equation}
Since
\[y(0) + \int_0^{t+\th} v(\tau)y\,d\tau = y(0) + \int_0^{t+\th} \dot x(\tau;y)\,d\tau = x(t+\th;y),\]
this corresponds exactly to \sef{eq:6.5a}. We conclude that the direct approach and the perturbation approach are fully consistent.

We conclude by showing that the assumptions of Theorem \ref{thm:5.5} are satisfied for RFDE. From \sef{eq:6.31} it follows that $k$ is of bounded variation. Furthermore,
\[\gb{(\la I - C_0)^{-1}y}(\th) = \frac{e^{\la\th}}{\la}y(0) + \int_{\th}^0 e^{\la(\th-\si)} y(\si)\,d\si\]
and $P_0 : Y \to Y$ is given by
\[P_0 y = y(0) \mathbbm{1},\]
where $\mathbbm{1} \in Y$ denotes the function that is identically one.

Using \sef{eq:6.29} and \sef{eq:6.30}, observe that 
\begin{equation*}
\pa{\dq}{P_0q}P_0y = \pa{\ze}{I\,\mathbbm{1}}\,y(0)\,\mathbbm{1} = \ze(1)y(0)\,\mathbbm{1}
\end{equation*}
and
\begin{equation*}
\pa{\dq}{P_0y}P_0q = \pa{\ze}{y(0)\,\mathbbm{1}}\,\mathbbm{1} = \ze(1)y(0)\,\mathbbm{1}.
\end{equation*}
This shows that condition \sef{eq:5.24} is satisfied for RFDE. Therefore an application of Theorem \ref{thm:5.5} yields that if
\[\det \gb{\la I - \int_0^1 e^{-z\si}\,d\ze(\si)}\qquad\hbox{ has no zeros for } \Re \la \ge 0,\]
then for $y \in C\gb{[-1,0];\BR^n}$  we have $S(t)y \to 0$ as $t \to \infty$. Since for RFDE $S(1)y \in \DOM{C}$ for every $y \in B\gb{[-1,0];\BR^n}$, we conclude that $S(t)y \to 0$ as $t \to \infty$ for every $y \in B\gb{[-1,0];\BR^n}$.

\section{RE - Renewal equations with ``smooth'' kernels}\label{sec:7}
\setcounter{equation}{0}
The RE
\begin{equation}\label{eq:7.1}
b(t) = \int_0^1 k(a)b(t-a)\,da
\end{equation}
arises in the context of age-structured population dynamics. In that context, $b(t)$ is the rate at which newborn individuals are added to the population at time $t$ and 
\[k(a) = \F(a)\be(a),\] 
with $\F(a)$ the probability to survive to at least age $a$ and $\be(a)$ the age-specific fecundity (it is helpful to think in terms of mothers and daughters, with the male subpopulation implicitly included via a fixed sex ratio). Note that we have scaled the time variable such that the maximum age at which reproduction is possible equals one. It is convenient to define $k(a) = 0$ for $a > 1$.

To facilitate statements and arguments based on the interpretation, we focus in this section our attention on a scalar equation. Generalization to $n$-vector valued functions $b$ and $n \times n$-matrix valued kernels $k$ is straightforward.

We consider \sef{eq:7.1} as a rule for extending the function $b$ and, to get started, supplement it by prescribing the history of $b$ at a particular time, say $t=0$:
\begin{equation}\label{eq:7.2}
b(\th) = \ph(\th),\qquad -1 \le \th \le 0.
\end{equation}
(Note that this leads to equation \sef{eq:7.7} below with $f$ given by \sef{eq:7.8}; even though we have not yet specified assumptions concerning the kernel $k$ and the initial history $\ph$, we like to mention already now that existence, uniqueness and regularity of a solution of this kind of linear Volterra integral equations is covered extensively in \cite{GLS90}; also see Theorem \ref{thm:renewal}).

By translation along the extended function, i.e., by putting
\begin{equation}\label{eq:7.3}
T(t)\ph = b_t
\end{equation}
which is a shorthand for
\begin{equation}\label{eq:7.4}
\gb{T(t)\ph}(\th) = b(t+\th;\ph)
\end{equation}
with $b(\novar;\ph)$ the unique solution of \sef{eq:7.1}--\sef{eq:7.2}, we define a dynamical system.

But what do we choose for the state space $X$ on which the dynamical system acts? Since $b$ is a rate, we get numbers by integrating with respect to time. So the interpretation suggests to take
\begin{equation}\label{eq:7.5}
X = L^1\gb{[-1,0];\BR}
\end{equation}
as is indeed done in \cite{DGG07}. The bonus is that the semigroup $\{T(t)\}$ defined by \sef{eq:7.4} is strongly continuous. But when we compute the infinitesimal generator $A$, we find (with $AC$  standing for ``absolutely continuous'')
\begin{equation}\label{eq:7.6}
\DOM{A} = \bigl\{\ph \in AC \mid \ph(0) = \int_0^1 k(a)\ph(-a)\,da\,\bigr\},\qquad A\ph = \ph'
\end{equation}
showing that all information about the rule for extension is in the domain of $A$ and that the action of $A$ only reflects the translation. The trouble with this is that even small changes in the rule for extension correspond, at the generator level, to unbounded perturbations.

In \cite{DGG07} it is shown how perturbation theory of dual semigroups, also known as sun-star calculus, can be used to overcome this difficulty. Here we show that the formalism of twin semigroups on a norming dual pair of spaces provides an alternative approach. In Section \ref{sec:12} we shall show that this new approach allows us to cover ``neutral'' RE as well, where the adjective neutral expresses that we replace $k(a)da$ by a measure.

Soon we will assume that $k$ is a given bounded measurable function (defined on $[0,\infty)$ but with support in $[0,1]$), but for the time being, while discussing the representation of the solution of \sef{eq:7.1}--\sef{eq:7.2}, it suffices that $k$ is in $L^1$. Combining \sef{eq:7.1} and \sef{eq:7.2} we obtain 
\begin{equation}\label{eq:7.7}
b = k \ast b + f
\end{equation}
with
\begin{equation}\label{eq:7.8}
f(t) = \int_t^1 k(a)\ph(t-a)\,da = \int_{t-1}^0 k(t-\th)\ph(\th)\,d\th\qquad\hbox{for } t <1
\end{equation}
and, by definition, $f(t) = 0$ for $t \ge 1$.
In the theory of RE, cf. Section \ref{sec:4}, in particular, Lemma \ref{lem:6.1}, \cite{GLS90} and Appendix A, the solution $r$ of
\begin{equation}\label{eq:7.9}
k * r + k = r = r * k + k
\end{equation}
is called the {\sl resolvent} of the kernel $k$, in particular since the solution of \sef{eq:7.7} is given by
\begin{equation}\label{eq:7.10}
b = f + r * f.
\end{equation}
Note that
\begin{equation}\label{eq:7.11}
r = \sum_{j=1}^\infty k^{j*}.
\end{equation}
As we will establish soon, the resolvent plays the role of fundamental solution in the present context. In order to have $f(t) = k(t)$ we need to replace in \sef{eq:7.8} $\ph(\th)\, d\th$ by the unit Dirac measure in zero. So we need to consider an initial condition that is not an integrable function, but rather a measure. Now recall that when working with continuous functions in the theory of delay differential equations, one finds that the fundamental solution corresponds to a discontinuous initial condition. Here the situation is reminiscent: while working with integrable functions, we find that the resolvent corresponds to a measure as initial condition (note that in the population dynamical context, the Dirac measure in zero represents a cohort of newborn individuals). And our strategy will be the same: enlarge the state space, even though this entails the loss of strong continuity.

In the tradition of delay equations we will represent measures by $NBV$ functions. So let now
\begin{equation}\label{eq:7.12}
Y = NBV\gb{[-1,0];\BR}
\end{equation}
but with the normalization convention that the elements are zero in the right end point $\th = 0$. Let
\begin{equation}\label{eq:7.13}
\DY = B\gb{[0,1];\BR}
\end{equation}
with pairing defined by 
\begin{equation}\label{eq:7.13a}
\pa{\dy}{y} = \int_{-1}^0 \dy(-\th)\,y(d\th)
\end{equation}
and let
\begin{equation}\label{eq:7.14}
k \in \DY
\end{equation}
be given. We still consider \sef{eq:7.7} but replace \sef{eq:7.8} by
\begin{equation}\label{eq:7.15}
f(t) = \int_{t-1}^0 k(t-\th)\psi(d\th)
\end{equation}
with $\psi \in Y$ considered as the initial condition. (So in the population dynamical context one should interpret $\psi$ as the cumulative number of newborns, but otherwise nothing changes. In particular there is still a population level birth {\sl rate} for $t > 0$. This will  change in Section \ref{sec:12}, where we work with cumulative quantities throughout.) For given $\psi \in Y$ equation \sef{eq:7.7} with $f$ given by \sef{eq:7.15} has a unique solution given explicitly by \sef{eq:7.10}. We define
\begin{equation}\label{eq:7.16}
B(t) = \int_0^t b(\tau)\,d\tau,\qquad t > 0,
\end{equation}
\begin{equation}\label{eq:7.17}
B(\th) = \psi(\th),\qquad \th \le 0,
\end{equation}
(where we have suppressed the dependence of $b$ on $\psi$ in the notation) and next $S(t) : Y \to Y$ by
\begin{equation}\label{eq:7.18}
\gb{S(t)\psi}(\th) = B(t+\th) - B(t).
\end{equation}
Note that we subtract $B(t)$ in order to comply with the normalization that the value in $\th = 0$ should be zero. It is very well possible to check that $\{S(t)\}$ is a twin semigroup on the norming dual pair $(Y,\DY)$ specified by \sef{eq:7.12} and \sef{eq:7.13} and to determine the generator via the Laplace transform. Here, however, we establish the relevant facts via the perturbation theory of Section \ref{sec:6}. This allows us to show that \sef{eq:5.8} and \sef{eq:7.7} are identical (in the sense that both the kernels $k$ and the forcing functions ($\dq S_0(\cdot) y$ and $f$, respectively) coincide; to call the kernel in \sef{eq:7.1} $k$, introduces the risk of ambiguity when invoking Section \ref{sec:6}, but in fact there is, as shall show, no need to worry).

The twin semigroup $\{S_0(t)\}$ defined by
\begin{equation}\label{eq:7.19}
\gb{S_0(t)\psi}(\th) =
\begin{cases}
\psi(t+\th) &t+\th \le 0\\
0              &t+\th > 0
\end{cases}
\end{equation}
corresponds to a kernel $k$ that is identically equal to zero, and so to trivial extension of the initial function. A straightforward calculation reveals that
\[\int_0^\infty e^{-\la t} \gb{S_0(t)\psi}(\th)\,dt = e^{\la\th}\int_\th^0 e^{-\la\si} \psi(\si)\,d\si\]
and next that
\[\int_0^\infty e^{-\la t}\dy S_0(t)\psi\,dt = \dy (\la I - C_0)^{-1}\psi,\]
where
\begin{align}\label{eq:7.20}
\begin{split}
\DOM{C_0} &= \bigl\{ \psi \mid \exists \ph \in Y \mid \psi(\th) = \int_0^\th \ph(\si)\,d\si\,\bigr\}\\
C_0\psi &= \bigl\{\ph \in Y \mid \psi(\th) = \int_0^\th \ph(\si)\,d\si\,\bigr\}.
\end{split}
\end{align}
Note that when $\ph_1$ and $\ph_2$ both belong to $C_0 \psi$ then -- they are equal in $\theta=0$, because of the normalization -- they are equal in $(-1,0)$, because they are equal almost everywhere in this interval and left (or, right, depending on the chosen normalization) continuous -- they might differ in $\theta=-1$ -- so if they differ, they differ by a jump in $-1$ (representing a Dirac mass in -1).

The subspace of strong continuity is given by
\begin{align}\label{eq:7.21}
X &= AC_0\gb{[-1,0];\BR}\nonumber\\
  &= \bigl\{ \psi \mid \exists \ph \in L^1\gb{[-1,0];\BR} \mid \psi(\th) = \int_0^\th \ph(\si)\,d\si\,\bigr\}
\end{align}
according to Theorem \ref{thm:3.1} and the fact that NBV functions are dense in $L^1$ (admittedly we ignore an isometric isomorphism when using the same symbol $X$ in \sef{eq:7.5} and \sef{eq:7.21}).

To capture  the true rule for extension, we define $q$ in $Y$ by
\begin{equation}\label{eq:7.22}
q(\th) = \begin{cases}
0 &\hbox{for } \th = 0\\
-1 &\hbox{for } -1 \le \th < 0
\end{cases}
\end{equation}
(i.e., $q$ is the Heaviside function that represents the Dirac measure in $\th = 0$) and, inspired by \sef{eq:5.2}, search for $\dq$ in $\DY$ such that
\begin{equation}\label{eq:7.23}
\dq S_0(t)q = k(t).
\end{equation}
It follows from \sef{eq:7.19} and \sef{eq:7.22} that
\begin{equation*}
\dq S_0(t)q = \begin{cases}
\dq(t)   &\mbox{for } 0 \le t \le 1\\
0         &\mbox{otherwise}
\end{cases}
\end{equation*}
so we can in fact identify $\dq$ and $k$.

The perturbed semigroup is defined by \sef{eq:5.10} and this involves the solution of \sef{eq:5.8}, which is a RE with kernel $k$ and forcing function
\begin{align*}
\dq S_0(t)\psi &= \int_{-1}^0 k(-\th)\gb{S_0(t)\psi}(d\th) = \int_{-1}^{-t} k(-\th)\psi(t+d\th)\\
                    &=\int_{t-1}^0 k(t-\si)\psi(d\si)
\end{align*}
which is exactly equal to $f(t)$ defined in \sef{eq:7.15}. We conclude that in the present setting \sef{eq:5.8} is simply another way of writing \sef{eq:7.7} and that, accordingly, we may replace $v(\tau)y$ in \sef{eq:5.10} by $b(\tau)$. It only remains to verify that \sef{eq:5.10} amounts to \sef{eq:7.18}.
  
With $Y$ and $\DY$ given by,  respectively, \sef{eq:7.12} and \sef{eq:7.13}, one can turn \sef{eq:5.10} into a pointwise equality (just use step functions from $\DY$ in the pairing that provides the precise meaning of the integral). It reads
\[\gb{S(t)\psi}(\th) = \gb{S_0(t)\psi}(\th) + \int_0^t \gb{S_0(t-\tau)q}(\th) b(\tau)\,d\tau\]
with
\[\gb{S_0(t)\psi}(\th) = \begin{cases}
\psi(t+\th)  & t+\th \le 0\\
0               & t+\th > 0
\end{cases}\]
and
\begin{align*}
\int_0^t \gb{S_0(t-\tau)q}(\th) b(\tau)\,d\tau &= \int_0^t - \chi_{t-\tau+\th < 0}\, b(\tau)\,d\tau\\
&= -\int_{\max\{t+\th,0\}}^t b(\tau)\,d\tau\\
&= \int_0^{\max\{t+\th,0\}} b(\tau)\,d\tau - \int_0^t b(\tau)\,d\tau.
\end{align*}
On account of \sef{eq:7.16}--\sef{eq:7.18} we conclude that the twin semigroup defined by \sef{eq:5.10} is equivalently described by \sef{eq:7.18}.

In terms of $B$ we can rewrite \sef{eq:7.1} as  the delay differential equation
\begin{equation}\label{eq:7.24}
B'(t) = \int_{-1}^0 k(-\si)B_t(d\si).
\end{equation}
If we formally differentiate \sef{eq:7.18} with respect to $t$ and next evaluate at $t=0$, we obtain for $\th < 0$
\[\frac{d}{dt}\gb{S(t)\psi}(\th)\big|_{t=0} = \psi'(\th) - B'(0) = \psi'(\th) - \int_{-1}^0 k(-\si)\psi(d\si)\]
which is completely in line with the characterization of the generator $C$ in \sef{eq:5.1} when \sef{eq:7.20}, \sef{eq:7.22} and $\dq = k$ are taken into account.

\vfill\eject
\vspace{1.5cm}
\centerline{\Large\bf Part III: Unbounded perturbations}
\smallskip
\centerline{\Large\bf describing neutral equations}
\bigskip

\section{Forcing functions with finite dimensional\\ range revisited}\label{sec:8}
\setcounter{equation}{0}

The aim of this section is to generalize the results of Section \ref{sec:5} in order to prepare for (the analysis of) relatively bounded perturbations in Sections 9 and 10 below. We now consider linear functionals
\begin{equation}\label{eq:8.2}
\dy \mapsto \int_0^t\dy  S(t-\tau)q \, F(d\tau)
\end{equation}
for a given $\BR$-valued $BV$ function $F$ and ask: when is such a functional represented by an element of $Y$?

The proof of Lemma \ref{lem:4.2} carries over verbatim if we replace $f(\tau)\,d\tau$ by $F(d\tau)$. The proof of Lemma \ref{lem:4.1}, on the other hand, breaks down. To save the underlying idea, we perform integration by parts and first rewrite \sef{eq:8.2} as
\begin{equation}\label{eq:8.3}
\dy \mapsto \int_0^t d_\si\bigl[\dy S(\si)q \bigr] F(t-\si) + F(t)\pa{\dy}{q}
\end{equation}
and next incorporate the last term into the first term by redefining $\dy S(\si)q$ as zero for $\si = 0$. In \sef{eq:8.3} we can allow $F$ to be a bounded measurable function, but we have to require that for every $\dy \in \DY$ the function
\[t \mapsto \dy S(t)q\qquad\hbox{for } t > 0\]
with value zero for $t=0$, is of bounded variation. Once this is assumed, the proof of Lemma \ref{lem:4.1} can be copied in order to show

\begin{lemma}\label{lem:8.1}
Let $\bigl\{S(t)\bigr\}$ be a twin semigroup on a norming dual pair $(Y,\DY)$. Assume \sef{eq:4.1} holds, i.e., assume that $(Y,\si(Y,\DY))$ is sequentially complete. Let $q \in Y$ be given. For $\dy \in \DY$ define
\begin{equation}\label{eq:8.4}
\dy W(\si) = \begin{cases}
0              &\hbox{for } \si = 0\\
\dy S(\si)q &\hbox{for } \si > 0.
\end{cases}
\end{equation}
Assume that for all $\dy \in \DY$ the function
\[\si \mapsto \dy W(\si)\]
belongs to $\NBV_{loc}\gb{[0,\infty),\BR}$. Let $F : [0,\infty) \to \BR$ be locally bounded and measurable. Then there exists $u(t) \in Y$ such that for all $\dy \in \DY$
\begin{equation}\label{eq:8.5}
\int_0^t d_\si\bigl[\dy W(\si)\bigr]\, F(t-\si) = \pa{\dy}{u(t)}.
\end{equation}
\end{lemma}

\medskip\noindent
For completeness we also state

\begin{lemma}\label{lem:8.2}
Let $\bigl\{S(t)\bigr\}$ be a twin semigroup on a norming dual pair $(Y,\DY)$. Assume \sef{eq:4.5} holds, i.e., assume that a linear map $\gb{\DY,\si(\DY,Y)} \to \BR$ is continuous if it is sequentially continuous. Let $q \in Y$ be given. Let $F : [0,\infty) \to \BR$ be of locally bounded variation. Then there exists $u(t) \in Y$ such that for all $\dy \in \DY$
\begin{equation}\label{eq:8.6}
\int_0^t \dy S(t-\tau)q \, F(d\tau) = \pa{\dy}{u(t)}.
\end{equation}
\end{lemma}

\medskip\noindent

\begin{lemma}\label{lem:8.3}
Let $\bigl\{S(t)\bigr\}$ be a twin semigroup on a norming dual pair $(Y,\DY)$. Assume \sef{eq:4.8} holds, i.e., assume that $(\DY,\si(\DY,Y))$ is sequentially complete. Let $\dq \in \DY$ be given. For $y \in Y$ define
\begin{equation}\label{eq:8.7}
V(\si)y = \begin{cases}
0              &\hbox{for } \si = 0\\
\dq S(\si)y &\hbox{for } \si > 0.
\end{cases}
\end{equation}
Assume that for all $y \in Y$ the function
\[\si \mapsto V(\si)y\]
belongs to $\NBV_{loc}\gb{[0,\infty),\BR}$. Let $F : [0,\infty) \to \BR$ be locally bounded and measurable. Then there exists $\du(t) \in \DY$ such that for all $y \in Y$
\begin{equation}\label{eq:8.8}
\int_0^t F(t-\si)\,d_\si\bigl[V(\si)y\bigr]\,  = \pa{\du(t)}{y}.
\end{equation}
\end{lemma}

\medskip\noindent

\begin{lemma}\label{lem:8.4}
Let $\bigl\{S(t)\bigr\}$ be a twin semigroup on a norming dual pair $(Y,\DY)$. Assume \sef{eq:4.9} holds, i.e., assume that a linear map $\gb{Y,\si(Y,\DY)} \to \BR$ is continuous if it is sequentially continuous. Let $\dq \in \DY$ be given. Let $F : [0,\infty) \to \BR$ be of locally bounded variation. Then there exists $\du(t) \in \DY$ such that for all $y \in Y$
\begin{equation}\label{eq:8.9}
\int_0^t F(d\tau) \,\dq S(t-\tau)y \,  = \pa{\du(t)}{y}.
\end{equation}
\end{lemma}

\newpage

\section{The main ideas explained by formula mani\-pulation}\label{sec:9}
\setcounter{equation}{0}

In Section \ref{sec:6} we perturbed the abstract ODE
\begin{equation}\label{eq:9.1}
\frac{du}{dt} \in C_0u
\end{equation}
by adding at the right hand side a {\sl bounded} finite rank perturbation. Here, instead, we shall add a {\sl relatively bounded} finite rank perturbation. Again we introduce $q_j \in Y$, $j = 1,\ldots, n$, to span the range of the perturbation. But the coefficients are now of the form $\pa{\DQ_j}{C_0u}$ for given $\DQ_j \in \DY$ (so we use a capital letter to alert the reader that the element of $\DY$ does now act on $C_0u$, rather than on $u$ itself). Thus the aim is to study
\begin{equation}\label{eq:9.2}
\frac{du}{dt} \in C_0u + q \cdot \DQ C_0u
\end{equation}
with
\begin{equation}\label{eq:9.3}
q \cdot \DQ C_0 u = \sum_{j=1}^n \pa{\DQ_j}{C_0u} q_j.
\end{equation}

We assume that $C_0$ is the generator of a twin semigroup $\{S_0(t)\}$ and our aim is to construct a twin semigroup $\{S(t)\}$ with a generator that has $\DOM{C_0}$ as its domain of definition and action given by the right hand side of \sef{eq:9.2}.

The construction starts from the variation-of-constants formula (cf. \sef{eq:5.5})
\begin{equation}\label{eq:9.4}
S(t) = S_0(t) + \int_0^t S_0(t-\tau)q \cdot \DQ C_0 S(\tau)\,d\tau,
\end{equation}
or from the variant (cf. \sef{eq:5.12})
\begin{equation}\label{eq:9.5}
S(t) = S_0(t) + \int_0^t S(t-\tau)q \cdot \DQ C_0 S_0(\tau)\,d\tau
\end{equation}
in which the roles of $\{S_0(t)\}$ and $\{S(t)\}$ are interchanged. Again the construction is based on solving a finite dimensional RE, but now this RE involves the Stieltjes integral and a bounded variation kernel (or, equivalently, a measure, cf. Appendix A).

To see how the Stieltjes integral might come in, please recall Lemma \ref{lem:2.3} and note that this suggests to replace the second term at the right hand side of \sef{eq:9.5} by
\[\int_0^t S(t-\tau)q \cdot d_\tau\DQ\gb{S_0(\tau) - I}\]
when indeed $\tau \mapsto \DQ \gb{S_0(\tau) - I}y$ is $(N)BV$ for all $y \in Y$. Another option, again motivated by Lemma \ref{lem:2.3}, is to integrate the second term at the right hand side of \sef{eq:9.4} by parts while assuming that $\tau \mapsto \dy S_0(\tau)q$ is $BV$ for all $\dy \in \DY$. The point of both options is to ``neutralize" the unbounded operator $C_0$ and the price we pay is that we have to work with Stieltjes integrals.
 
Define $V_0(t) : Y \to \BR^n$ by
\begin{align}\label{eq:9.6}
V_0(t)y &= \DQ C_0 \int_0^t S_0(\si)\,d\si y\nonumber\\
&= \DQ\gb{S_0(t) - I}y
\end{align}
and define a $\BR^{n \times n}$-valued kernel $K$ by
\begin{equation}\label{eq:9.7}
K(t) = \DQ\gb{S_0(t) - I}q
\end{equation}
or, in more detail, 
\begin{equation}\label{eq:9.8}
K_{ij}(t) = \DQ_i\gb{S_0(t) - I}q_j.
\end{equation}

If we first change the integration variable in \sef{eq:9.4} to $\si = t - \tau$, next integrate both sides of the equation with respect to time, and finally apply $\DQ C_0$ to both sides, we obtain the equation
\begin{equation}\label{eq:9.9}
V(t) = V_0(t) + \int_0^t K(d\si) V(t-\si)
\end{equation}
with
\begin{equation}\label{eq:9.10}
V(t) = \DQ C_0 \int_0^t S(\tau)\,d\tau.
\end{equation}
Here \sef{eq:9.9} is short hand for
\begin{equation}\label{eq:9.11}
V(t)y = V_0(t)y + \int_0^t K(d\si) V(t-\si)y
\end{equation}
which is, for given $y \in Y$, an equation for the $\BR^n$-valued function $t \mapsto V(t)y$. Introducing the notation (cf. Theorem \ref{thm:convo-meas-Borel-fun} while taking Theorem \ref{thm:bv2} into account to switch back and forth between measures and NBV functions)
\begin{equation}\label{eq:9.12}
\gb{K \bast V}(t) := \int_0^t K(d\si)V(t-\si),
\end{equation}
we can write \sef{eq:9.9}, and hence \sef{eq:9.11}, in the even more compact form
\begin{equation}\label{eq:9.13}
V = V_0 + K \bast V.
\end{equation}

Let $R$ be the resolvent of $K$, i.e., the unique solution of (cf. Theorem \ref{thm:resolvent})
\begin{equation}\label{eq:9.14}
K + R \bast K = R = K + K \bast R,
\end{equation}
then the solution of \sef{eq:9.13} is given by
\begin{equation}\label{eq:9.15}
V = V_0 + R \bast V_0.
\end{equation}
Define $W_0(t) : \DY \to \BR^n$ by
\begin{equation}\label{eq:9.16}
\dy W_0(t) = \dy S_0(t)q,\qquad\hbox{for } t > 0,
\end{equation}
with value zero for $t=0$, where we allow ourselves once more the freedom of writing the element of $\DY$, on which the operator acts, to the left of the operator itself. By applying \sef{eq:9.5} to $q$ we obtain the equation
\begin{equation}\label{eq:9.17}
W(t) = W_0(t) + \int_0^t W(t-\tau)K(d\tau)
\end{equation}
and accordingly we find for
\begin{equation}\label{eq:9.18}
W(t) = S(t)q,\qquad\hbox{for } t > 0,
\end{equation}
with value zero for $t=0$, the formula
\begin{equation}\label{eq:9.19}
W(t) = W_0(t) + \int_0^t W_0(t-\tau)R(d\tau).
\end{equation}
Again we abbreviate and write \sef{eq:9.17} as
\begin{equation}\label{eq:9.20}
W = W_0 + W \bast K
\end{equation}
and \sef{eq:9.19} as
\begin{equation}\label{eq:9.21}
W = W_0 + W_0 \bast R.
\end{equation}
(Please note a notational difficulty: in principle we would like to indicate by the order of the factors in the product which of the two factors is considered as a measure, but, on the other hand, we also want to indicate by the order how the matrix acts on the vector. In \sef{eq:9.20} and \sef{eq:9.21} we sacrificed the first in order to realize the second.)

Motivated by \sef{eq:9.4} we now define
\begin{equation}\label{eq:9.22}
S(t) = S_0(t) + \int_0^t W_0(t-\tau) \cdot V(d\tau)
\end{equation}
by which we mean that
\[\dy S(t)y = \dy S_0(t)y + \int_0^t \dy W_0(t-\tau) \cdot V(d\tau)y.\]
This is compatible with the formula
\begin{equation}\label{eq:9.23}
S(t) = S_0(t) + \int_0^t W(t-\tau) \cdot V_0(d\tau)
\end{equation}
that is motivated by \sef{eq:9.5}. (Because of \sef{eq:9.15} and \sef{eq:9.21} the checking amounts to verifying
\[W_0 \bast \gb{V_0 + R \bast V_0} = \gb{W_0 + W_0 \bast R} \bast V_0\]
which is a direct consequence of the associativity and distributivity of the $\bast$-convolution product.) 

To make all this work, we need in any case that the kernel $K$ defined by $\sef{eq:9.7}$ is of bounded variation. In the next section we shall indeed assume that $t \mapsto K(t)$ belongs to $NBV_{loc}$ (note that \sef{eq:9.7} is compatible with $K(0) = 0$). The formulas \sef{eq:9.22} and \sef{eq:9.23} are based on the additional assumption that for all $y \in Y$ the function
\begin{equation}\label{eq:9.23a}
t \mapsto V_0(t)y = \DQ\gb{S_0(t) - I}y\hbox{ belongs to } NBV_{loc}.
\end{equation}
Note that \sef{eq:9.15} guarantees that this property of $V_0$ is inherited by $V$, making also \sef{eq:9.22} well-defined.

In Section \ref{sec:12}, when applying the theory to renewal equations involving a measure as kernel, we shall find that this assumption indeed holds. But in Section \ref{sec:11}, when dealing with NFDE (neutral functional differential equations), we shall need to replace \sef{eq:9.22} and \sef{eq:9.23} by their counterparts
\begin{equation}\label{eq:9.24}
S(t) = S_0(t) + \int_0^t W_0(d\si) \cdot V(t-\si)
\end{equation}
and
\begin{equation}\label{eq:9.25}
S(t) = S_0(t) + \int_0^t W(d\si) \cdot V_0(t-\si)
\end{equation}
that are obtained by partial integration. Note carefully that \sef{eq:9.24} and \sef{eq:9.25} are based on the definition
\begin{equation}\label{eq:9.26}
W_0(0) = 0\quad\hbox{and}\quad W(0) = 0
\end{equation}
and therefore both of these have a jump of size $q$ in zero, cf. \sef{eq:9.16} and \sef{eq:9.17}. By this we mean that both $\dy W_0(t)$ and $\dy W(t)$ have $\pa{\dy}{q}$ as limit for $t \da 0$. When working with \sef{eq:9.24} or \sef{eq:9.25}, we replace the earlier additional assumption \sef{eq:9.23a} by the new additional assumption that for all $\dy \in \DY$
\begin{equation}\label{eq:9.26a}
\hbox{the function } t \mapsto \dy W_0(t) = \dy S_0(t)q\hbox{ belongs to } NBV_{loc},
\end{equation}
where we define the function to be zero at $t=0$, cf. \sef{eq:9.26}.

\section{Relatively bounded finite dimensional range perturbation of twin semigroups}\label{sec:10}
\setcounter{equation}{0}

Throughout this section we assume

\begin{itemize}
\item $\{S_0(t)\}$ is a twin semigroup with generator $C_0$;
\item the elements $q_j \in Y$ and $\DQ_j \in \DY$, $j=1,2,\ldots,n$, are such that for $i,j = 1,\ldots,n$ the function
\begin{equation}\label{eq:10.1}
t \mapsto K_{ij}(t) := \DQ_i \gb{S_0(t) - I}q_j
\end{equation}
belongs to $NBV_{loc}\gb{[0,\infty);\BR}$ and is continuous in $t = 0$.
\end{itemize}

Moreover, we use the following notation and definitions

\begin{itemize}
\item $R$ denotes the resolvent of $K$, i.e., the solution of \sef{eq:9.14}. Note that $R$ too is continuous in $t=0$.
\item $V_0(t) : Y \to \BR^n$ is for $t \ge 0$ defined by
\begin{equation}\label{eq:10.2}
V_0(t) = \DQ\gb{S_0(t) - I}.
\end{equation}
\item $W_0(t) : \DY \to \BR^n$ is for $t > 0$ defined by
\begin{equation}\label{eq:10.3}
W_0(t) = S_0(t)q
\end{equation}
and $W_0(0) = 0$.
\item $V(t) : Y \to \BR^n$ is for $t \ge 0$ defined by
\begin{equation}\label{eq:10.4}
V(t) = V_0(t) + \int_0^t R(d\tau)\,V_0(t-\tau).
\end{equation}
\item $W(t) : \DY \to \BR^n$ is for $t > 0$ defined by
\begin{equation}\label{eq:10.5}
W(t) = W_0(t) + \int_0^t W_0(t-\tau)\, R(d\tau)
\end{equation}
and $W(0) = 0$.
\end{itemize}

\smallskip

We now formulate two theorems. The first will be used in Section \ref{sec:12} to deal with neutral RE. As we show in the next section, the second covers NFDE.

\begin{theorem}\label{thm:10.1}
Let $(Y,\DY)$ be a norming dual pair such that \sef{eq:4.8} and \sef{eq:4.5} hold. Assume that $t \mapsto V_0(t)y$ belongs to $NBV_{loc}\gb{[0,\infty); \BR^n}$ for all $y \in Y$. Then the same holds for the function $t \mapsto V(t)y$ and
\begin{equation}\label{eq:10.6}
S(t) = S_0(t) + \int_0^t W_0(t-\tau) \cdot V(d\tau)
\end{equation}
defines a twin semigroup with generator $C$ given by
\begin{equation}\label{eq:10.7}
\DOM{C}= \DOM{C_0},\qquad Cy = C_0y + \pa{\DQ}{C_0y}\cdot q.
\end{equation}
\end{theorem}

\begin{theorem}\label{thm:10.2}
Let $(Y,\DY)$ be a norming dual pair such that \sef{eq:4.1} and \sef{eq:4.9} hold.
Assume that $t \mapsto \dy W_0(t)$ belongs to $NBV_{loc}\gb{[0,\infty); \BR^n}$ for all $\dy \in \DY$. Then the same holds for the function $t \mapsto \dy W(t)$ and
\begin{equation}\label{eq:10.8}
S(t) = S_0(t) + \int_0^t W(d\tau) \cdot V_0(t-\tau)
\end{equation} 
defines a twin semigroup with generator $C$ given by \sef{eq:10.7}.
\end{theorem}

\medskip\noindent
{\sl Proof of Theorem $\ref{thm:10.1}$.} We follow the lines of the proof of Theorem \ref{thm:5.1}, but adapt the details to the somewhat different situation. In order to show that
\[(\dy,y) \mapsto \int_0^t \dy W_0(t-\tau) \cdot V(d\tau)y\]
defines a twin operator, we have to verify

\begin{itemize}
\item[(i)] for given $y \in Y$, the linear functional on $\DY$ defined by
\[\dy \mapsto \int_0^t \dy W_0(t-\tau) \cdot V(d\tau)y\]
is represented by an element of $Y$ and
\item[(ii)] for given $\dy \in \DY$, the linear functional on $Y$ defined by
\[y \mapsto \int_0^t \dy W_0(t-\tau) \cdot V(d\tau)y\]
is represented by an element of $\DY$.
\end{itemize}

To verify (i) we invoke Lemma \ref{lem:8.2} and to verify (ii) we invoke Lemma \ref{lem:8.3}. So $S(t)$ defined by \sef{eq:10.6} is a twin operator and we proceed by verifying properties (i)--(iv) of Definition \ref{def:2.1}.

We start with the semigroup property (i). Clearly $S(0) = S_0(0) = I$. To verify the semigroup property, we first derive \sef{eq:10.9}, i.e., we show that $\bigl\{V(t)\bigr\}$ is a cumulative output family \cite{DGT93} for the semigroup $\bigl\{S(t)\bigr\}$.

\begin{lemma}\label{lem:10.3}
We have
\begin{equation}\label{eq:10.9}
V(t+s) - V(t) = V(s)S(t).
\end{equation}
\end{lemma}

\begin{proof}
From \sef{eq:9.11} we deduce that
\begin{equation}\label{eq:10.10}
V(t+s)y - V(t)y = f(t,y) + \int_0^s K(d\si)\,\bigl[V(t+s-\si)y - V(t)y\bigr]
\end{equation}
with
\begin{align}\label{eq:10.11}
f(t,y) &= V_0(t+s)y - V_0(t)y + \int_s^{t+s} K(d\si)\,V(t+s-\si)y\nonumber\\
&\qquad\qquad+ K(s)V(t)y - \int_0^t K(d\si)\, V(t-\si)y.
\end{align}
We claim that
\begin{equation}\label{eq:10.12}
f(t,y) = V_0(s)S(t)y.
\end{equation}
If the claim is justified, we can write \sef{eq:10.10} as
\begin{equation}\label{eq:10.13}
U(s)y = V_0(s)S(t)y + \int_0^s K(d\si)\, U(s-\si)y
\end{equation}
with
\begin{equation}\label{eq:4.19bis}
U(s)y := V(t+s)y - V(t)y.
\end{equation}
Comparing \sef{eq:10.13} to \sef{eq:9.11} we conclude that
\begin{equation}\label{eq:4.19}
U(s)y = V(s)S(t)y
\end{equation}
which, on account of \sef{eq:4.19bis}, amounts to \sef{eq:10.9}.

To verify the claim, we first rewrite \sef{eq:10.11} as
\[f(t,y) = V_0(s)S_0(t)y + \int_0^t \bigl[K(s+d\tau)-K(d\tau)\bigr]\,V(t-\tau)y + K(s)V(t)y\]
and observe that we need to show that
\[V_0(s)\bigl[S(t)y - S_0(t)y\bigr] = \int_0^t \bigl[K(s+d\tau)-K(d\tau)\bigr]V(t-\tau)y + K(s)V(t)y.\]

The left hand side equals
\begin{align*}
&\DQ\gb{S_0(s)-I}\int_0^t S_0(t-\tau)q V(d\tau)y\\
&\qquad\qquad\qquad= \int_0^t \DQ\gb{S_0(t+s-\tau) - I - S_0(t-\tau) + I}q\,V(d\tau)y\\
&\qquad\qquad\qquad= \int_0^t \bigl[K(t+s-\tau) - K(t-\tau)\bigr]\,V(d\tau)y.
\end{align*}
Partial integration shows that this is equal to the right hand side.
\end{proof}

To prove the exponential estimate (ii) for $\dy S(t)y$, first note that the exponential estimates for $\dy S_0(t)y$ directly yield exponential estimates for $V_0(t)y$, $\dy W_0(t)$ and $K(t)$. The exponential estimate for $R(t)$ follows from Theorem \ref{thm:resolvent} and the exponential estimate for $K(t)$. Using \sef{eq:10.4} the exponential estimate for $V(t)$ follows from the exponential estimates for $V_0(t)y$ and $R(t)$ together with Theorem \ref{thm:convo-meas-prop}. The exponential estimate for $\dy S(t)y$ now follows from the exponential estimates for $\dy W_0(t)$ and $V(t)$ and again Theorem \ref{thm:convo-meas-prop}.

The proof of (iii) that $t \mapsto \dy S(t)y$ is measurable follows from the measurability of both $t \mapsto \dy S_0(t)y$ and, using Theorem \ref{thm:convo-meas-Borel-fun},
\[t \mapsto \int_0^t \dy W_0(t-\tau) \cdot V(d\tau)y.\]

It remains to prove (iv). In order to compute the Laplace transform, we use the notation
\begin{equation}\label{eq:10.15}
\what{K}(\la) := \int_0^\infty e^{-\la \tau}\,K(d\tau)
\end{equation}
and note that, since $K(0) = 0$, we have
\begin{equation}\label{eq:10.16}
\what{K}(\la) = \la \wbar{K}(\la)
\end{equation}
with $\wbar{K}(\la) := \int_0^\infty e^{-\la \tau} K(\tau)\,d\tau$.

The relation
\begin{equation}\label{eq:10.17}
\wbar{V}(\la) = \gb{I - \what{K}(\la)}^{-1}\wbar{V}_0(\la)
\end{equation}
can either be derived from \sef{eq:9.9} or from \sef{eq:10.4} in combination with \sef{eq:9.14}. From \sef{eq:10.2} we deduce that 
\begin{equation}\label{eq:10.18}
\wbar{V}_0(\la) = \DQ\gb{\wbar{S}_0(\la) - \la^{-1} I}
\end{equation}
and from \sef{eq:10.1} that
\begin{equation}\label{eq:10.19}
\what{K}(\la) = \la \wbar{K}(\la) = \la \DQ \wbar{S}_0(\la)q - \pa{\DQ}{q}.
\end{equation}
Laplace transformation of \sef{eq:10.6} yields, when using \sef{eq:10.3},
\begin{equation}\label{eq:10.20}
\wbar{S}(\la) = \wbar{S}_0(\la) + \la \wbar{S}_0(\la) q \wbar{V}(\la).
\end{equation}
Since $\wbar{S}_0(\la)$ is a bounded bilinear map from $\DY \times Y$ to $\BR$ and $\wbar{V}(\la)$ is a bounded linear map from $Y$ to $\BR^n$, \sef{eq:10.20} defines a bounded bilinear map from $\DY \times Y$ to $\BR$ as well as a bounded linear map from $Y$ to $Y$ (a perturbation of $\wbar{S}_0(\la)$ with an operator with range spanned by $\wbar{S}_0(\la)q$). From
\sef{eq:10.17} and \sef{eq:10.18} it follows that  $\wbar{V}(\la)$ is defined in terms of functionals that belong to $\DY$. Hence \sef{eq:10.20} defines as well a bounded linear operator from $\DY$ to $\DY$. We conclude that \sef{eq:10.20} defines a twin operator and thus verified (iv) of Definition \ref{def:2.1}.

We proceed by showing that the right hand side of \sef{eq:10.20} equals $(\la I - C)^{-1}$ with $C$ defined in \sef{eq:10.7}. 
By combining \sef{eq:10.20}, \sef{eq:10.17}, \sef{eq:10.18} and \sef{eq:10.19} we find that
\begin{align}\label{eq:10.21}
\wbar{S}(\la) &= \wbar{S}_0(\la) + \la \wbar{S}_0(\la)q\Bigl(I + \pa{\DQ}{q} - \la \DQ \wbar{S}_0(\la)q\Bigr)^{-1}\nonumber\\
&\qquad\qquad\qquad\qquad\qquad\times \DQ\gb{\wbar{S}_0(\la) - \la^{-1} I}.
\end{align}
The equation $(\la I - C)z = y$ can be rewritten in the form
\[(\la I - C_0)z = y + \pa{\DQ}{C_0z}q.\]
So it follows that
\[z = (\la I - C_0)^{-1}y + \pa{\DQ}{C_0z}(\la I - C_0)^{-1}q\]
and hence that $\pa{\DQ}{C_0z}$ should satisfy the equation
\begin{align*}
\pa{\DQ}{C_0z} &= \pa{\DQ}{C_0(\la I - C_0)^{-1}y} + \pa{\DQ}{C_0z}\pa{\DQ}{C_0(\la I - C_0)^{-1}q}.
\end{align*}
So necessarily
\begin{align*}
\pa{\DQ}{C_0z} &= \gb{ I - \pa{\DQ}{C_0(\la I - C_0)^{-1}q}}^{-1} \pa{\DQ}{C_0(\la I - C_0)^{-1}y}.
\end{align*}
Using
\[C_0(\la I - C_0)^{-1} = \la(\la I - C_0)^{-1} - I\]
we find that
\begin{align}\label{eq:10.21a}
z &= (\la I - C)^{-1}y\nonumber\\
&= (\la I - C_0)^{-1}y + \bigl[ I + \pa{\DQ}{q} - \la \pa{\DQ}{(\la I - C_0)^{-1}q}\bigr]^{-1}\nonumber\\
&\qquad\qquad\quad \times \la (\la I - C_0)^{-1}q\pa{\DQ}{(\la I - C_0)^{-1}y - \la^{-1}y}\nonumber\\
&= (\la I - C_0)^{-1}y + \bigl[ I - \what K(\la)\bigr]^{-1}\nonumber\\
&\qquad\qquad\quad \times \la (\la I - C_0)^{-1}q\pa{\DQ}{(\la I - C_0)^{-1}y - \la^{-1}y}
\end{align}
Since $(\la I - C_0)^{-1} = \wbar{S}_0(\la)$ this is identical to the right hand side of \sef{eq:10.21}.
\QED

\medskip\noindent
{\sl Proof of Theorem $\ref{thm:10.2}$.} We follow the lines of the proof of Theorem \ref{thm:10.1}, but adapt the details. In order to show that
\[(\dy,y) \mapsto \int_0^t \dy W(d\tau) \cdot V_0(t-\tau)y\]
defines a twin operator, we have to verify

\begin{itemize}
\item[(i)] for given $y \in Y$, the linear functional on $\DY$ defined by
\[\dy \mapsto \int_0^t \dy W(d\tau) \cdot V_0(t-\tau)y\]
is represented by an element of $Y$ and
\item[(ii)] for given $\dy \in \DY$, the linear functional on $Y$ defined by
\[y \mapsto \int_0^t \dy W(d\tau) \cdot V_0(t-\tau)y\]
is represented by an element of $\DY$.
\end{itemize}

To verify (i) we invoke Lemma \ref{lem:8.1} and to verify (ii) we invoke Lemma \ref{lem:8.4}. So $S(t)$ defined by \sef{eq:10.6} is a twin operator and the verification of properties (ii)--(iv) of Definition \ref{def:2.1} proceeds as in the proof of Theorem \ref{thm:10.1}.
\QED

\medskip
Exactly as in Section \ref{sec:6}, the special representation of the perturbed semigroup $S(t)$ given in respectively \sef{eq:10.6} and \sef{eq:10.8} allows us to derive a strong result about the asymptotic behaviour of $S(t)$. 

\begin{theorem}\label{thm:10.4}
Let $K$ be given by \sef{eq:10.1} and let $S(t)$ be given by \sef{eq:10.6} or \sef{eq:10.8} with generator $C$ given by \sef{eq:10.7}. Suppose that $S_0(t)$ is bounded and that $(\la I - C_0)^{-1}$ is bounded for $\Re \la \ge 0$. If
\begin{equation}\label{eq:cond-asymp-behav}
\inf_{\Re z \ge 0} \bigl | \det \gb{ I - \what K(z)} \bigr| > 0,
\end{equation}
then
\begin{equation}\label{eq:10.25}
\|S(t)C^{-1}\| \to 0\qquad\hbox{as}\quad t \to \infty.
\end{equation}
As a consequence we have that $S(t)y \to 0$ as $t \to \infty$ for every $y$ in the norm-closure of $\DOM{C}$.
\end{theorem}

\begin{proof}
We first show that the semigroup $S(t)$ is bounded. From the half-line Gel'fand theorem, see Theorem \sef{thm:Gelfand}, it follows that the resolvent $R$ of $K$ belongs to $NBV\gb{[0,\infty); \BR^n}$. Fix $y \in Y$ and $\dy \in \DY$. Since $S_0(t)$ is a bounded semigroup, $t \mapsto V_0(t)y$ and $t \mapsto \dy W_0(t)$ are bounded Borel functions on $[0,\infty)$. 

Suppose $S(t)$ is given by \sef{eq:10.6}. From the assumption that $t \mapsto V_0(t)y$ belongs to $NBV_{loc}\gb{[0,\infty); \BR^n}$ and the fact that  $t \mapsto V_0(t)y$ is bounded, it follows from Theorem \ref{thm:convo-meas-fun2} that $t \mapsto V(t)y$ defined by \sef{eq:10.4} belongs to $NBV\gb{[0,\infty); \BR^n}$ as well. Therefore, it follows from Theorem \ref{thm:convo-meas-Borel-fun} that there exists $M \ge 0$ such that $|\dy S(t)y| \le M\|\dy\|\,\|y\|$.  

Suppose $S(t)$ is given by \sef{eq:10.8}.  From the assumption that $t \mapsto \dy W_0(t)$ belongs to $NBV_{loc}\gb{[0,\infty); \BR^n}$ and the fact that  $t \mapsto \dy W_0(t)$ is bounded, it follows from Theorem \ref{thm:convo-meas-fun2} that $t \mapsto \dy W(t)$ defined by \sef{eq:10.5} belongs to $NBV\gb{[0,\infty); \BR^n}$ as well. Therefore it follows from Theorem \ref{thm:convo-meas-Borel-fun} that there exists $M \ge 0$ such that $|\dy S(t)y| \le M\|\dy\|\,\|y\|$ and this proves that $S(t)$ is bounded.

Finally, the representation \sef{eq:10.21a} implies that, under the assumptions of the theorem, $(\la I - C)^{-1}$ is bounded, for $\Re \la \ge 0$. This completes the proof that $S(t)$ is bounded and that $\si(C) \mcap i\BR = \emptyset$. So an application of Theorem \ref{thm:Tauber} yields the proof.
\end{proof}

\section{NFDE -- Neutral Functional Differential\\ Equations}\label{sec:11}
\setcounter{equation}{0}

Much of our motivation for developing the abstract perturbation theory of Section \ref{sec:10} came from our interest in the NFDE
\begin{equation}\label{eq:11.1} 
\frac{d}{dt}\bigl[x(t) - \int_{[0,1]} d\eta(\si)x(t-\si)\bigr] = \int_{[0,1]} d\ze(\si)x(t-\si),\qquad t > 0,
\end{equation}
with initial condition
\begin{equation}\label{eq:11.2}
x(\th) = \ph(\th),\qquad -1 \le \th \le 0.
\end{equation}
Here both $\eta$ and $\ze$ belong to $NBV\gb{[0,1],\BR^{n \times n}}$ and $\ph \in B\gb{[-1,0],\BR^n}$. So we work with the norming dual pair
\begin{equation}\label{eq:11.3}
Y = B\gb{[-1,0],\BR^n},\qquad \DY = NBV\gb{[0,1],\BR^{n}}
\end{equation}
with pairing
\begin{equation}\label{eq:11.4}
\pa{\dy}{y} = \int_{[0,1]} d\dy(\si) \cdot y(-\si).
\end{equation}
The rows of both $\ze$ and $\eta$ are considered as elements of $\DY$. Concerning $\eta$ we additionally assume that
\begin{equation}\label{eq:11.5}
\eta\hbox{ is continuous at zero},
\end{equation}
the idea being that we normalize the jump at zero and write its contribution separately as the term $x(t)$ at the left hand side of \sef{eq:11.1}.

The special case that $\eta$ in \sef{eq:11.1} is identically zero was considered in Section \ref{sec:4}. Here we take the twin semigroup constructed in that section as our starting point. In order to stay in line with the framework of Section \ref{sec:10}, we add an index zero when referring to this ``unperturbed" semigroup $S_0(t)$:
\begin{equation}\label{eq:11.6}
S_0(t)\hbox{ is defined by the right hand side of \sef{eq:6.18}}
\end{equation}
with generator
\begin{align}\label{eq:11.7}
\DOM{C_0} &= \hbox{Lip}\gb{[-1,0],\BR^n}\nonumber\\
C_0\ph &= \bigl\{ \ph' \in Y \mid \ph(\th) = \ph(-1) + \int_{-1}^\th \ph'(\si)\,d\si,\quad \ph'(0) = \pa{\ze}{\ph}\,\bigr\}.
\end{align}
Equivalently
\begin{equation}\label{eq:11.8}
S_0(t)\ph = z(t+\novar;\ph),
\end{equation}
where $z$ is the unique solution of \sef{eq:11.1}--\sef{eq:11.2} with $\eta = 0$.

As in \sef{eq:6.29} we define, for $i=1,\ldots,n$, the elements $q_i \in Y$ by
\begin{equation}\label{eq:11.9}
q_i(\th) = \begin{cases}
0 &\th < 0\\
e_i &\th = 0,
\end{cases}
\end{equation}
where $e_i$ is the $i$-th unit vector in $\BR^n$. The elements $Q_i^\diamond$ of $\DY$ are defined by
\begin{equation}\label{eq:11.10}
Q_i^\diamond(\th) = \eta_i(\th),
\end{equation}
where $\eta_i$ is the $i$-th row of the matrix valued function $\eta$.

The aim of the present section is to show that, with these definitions, the twin semigroup $\bigl\{S(t)\bigr\}$ defined in Theorem \ref{thm:10.2} is exactly the semigroup of solution operators of \sef{eq:11.1}--\sef{eq:11.2} (here, as detailed below, solution refers to the integral equation obtained from \sef{eq:11.1}--\sef{eq:11.2} by integration with respect to time; it is straightforward to prove existence and uniqueness of a solution for this integral equation, cf. \cite[Theorem 9.1.2]{HVL93}).

At a formal level this is immediate: if we rewrite \sef{eq:11.1} as
\[\dot x(t) = \int_{[0,1]} d\eta(\si)\dot x(t-\si) + \int_{[0,1]} d\ze(\si)x(t-\si)\]
and proceed as in the formal derivation of \sef{eq:6.10} from \sef{eq:6.1}, we obtain 
\[\frac{du}{dt} \in C_0u + q \cdot \pa{\DQ}{C_0u}\]
by making the crucial observation that, on account of \sef{eq:11.10} and \sef{eq:11.5}, the value of $\gb{C_0u}(0)$ is irrelevant when evaluating the second term at the right hand side.

The rigorous proof of the general case, presented below, involves an unpleasant amount of formula manipulation. We therefore first present the proof for the relatively simple situation that the kernel $\ze$ in \sef{eq:11.1} is identically zero. In that case we have
\begin{equation}\label{eq:11.11}
\gb{S_0(t)y}(\th) = y(t+\th),
\end{equation}
where by definition $y(t) = y(0)$ for $t \ge 0$. Hence
\begin{align}\label{eq:11.12}
V_0(t)y &= \DQ\gb{S_0(t) - I}y\nonumber\\
&= \int_{[0,1]} d\eta(\si)\bigl[y(t-\si) - y(-\si)\bigr]\nonumber\\
&= \int_{(t,1]} d\eta(\si)y(t-\si) + \eta(t)y(0) - \int_{[0,1]} d\eta(\si)y(-\si).
\end{align}
It follows that
\begin{equation}\label{eq:11.13}
K(t) = V_0(t)q = \eta(t).
\end{equation}
Moreover
\begin{align}\label{eq:11.14}
\dy W_0(t) = \dy S_0(t)q &= \int_{[0,1]} d\dy(\si)q(t-\si)\nonumber\\
&= \int_{[0,t]} d\dy(\si) = \dy(t)
\end{align}
(where now $\dy(t) = \dy(1)$ for $t \ge 1$, by definition) and accordingly
\begin{equation}\label{eq:11.15}
\int_0^t \dy W_0(d\tau) \cdot V(t-\tau)y = \int_0^t d\dy(d\tau) \cdot V(t-\tau)y.
\end{equation}
Let, for $\th \in [0,1]$,
\begin{equation}\label{eq:11.16}
y^\diamond_\th(\si) = \begin{cases}
0,  &0 \le \si < \th;\\
(1,1,\ldots,1)^T, &\th \le \si \le 1.
\end{cases}
\end{equation}
The identity
\[y_\th^\diamond S(t)y = y_\th^\diamond S_0(t)y + \int_0^t y_\th^\diamond W_0(d\tau) \cdot V(t-\tau)y\]
reads
\begin{align}\label{eq:11.17}
\gb{S(t)y}(-\th) &= \gb{S_0(t)y}(-\th) + \begin{cases}
0 &\hbox{if } \th \not\in [0,t]\\
V(t-\th)y &\hbox{if } \th \in [0,t]
\end{cases}\nonumber\\
&= y(t-\th) + V(t-\th)y\, \chi_{[0,t]}(\th).
\end{align}
So in order to establish that the twin semigroup of Theorem \ref{thm:10.2} is indeed the semigroup of solution operators of \sef{eq:11.1}, for the special case $\ze = 0$, we need to verify that
\begin{equation}\label{eq:11.18}
V(t)y = x(t) - y(0).
\end{equation}
By elementary operations one derives from \sef{eq:11.1} the equation
\begin{equation}\label{eq:11.19}
x(t) - y(0) = \int_{[0,t)} d\eta(\si)\bigl[x(t-\si) - y(0)\bigr] + V_0(t)y
\end{equation}
with $V_0(t)y$ as specified in \sef{eq:11.12}. From \sef{eq:11.13}, \sef{eq:9.9} and \sef{eq:11.19} it follows, by uniqueness, that \sef{eq:11.18} holds. This completes the proof in the special case when $\zeta = 0$.

\medskip

For the general case we have to replace \sef{eq:11.11} by
\begin{equation}\label{eq:11.20}
\gb{S_0(t)y}(\th) = z(t+\th),
\end{equation}
with $z$ the solution of \sef{eq:6.15}, as given in \sef{eq:6.17} in terms of the resolvent $\rho$ of $\ze$ and $f$ defined by \sef{eq:6.16}. So \sef{eq:11.12} is replaced by
\begin{equation}\label{eq:11.21}
V_0(t)y = \gb{\eta \star z}(t) + g(t)
\end{equation}
with 
\begin{equation}\label{eq:11.21a}
g(t) = \int_{(t,1]} d\eta(\si)y(t-\si) - \int_{[0,1]} d\eta(\si)y(-\si).
\end{equation}
For $y=q$ we have $f = I$ and hence for $t \ge 0$
\[z(t) = I + \int_0^t \rho(\tau)\,d\tau.\]
Since $g(0) = 0$ for $y=q$ we find
\begin{align}\label{eq:11.22}
K(t) = V_0(t)q &= \int_{[0,t]} d\eta(\si)\bigl[I + \int_0^{t-\si} \rho(\tau)\,d\tau\bigr]\nonumber\\
&= \eta(t) + \int_0^t \eta(\si) \rho(t-\si)\,d\si.
\end{align}
For $t > 0$
\begin{align}\label{eq:11.23}
\dy W_0(t) = \dy S_0(t)q &= \int_{[0,1]} d\dy(\si)\bigl[I + \int_0^{t-\si} \rho(\tau)\,d\tau\bigr]\,\chi_{\si \le t}(\si)\nonumber\\
&= \dy(t) + \int_0^t \dy(\si) \rho(t-\si)\,d\si
\end{align}
(so note that $t \mapsto \dy W_0(t)$ is actually continuous in $t=0$!) and accordingly
\begin{align}\label{eq:11.24}
&\int_0^t \dy W_0(d\tau) \cdot V(t-\tau)y\nonumber\\
&\qquad= \int_0^t \dy(d\tau) \cdot V(t-\tau)y + \int_0^t d_\tau\bigl[\int_0^\tau \dy(\si)\rho(\tau-\si)\,d\si\bigr] \cdot V(t-\tau)y\nonumber\\
&\qquad= \int_0^t \dy(d\tau) \cdot \bigl[ V(t-\tau)y + \int_0^{t-\tau} \rho(\th)V(t-\tau-\th)y\,d\th\bigr]
\end{align}
(where in the last step we have used integration by parts).

So, repeating the argument embodied in \sef{eq:11.16} and \sef{eq:11.17}, we find that the general version of \sef{eq:11.18} reads
\begin{equation}\label{eq:11.25}
V(t)y + \int_0^t \rho(\th)V(t-\th)\,d\th = x(t) - z(t).
\end{equation}

We rewrite \sef{eq:11.1}--\sef{eq:11.2} as
\[x = \eta \star x + \ze \ast x + f +g\]
with $f$ given by \sef{eq:6.16} and $g$ by \sef{eq:11.22}. Subtracting
\[z = \ze \ast z + f\]
we obtain (using \sef{eq:11.21} in the second step)
\begin{align*}
x - z &= \eta \star x + \ze \ast (x-z) + g\\
&= \eta \star (x-z) + \ze \ast (x-z) + V_0.
\end{align*}
Applying $\rho \ast$ to both sides and using \sef{eq:6.13} we find
\[\ze \ast (x-z) = \rho \ast \eta \star (x-z) + \rho \ast V_0\]
and accordingly we can rewrite the equation for $x-z$ in the form
\begin{equation}\label{eq:11.27}
x -z = \eta \star (x-z) + \rho \ast \eta \star (x-z) + V_0 + \rho \ast V_0.
\end{equation}
The general equation \sef{eq:9.9} amounts, when $K$ is given by \sef{eq:11.23}, to
\begin{align*}
V &= \eta \star V + \eta \star \rho \ast V + V_0\\
&= \eta \star (V + \rho \ast V) + V_0.
\end{align*}
So
\[\rho \ast V = \rho \ast \eta \star V + \rho \ast \eta \star \rho \ast V + \rho \ast V_0\]
and
\begin{equation}\label{eq:11.28}
V + \rho \ast V = \eta \star (V + \rho \ast V) + \rho \ast \eta \star (V + \rho \ast V) + V_0 +\rho \ast V_0.
\end{equation}
Comparing \sef{eq:11.27} and \sef{eq:11.28} we deduce from the uniqueness of a solution that \sef{eq:11.25} holds.

We summarize our conclusions as

\begin{theorem}\label{thm:11.1}
The semigroup of solution operators of \sef{eq:11.1}--\sef{eq:11.2}, with the assumption \sef{eq:11.5}, is identical  to the twin semigroup of Theorem $\ref{thm:10.2}$ when the specifications \sef{eq:11.3}, \sef{eq:11.4}, $\sef{eq:11.6}\slash\sef{eq:11.8}$,\sef{eq:11.7},\sef{eq:11.9}, and \sef{eq:11.10} are made.
\end{theorem}

Motivated by Theorem \ref{thm:10.4} we add a result about the asymptotic behaviour for $t \to \infty$.

\begin{theorem}\label{thm:11.2}
Suppose that $\eta$ has no singular part \(see \sef{eq:decomp-meas}\). The semigroup of solution operators of \sef{eq:11.1}--\sef{eq:11.2} restricted to $C\gb{[-1,0];\BR^n}$ is asymptotically stable if the following two conditions are satisfied
\begin{enumerate}
\item[\textup{(i)}] $\det \bigl[ zI - \int_0^1 e^{-z\si} d\ze(\si) \bigr] \not= 0$ for $\Re z \ge 0$;
\item[\textup{(ii)}] $\inf_{\Re z \ge 0}\ \bigl|\det \bigl[ I - \int_0^1 e^{-z\si} d\eta(\si) \bigr] \bigr| > 0$.
\end{enumerate}
\end{theorem}

\begin{proof}
The first condition and Theorem \ref{thm:5.5} imply that the unperturbed semigroup $\{S_0(t)\}$ is bounded. Furthermore, in the present setting $K$ = $\eta$ and the second condition is equivalent to \sef{eq:cond-asymp-behav}. Since in the present setting the norm closure of $\DOM{C}$ equals $C\gb{[-1,0];\BR^n}$, the result follows from an application of Theorem \ref{thm:10.4}.
\end{proof}
 
 \medskip
In contrast to RFDE, general NFDE do not have smoothing properties and it is a delicate question whether the semigroup of solution operators of \sef{eq:11.1}--\sef{eq:11.2} is asymptotically stable under the assumptions of Theorem \ref{thm:11.2} on $B\gb{[-1,0];\BR^n}$. The solutions of NFDE with an absolutely continuous measure $\ze$ do become continuous and Theorem \ref{thm:11.2} can be used to study the asymptotic stability of the semigroup on $B\gb{[-1,0];\BR^n}$. 

As an illustration, consider the following NFDE
 \begin{equation}\label{eq:nfde-example}
 \frac{d}{dt}\bigl[x(t) - \int_0^1 a(s)x(t-s)\,ds\bigr] = -c x(t)
 \end{equation}
 with $c > 0$ and $\int_0^1 |a(s)|\,ds < 1$. 
 
Note that $c > 0$ implies that the first condition in Theorem \ref{thm:11.2} is satisfied and that $\int_0^1 |a(s)|\,ds < 1$ implies that the second condition in Theorem \ref{thm:11.2} is satisfied as well. Therefore, the zero solution of \sef{eq:nfde-example} is asymptotically stable.

\section{RE - Renewal equations with BV kernels}\label{sec:12}
\setcounter{equation}{0}

If a cell divides into two daugther cells after a cell cycle of fixed length, and we take this length as the unit of time, we may replace \sef{eq:7.1} by
\begin{equation}\label{eq:12.1}
b(t) = 2b(t-1)
\end{equation}
when mortality is negligible. More generally we may consider
\begin{equation}\label{eq:12.2}
b(t) = \int_0^1 L(da) b(t-a),
\end{equation}
where the model ingredient $L$ specifies the age-specific expected {\sl cumulative} number of offspring. (The motivation for writing the $L$ factor first is that in the generalization to systems of equations, $L$ is a matrix and $b$ is a vector.) 

The assumption
\begin{equation}\label{eq:12.3}
L \hbox{ is continuous in } a = 0
\end{equation}
reflects that instantaneous reproduction by a newborn individual is impossible. It turns out to be useful to extend the domain of definition of $L$ via
\begin{equation}\label{eq:12.15}
L(a) = 0 \quad \hbox{for } a \le 0\quad\hbox{and}\quad L(a) = L(1)\quad \hbox{for } a \ge 1.
\end{equation}
In terms of $B$ defined by (cf.  \sef{eq:7.16})
\begin{equation}\label{eq:12.4}
B(t) = \int_0^t b(\tau)\,d\tau,\qquad t > 0,
\end{equation}
we can write \sef{eq:12.2} as the neutral delay differential equation
\begin{equation}\label{eq:12.5}
B'(t) = \int_0^1 L(da) B'(t-a)
\end{equation}
Provided $B'$ is of bounded variation and $B_t'$ is normalized to be zero in zero, we can rewrite \sef{eq:12.5} as
\begin{equation}\label{eq:12.6}
B'(t) = \int_{-1}^0 \gb{L(-\si) - L(1)} B_t'(d\si).
\end{equation}
To verify this transformation, we fix $t$ and show that we can interpret \sef{eq:12.5} and \sef{eq:12.6} as convolution of measures. Note that $L(-\cdot) - L(1) \in NBV\gb{\BR}$ and so by Theorem \ref{thm:bv2} there exists a measure $\mu$ such that
\begin{equation}\label{eq:12.6a}
L(a) - L(1) = \mu\gb{(-\infty,a]}.
\end{equation}
Also note that we can normalize $B'$ such that $B'(a) = 0$ for $a \ge t$ and again by Theorem \ref{thm:bv2} there exists a measure $\nu$ such that
\begin{equation}\label{eq:12.6b}
B'(a) = \nu\gb{(-\infty,a]}.
\end{equation}
An application of \sef{eq:convo-meas-ident} now yields
\begin{align*}
B'(t) = \int_0^1 L(da)\,B'(t-a) &= \int_\BR L(da)\, B'(t-a)\\
&= \int_\BR \mu(da)\, \nu\gb{(-\infty,t-a]}\\
&= \int_\BR \mu\gb{(-\infty,t-a]}\,\nu(da)\quad(\mbox{by \sef{eq:convo-meas-ident}})\\
&= \int_\BR \gb{L(t-a) - L(1)}\, B'(da)\\
&= \int_\BR \gb{L(-\si) - L(1)}\, B_t'(d\si)\quad(\mbox{with } t-a = -\si)\\
&= \int_{-1}^0 \gb{L(-\si) - L(1)}\, B_t'(d\si)
\end{align*}
and this shows that the equations \sef{eq:12.5} and \sef{eq:12.6} are equivalent.

Now recall the definition of $Y$, $\DY$ and $C_0$ in \sef{eq:7.12}, \sef{eq:7.13} and \sef{eq:7.20}, respectively. It appears that the right hand side of \sef{eq:12.6} can be written as
\[\pa{\DQ}{C_0B_t}\]
when we define
\begin{equation}\label{eq:12.7}
\DQ(\th) = L(\th) - L(1),\qquad 0 \le \th \le 1.
\end{equation}
Thus we are led to believe that the theory of Section \ref{sec:10} applies to equation \sef{eq:12.5}. The aim of this section is to show that this is indeed the case by elaborating the details.

We supplement \sef{eq:12.5} by the initial condition
\begin{equation}\label{eq:12.8}
B(\th) = \psi(\th),\qquad -1 \le \th \le 0
\end{equation}
with $\psi \in Y$, so in particular $B(0) = \psi(0) = 0$.
Integrating both sides of \sef{eq:12.5} with respect to time from $0$ to $t$, we obtain first
\begin{equation}\label{eq:12.9}
B(t) = \int_0^1 L(da)\bigl[B(t-a) - B(-a)\bigr]
\end{equation}
and next, using \sef{eq:12.8},
\begin{equation}\label{eq:12.10}
B(t) = \int_{[0,t]} L(da)B(t-a) + f(t)
\end{equation}
with
\begin{equation}\label{eq:12.11}
f(t) := \int_{(t,1]} L(da)\psi(t-a) - \int_{[0,1]} L(da)\psi(-a).
\end{equation}
The {\sl resolvent} $R$ of $L$ is the solution of (cf. Theorem \ref{thm:resolvent})
\begin{equation}\label{eq:12.12}
R(a) = \int_{[0,a]} L(a-\si) R(d\si) + L(a)
\end{equation}
which is consistent with $R(0)=0$ and shows that $R$, just like $L$ (recall \sef{eq:12.3}), is continuous from the right in $a=0$. Because of this property of both $R$ and $L$ we have
\[\int_{[0,a]} L(a-\si)R(d\si) = \int_{[0,a]} L(d\si)R(a-\si).\]
We also note that in general, i.e., even for systems, so for functions taking values in $\BR^n$,
\[\int_{[0,a]} L(a-\si)R(d\si) = \int_{[0,a]} R(d\si)L(a-\si)\]
whenever $R$ is the resolvent of $L$, cf. Theorem \ref{thm:resolvent}. According to Theorem \ref{thm:renewal} the solution of \sef{eq:12.10} is given by
\begin{equation}\label{eq:12.13}
B(t) = f(t) + \int_{[0,t]} R(da)f(t-a).
\end{equation}
Starting from the initial condition \sef{eq:12.8} we thus provided a constructive definition of $B(t)$ for $t > 0$. Clearly the definition of the operators $S(t)$ in \sef{eq:7.18} extends to the current situation. We want to identify these operators with the semigroup of Theorem \ref{thm:10.1} when $Y$, $\DY$, $\{S_0(t)\}$, $C_0$, $q$ and $\DQ$ are given by, respectively, \sef{eq:7.12}, \sef{eq:7.13}, \sef{eq:7.19}, \sef{eq:7.20}, \sef{eq:7.22} and \sef{eq:12.7}.

In \sef{eq:9.6} a family of maps $V_0(t) : Y \to \BR$ was defined by
\begin{equation}\label{eq:12.14}
V_0(t)\psi = \DQ\gb{S_0(t)-I}\psi.
\end{equation}
Our first step will be to spell out the right hand side for the current situation.

\begin{lemma}\label{lem:12.1}
Let $V_0$ be defined by \sef{eq:12.14}, then
\begin{equation}\label{eq:12.16}
V_0(t)\psi = \int_{-1}^0 \psi(d\th)\,\bigl[L(t-\th) - L(-\th)\bigr].
\end{equation}
\end{lemma}

\begin{proof}
Since
\begin{align*}
\DQ\gb{S_0(t)-I}\psi &= \int_{-1}^{-t} d_\th \psi(t+\th)\bigl[ L(-\th) - L(1)\bigr]\\
&\qquad\qquad -\int_{-1}^0 \psi(d\th)\bigl[L(-\th) - L(1)\bigr],
\end{align*}
the claim follows from
\begin{align*}
\int_{-1}^{-t} d_\th \psi(t+\th)\bigl[ L(-\th) - L(1)\bigr] &= \int_{t-1}^0 \psi(d\si)\bigl[ L(t-\th) - L(1)\bigr]\\
&= \int_{-1}^0 \psi(d\th)\,\bigl[L(t-\th) - L(1)\bigr]
\end{align*}
(where in the last step we used \sef{eq:12.15}).
\end{proof}

\begin{corollary}\label{col:12.2}
For the kernel $K = V_0(\novar)q$, cf. \sef{eq:9.7}, we find
\begin{equation}\label{eq:12.17}
K(t) = L(t)
\end{equation}
\end{corollary}

\smallskip

\begin{lemma}\label{lem:12.3}
Let $f$ be defined by \sef{eq:12.11}, then
\begin{equation}\label{eq:12.18}
f(t) = V_0(t)\psi.
\end{equation}
\end{lemma}

\begin{proof}
Extending $\psi$ by zero for positive arguments, we can write \sef{eq:12.11} as
\[f(t) = \int_0^1 L(da)\,\bigl[\psi(t-a) - \psi(-a)\bigr]\]
and next use partial integration to obtain
\begin{align*}
f(t) &= L(1)\bigl[\psi(t-1) - \psi(-1)\bigr] + \int_{t-1}^0 L(t-\th)\,\psi(d\th) - \int_{-1}^0 L(-\th)\,\psi(d\th)\\
&= \int_{-1}^0 \bigl[L(t-\th) - L(-\th)\bigr]\,\psi(d\th) = V_0(t)\psi.
\end{align*}
\end{proof}

\begin{corollary}\label{col:12.4}
Let $B$ be defined by \sef{eq:12.13} then, by comparing \sef{eq:12.13} to \sef{eq:9.15}, we find
\begin{equation}\label{eq:12.19}
B(t) = V(t)\psi.
\end{equation}
\end{corollary}

\begin{theorem}\label{thm:12.5}
Let $\{S(t)\}$ be the twin semigroup defined by \sef{eq:9.22}, then
\begin{equation}\label{eq:12.20}
\gb{S(t)\psi}(\th) = B(t+\th) - B(t)
\end{equation}
holds for the special case of $Y$, $\DY$, $\{S_0(t)\}$, $C_0$, $q$ and $\DQ$ considered in this section.
\end{theorem}

\begin{proof}
By pairing with step functions from $\DY$ we deduce from \sef{eq:9.19} the pointwise definition
\begin{equation}\label{eq:12.21}
\gb{S(t)\psi}(\th) = \gb{S_0(t)\psi}(\th) + \int_0^t \gb{S_0(t-\tau)q}(\th) V(d\tau)\psi.
\end{equation}
For $t+\th \le 0$ we have
\[\gb{S_0(t)\psi}(\th) = \psi(t+\th)\]
and $\gb{S_0(t-\tau)q}(\th) = -1$ for $0 \le \tau \le t$. Hence
\[\gb{S(t)\psi}(\th) = \psi(t+\th) - V(t)\psi = \psi(t+\th) - B(t).\]
For $t+\th > 0$ we have $\gb{S_0(t)\psi}(\th) = 0$ and
\[\gb{S_0(t-\tau)q}(\th) = -1\quad \hbox{for } t+\th < \tau \le t\]
(and zero otherwise), showing that \sef{eq:12.21} holds.
\end{proof}

An application of Theorem \ref{thm:10.4}, note that $S_0(t)$ defined by \sef{eq:7.19} is bounded and identically zero for $t \ge 1$, yields the following asymptotic stability result.

\begin{theorem}\label{thm:12.6}
Suppose that $L$ has no singular part \(see \sef{eq:decomp-meas}\). The twin semigroup  $\{S(t)\}$ defined by \sef{eq:12.20} is asymptotically stable if
\begin{equation}\label{eq:12.22}
\inf_{\Re z \ge 0} \bigl| \det\gb{I - \int_0^1 e^{-z\tau}L(d\tau)} \bigr| > 0.
\end{equation}
\end{theorem}

\section{Discussion}\label{sec:13}
\setcounter{equation}{0}

When supplemented by an appropriate initial condition, a delay equation has, as a rule, a unique solution. The proof consists of formulating a fixed point problem and verifying the conditions of the contraction mapping theorem. Next a semigroup of solution operators is defined by translation along the constructed solution.

In pioneering fundamental work \cite{Hale71}, J.K. Hale developed the qualitative theory of delay equations along the lines of the corresponding theory for ODE, but with due attention for the infinite dimensional character of the state space. The variation-of-constants formula is an essential instrument for building such a theory. This formula involves both the right hand side of the equation (corresponding to the derivative of the point value in zero of the function that describes the current state, taking values in $\BR^n$) and integration. If one wants to work with the Riemann-integral, the state space needs to be such that the semigroup is strongly continuous. If one wants that the right hand side of the equation corresponds to a well-defined bounded operator on the state space, this space needs to be such that point evaluation is well defined and that point values are not constrained by values in nearby points. As explained in the introduction, these requirements are incompatible. So a fundamental difficulty arises. (In our opinion, the challenge arising from this difficulty actually gives the theory of delay equations its charm.) 

As far as we know, until now state spaces have been chosen such that one can work with the Riemann integral. In \cite{Hale71} the semigroup is strongly continuous and the difficulty is addressed by introducing the fundamental solution (corresponding to an initial condition that does NOT belong to the state space) and letting the formula define the point values of the function that represents the state. In \cite{Diek95}, first an auxiliary space is introduced. This is in fact a dual space ‘containing’ the fundamental solution. Next one checks that the weak* Riemann integral defines an element of the original state space. In \cite{MR09,MR18} integrated semigroups are used to avoid the need of considering elements that do not belong to the state space.

Here we have chosen to work with a state space $Y$ that is `big' enough to contain the fundamental solution. This has two consequences

\begin{itemize}
\item[i)] we lose strong continuity of the semigroups on $Y$
\item[ii)] the dual space $Y^\ast$ does not allow a characterization that enables to represent the adjoint semigroup in a manner that provides information via formula manipulation.
\end{itemize}

To overcome these difficulties, we have in a first step singled out an explicitly characterized subspace $\DY$ of the dual space $Y^\ast$ that is both rich enough and not too rich. By this we mean that the combination of $Y$ and $\DY$ forms a norming dual pair, i.e., an element of $Y$ is completely determined by the action of the elements of $\DY$ on it and, vice versa, an element of $\DY$ is completely determined by the action of the elements of $Y$ on it. Integrals of functions in either one of these spaces are next defined by integrating (after requiring measurability) the scalar functions obtained by pairing with elements of the other space. This yields elements of, respectively, $Y^\ast$ and $Y^{\diamond\ast}$ and a priori it is not guaranteed that these are represented by elements of, respectively, $\DY$ and $Y$. To verify that actually they are, we equip both spaces with a second topology, the weak topology generated by the other space. Viewed thus as locally convex spaces, one space is the dual of the other and the verification reduces to checking the continuity of linear functionals with respect to the right topology. This is where the dominated convergence theorem and additional assumptions enter the story. In this paper we developed the relevant linear theory and showed that, with appropriate choice of $Y$ and $\DY$, it covers perturbation theory for both delay differential equations and renewal equations, not only in the retarded, but also in the neutral case.

We plan to extend our work in several directions. We are confident that equations with infinite delay can be dealt with in the spirit of \cite{DiekGyll12} and that the proofs in \cite{Diek95} of the Principle of Linearized Stability, the Centre Manifold Theorem et cetera, generalize, mutatis mutandae, to the nonlinear version of the present setting. But this has to be checked, with special attention for the neutral case. 

For Renewal Equations it is not yet entirely clear what exactly qualifies as ‘the nonlinear version of the present setting’. And, on top of that, in population dynamical models with individuals characterized by a multi-dimensional variable (e.g., age and size) that can assume a continuum of birth values, we have to deal with an infinite dimensional Renewal Equation. Ideally, we connect the modelling and bookkeeping approach of \cite{DGMT98,DGHKMT01} to our nonlinear extension.

\medskip
\paragraph{Acknowledgement}
A referee provided detailed constructive feedback, leading to substantial improvement of the manuscript. We are most thankful to this anonymous referee.

\begin{appendices}

\section{Renewal equations and their resolvents}

\renewcommand{\theequation}{\Alph{section}.\arabic{equation}}

In this appendix $\SB$ denotes the Borel $\si$-algebra on $[0,\infty)$. For $E \in \SB$, we call a sequence of disjoint sets $\{E_j\}$ in $\SB$ a partition of $E$ if $\mcup_{j=1}^\infty E_j = E$. A complex bounded Borel measure is a map $\mu : \SB \to \BC$ such that $\mu(\emptyset) = 0$ and 
\[\mu(E) = \sum_{j=1}^\infty \mu(E_j),\]
for every partition $\{E_j\}$ of $E$ with the series converging absolutely. In the following we will often omit the adjective `bounded'. The \textit{total variation} measure $\abs{\mu}$ of a complex Borel measure $\mu$ is given by
\begin{equation}\label{eq:bv3}
\abs{\mu}(E) = \sup\big\{\sum_{j=0}^n \abs{\mu(E_j)} \mid n \in \BN,\ \{E_j\} \hbox{ a partition of } E \hbox{ in } \SB\,\big\}.
\end{equation}
The vector space of complex Borel measures of bounded total variation is denoted by $M\gb{[0,\infty)}$. Provided with the total variation norm given by
\begin{equation}\label{eq:norm-TV2}
\|\mu\|_{TV} = \abs{\mu}\gb{[0,\infty)},
\end{equation}
the vector space $M\gb{[0,\infty)}$ becomes a Banach space. 

If needed or handy, we extend measures on $[0,\infty)$ to measures on $\BR$ by defining them to be zero on $(-\infty,0)$, i.e., we define $\mu(E) := \mu(E \cap [0,\infty))$ for every Borel set $E \subset \BR$.

\medskip

Let $f : [0,\infty) \to \BC$. For a given partition $\{E_j\}$ of $[0,t]$ with $E_j = [t_{j-1},t_j)$ and $0 = t_0 < t_1 < \cdots < t_n = t$. we define $T_f : [0,\infty) \to [0,\infty]$ by
\begin{equation}\label{eq:bv1}
T_f(t) := \sup \sum_{j=1}^n \abs{f(t_j)-f(t_{j-1})},
\end{equation}
where the supremum is taken over $n \in \BN$ and all such partitions of $[0,t]$.
The extended real function $T_f$ is called the \textit{total variation function} of $f$. Note that if $0 \le a < b$, then $T_f(b) - T_f(a) \ge 0$ and hence $T_f$ is an increasing function. 

If $\lim_{t \to \infty} T_f(t)$ is finite, then we call $f$ a function of bounded variation. We denote the space of all such functions by $BV$. The space $NBV([0,\infty))$ of \textit{normalized} functions of bounded variation is defined by
\begin{align*}
NBV([0,\infty) ) = \{f \in BV &\mid  f \hbox{ is continuous from the right on } (0,\infty)\\
&\qquad \hbox{ and } f(0) = 0\,\}.
\end{align*}
Provided with the norm 
\begin{equation}\label{eq:bv2}
\|f\|_{TV} := \lim_{t \to \infty} T_f(t)
\end{equation}
the space $NBV([0,\infty))$ becomes a Banach space. More generally, we define for $-\infty < a < b < \infty$, the vector space $NBV\gb{[a,b]}$ to be the space of functions $f : [a,b] \to \BC$ such that $f(a) = 0$, $f$ is continuous from the right on the open interval $(a,b)$, and whose total variation on $[a,b]$, given by $T_f(b)-T_f(a) = T_f(b)$, is finite. Provided with the norm $
\|f\|_{TV} := T_f(b)$, the space  $NBV\gb{[a,b]}$ becomes a Banach space. We extend the domain of definition of a function of bounded variation by defining $f(t) = 0$ for $t < 0$ if $f \in NBV([0,\infty))$ and $f(t) = 0$ for $t < a$ and $f(t) = f(b)$ for $t > b$ if $f \in NBV\gb{[a,b]}$.

\medskip

The following fundamental result, see \cite[Theorem 3.29]{Fol90} provides the correspondence between functions of bounded variation and complex Borel measures.

\begin{theorem}\label{thm:bv2}
Let $\mu$ be a complex Borel measure on $[0,\infty)$. If $f : [0,\infty) \to \BC$ is defined by $f(0) = 0$ and $f(t) = \mu([0,t])$ for $t > 0$, then $f \in NBV([0,\infty))$. Conversely, if $f \in NBV([0,\infty))$ is given, then there is a unique complex Borel measure $\mu_f$ such that $\mu_f([0,t]) = f(t)$ for $t > 0$. Moreover $\abs{\mu_f} = \mu_{T_f}$.
\end{theorem}

\medskip

Given a function $f \in NBV\gb{[a,b]}$ with corresponding measure $\mu_f$, we define the Lebesgue-Stieltjes integral $\int g\,df$ or $\int g(x)\,f(dx)$ to be $\int g\,d\mu_f$. Thus, a Lebesgue-Stieltjes integral is a special Lebesgue integral and the theory for the Lebesgue integral applies to the Lebesgue-Stieltjes integral. We embed $L^1\gb{[0,\infty)}$ into $M\gb{[0,\infty)}$ by identifying $f \in L^1\gb{[0,\infty)}$ with the measure $\mu$ defined by 
\[\mu(E) = \int_E f(x)\,dx\quad\mbox{or, in short,}\quad \mu(dx) = f(x)dx.\] 

From the Radon-Nikodym theorem it follows that we can split a \mbox{scalar-,} vector-, or matrix-valued Borel measure $\mu$ on $[0,\infty)$  into three parts, the absolutely continuous part, the discrete part, and the singular part:
\begin{equation}\label{eq:decomp-meas}
\mu(dx) = b(x)\,dx + \sum_{k=1}^\infty a_k \de_{x_k}(dx) + \mu_s(dx),
\end{equation}
where $b \in L^1([0,\infty))$ represents the absolutely continuous part of $\mu$, $a_k$ are absolutely summable constants and $\de_{x_k}$ denotes the Dirac measure at $x_k$, and $\mu_s$ denotes the singular part of $\mu$.
 
\medskip

In this appendix we collect some results about the convolution of a measure and a function and the convolution of two measures needed to study renewal equations. For details and further results we refer to \cite{Fol90,GLS90}.

\smallskip

Let $B\gb{[0,\infty)}$ denote the vector space of all bounded, Borel measurable functions $f : [0,\infty) \to \BR$. Provided with the supremum norm (denoted by $\|\cdot\|$), the space $B\gb{[0,\infty)}$ becomes a Banach space. With $B\gb{[a,b]}$ we denote the Banach space of all bounded, Borel measurable functions $f : [a,b] \to \BR$ provided with the supremum norm.

\smallskip

The half-line convolution $\mu \bast f$ of a measure $\mu \in M([0,\infty))$ and a Borel measurable function $f$ is the function
\begin{equation}\label{eq:convo-meas-fun}
(\mu \bast f) (t) = \int_{[0,t]} \mu(ds)f(t-s)
\end{equation}
defined for those values of $t$ for which $[0,t] \ni s \mapsto f(t-s)$ is $|\mu|$-integrable.

\smallskip

The following result can be found in \cite[Theorem 3.6.1(ii)]{GLS90}.

\begin{theorem}\label{thm:convo-meas-Borel-fun}
If $f \in B\gb{[0,\infty)}$ and $\mu \in M\gb{[0,\infty)}$, then the convolution of $f$ and $\mu$ satisfies $\mu \bast f \in B\gb{[0,\infty)}$ and
\[\|\mu \bast f \| \le \|\mu\|_{TV}\|f\|.\]
\end{theorem}

\medskip

The half-line convolution $\mu \ast \nu$ of two measures $\mu,\nu \in M\gb{[0,\infty)}$ is defined as the complex Borel measure that to each Borel set $E \in \SB$ assigns the value
\begin{equation}\label{eq:convo-meas}
(\mu \ast \nu) (E) = \int_{[0,\infty)} \mu(ds)\nu\gb{(E-s)_+},
\end{equation}
where $(E - s)_+ \,:=\, \{e-s \mid e \in E\} \,\mcap\, [0,\infty)$ (cf. \cite[Definition 4.1.1]{GLS90}).

If $\chi_E$ is the characteristic function of the set $E$, then 
\[\nu((E-s)_+) = \int_{[0,\infty)} \chi_E(\si + s)\nu(d\si),\]
since $[0,\infty) \ni \si \mapsto \chi_E(\si + s)$ is the characteristic function of $(E-s)_+$.
It follows from Theorem \ref{thm:convo-meas-Borel-fun} that $s \mapsto \nu(E-s)_+)$ belongs to  $B\gb{[0,\infty)}$ and hence the definition of the convolution of two measures $\mu \ast \nu : \SB \to \BC$ given in \sef{eq:convo-meas} makes sense. Furthermore, using Fubini's Theorem, we have the following useful identity
\begin{align}\label{eq:convo-meas-ident}
\mu \ast \nu (E) &= \int_{[0,\infty)} \mu(ds)\nu\gb{(E-s)_+}\nonumber\\
&= \int_{[0,\infty)}\int_{[0,\infty)} \chi_E(\si+s) \mu(ds)\nu(d\si)\nonumber\\
&= \int_{[0,\infty)} \mu\gb{(E-\si)_+}\nu(d\si)
\end{align}

The following result can be found in \cite[Theorem 4.1.2]{GLS90}.

\begin{theorem}\label{thm:convo-meas-prop}
Let $\mu,\nu \in M\gb{[0,\infty)}$.
\begin{itemize}
\item[(i)] The convolution $\mu \ast \nu$ belongs to $M\gb{[0,\infty)}$ and
\[\|\mu \ast \nu\|_{TV} \le \|\mu\|_{TV}\|\nu\|_{TV}.\]
\item[(ii)] For any bounded Borel function $h \in B\gb{[0,\infty)}$, we have
\[\int_{[0,\infty)} h(t) \gb{\mu \ast \nu}(dt) = \int_{[0,\infty)} \int_{[0,\infty)} h(t+s)\,\mu(dt)\nu(ds).\]
\end{itemize}
\end{theorem}

Let $M_{loc}\gb{[0,\infty)}$ denote the vector space of local measures, i.e., set functions that are defined on relatively compact Borel measurable subsets of $[0,\infty)$ and that locally behave like bounded measures: for every $T > 0$ the set function $\mu_T$ defined by
\[\mu_T\gb{E} := \mu\gb{E \mcap\, [0,T]},\qquad E \in \SB\]
belong to $M\gb{[0,\infty)}$. The elements of $M_{loc}\gb{[0,\infty)}$ are called Radon measures. Since the restriction to $[0,T]$ of $\mu \ast \nu$ depends only on the restrictions of $\mu$ and $\nu$ to $[0,T]$, we can unambiguously extend the convolution product to $M_{loc}\gb{[0,\infty)}$.

\smallskip

The following corollary to Theorem \ref{thm:convo-meas-prop} can be found in \cite[Corollary 4.1.4]{GLS90}.

\begin{corollary}\label{col:convo-meas-prop}
Let $\mu,\nu,\rho \in M_{loc}\gb{[0,\infty)}$.
\begin{itemize}
\item[(i)] The convolution $\mu \ast \nu$ belongs to $M_{loc}\gb{[0,\infty)}$ and for any $T > 0$
\[\|\mu_T \ast \nu_T\|_{TV} \le \|\mu_T\|_{TV}\|\nu_T\|_{TV}.\]
\item[(ii)] For any locally bounded Borel function $h \in B\gb{[0,\infty)}$, we have
\[\gb{(\mu \ast \nu) \bast h}(t) = \gb{\mu \bast (\nu \bast h)}(t).\]
\item[(iii)] $(\mu \ast \nu) \ast \rho = \mu \ast (\nu \ast \rho)$.
\end{itemize}
\end{corollary}

\medskip

Using the one-to-one correspondence between complex Borel measures and functions of bounded variation, see Theorem \ref{thm:bv2}, we can combine the above results to obtain the following theorem.

\begin{theorem}\label{thm:convo-meas-fun2}
If $f \in NBV([0,\infty))$ and $\mu \in M([0,\infty))$, then the convolution of $\mu$ and $f$ satisfies $\mu \bast f \in NBV([0,\infty))$ and 
\[\|\mu \bast f\|_{TV} \le \|\mu\|_{TV} \|f\|_{TV}.\]
\end{theorem}

\begin{proof}
If $\nu$ is the unique complex Borel measure such that $f(t) = \nu\gb{[0,t]})$ for every $t \in [0,\infty)$, then
with $E = [0,t]$
\begin{align}\label{eq:convo-meas-fun2}
\mu \bast f(t) &= \int_{[0,t]} \mu(ds) f(t-s)\nonumber\\
&= \int_{[0,t]} \mu(ds)\nu\gb{[0,t-s]}\nonumber\\
&= \int_{[0,\infty)} \mu(ds)\nu\gb{(E-s)_+}\nonumber\\
&= \gb{\mu \ast \nu}(E),
\end{align}
where we have used \sef{eq:convo-meas-ident}. Since $\mu \ast \nu \in M\gb{[0,\infty)}$, we can use \sef{eq:convo-meas-fun2} to define $g : [0,\infty) \to \BR$ by
\begin{equation}\label{eq:convo-meas-fun2c}
g(t) = \mu \bast f(t) = \mu \ast \nu\gb{[0,t]}.
\end{equation}
According to Theorem \ref{thm:bv2}, the function $g$ belongs to $NBV([0,\infty))$. Finally, the norm estimate follows from Theorem \ref{thm:convo-meas-prop}(i).
\end{proof}

We also need the following result.

\begin{theorem}\label{thm:basic-est}
Let $\mu \in M\gb{[0,\infty)}$ and let $f : [0,\infty) \to \BC$ be a bounded continuous function.
\begin{itemize}
\item[(i)] If $f(0) = 0$, then $\mu \bast f$ is a bounded continuous function and
\[\|\mu \bast f \| \le \|\mu\|_{TV}\|f\|.\]
\item[(ii)] If $\mu$ has no discrete part, then $\mu \bast f$ is a bounded continuous function and
\[\|\mu \bast f \| \le \|\mu\|_{TV}\|f\|.\]
\end{itemize}
\end{theorem}

\begin{proof}
To prove (i), observe first that if $f(0) = 0$, then we can extend $f$ to a continuous function on $\BR$ by defining $f(t) = 0$ for $t < 0$. From \sef{eq:convo-meas-fun} we obtain
\begin{align*}
\big| \gb{\mu \bast f}(t+h) - \gb{\mu \bast f}(t) \big| 
&\le \int_{[0,\max\{t,t+h\}]} \abs{\mu}(ds)\, \babs{f(t+h-s) - f(t-s)}\\
&\le \|\mu\|_{TV} \sup_{0 \le \si \le \max\{t,t+h\}} \babs{f(\si) - f(\si-h)}.
\end{align*}
Since $f$ is continuous, for any $t \ge 0$ the right hand side converges to zero as $h \to 0$, showing that $\mu \bast f$ is continuous.

To prove (ii), we first write
\begin{align}\label{eq:basic-est}
\gb{\mu \bast f}(t) &= \int_{[0,t]} \mu(ds)f(t-s)\nonumber\\
&= \int_{[0,t]} \mu(ds)\gb{f(t-s) - f(0)} + \mu\gb{[0,t]}f(0).
\end{align}
If $g(s) = f(s) - f(0)$, then $g(0) = 0$ and
\[\int_{[0,t]} \mu(ds)\gb{f(t-s) - f(0)} = \int_{[0,t]} \mu(ds)g(t-s)\]
and by the first part it follows that this term is continuous. Since $\mu$ has no discrete part, the function $t \mapsto \mu\gb{[0,t]}f(0)$ is also continuous. This shows that $\mu \bast f$ is a bounded continuous function and the norm estimate follows from the corresponding estimate given in Theorem \ref{thm:convo-meas-Borel-fun}.
\end{proof}

\medskip

Let $\ga$ be a real number. For $\mu \in M_{loc}\gb{[0,\infty);\BC^{n \times n}}$ we define the local measure $\mu^\ga  \in M_{loc}\gb{[0,\infty);\BC^{n \times n}}$ by
\begin{equation}\label{def:scaled-meas}
\mu^\ga(E) = \int_{[0,T]} \chi_E(s)e^{-\ga s}\,\mu(ds),
\end{equation}
for $T$ large enough to guarantee that $E \subset [0,T]$ and where $\chi_E$ denotes the characteristic function of $E$. 

We continue with the existence of the resolvent $\rho$ of a complex Borel measure $\mu$ supported on $[0,\infty)$. See \cite[Theorem 4.1.5]{GLS90}.

\begin{theorem}\label{thm:resolvent}
Suppose that $\mu \in M_{loc}\gb{[0,\infty), \BC^{n \times n}}$. There exists a unique measure $\rho \in M_{loc}\gb{[0,\infty), \BC^{n \times n}}$ satisfying either one (and hence both) of the following identities
\begin{equation}\label{eq:resol-eqn-rho}
\rho - \mu \ast \rho = \mu = \rho - \rho \ast \mu 
\end{equation}
if and only if $\det\bigl[I - \mu(\{0\})\bigr] \not= 0$.

Furthermore, if there exists a positive real $\ga$ such that the measure $\mu^\ga$ is a bounded Borel measure, then there exists $\al$ with $\al \ge \ga$ such that $\rho^\al$ is a bounded Borel measure. Here $\mu^\ga$ and $\rho^\al$ are defined as in \sef{def:scaled-meas}.
\end{theorem}

\begin{proof}
Suppose that there exists a measure $\rho$ such that $\rho - \mu \ast \rho = \mu$, then 
$\gb{\de_0 - \mu} \ast \gb{\de_0 + \rho} = \de_0$, where $\de_0$ denotes the Dirac measure with as its value the identity matrix at zero.
Therefore,
\[\bigl[I - \mu(\{0\})\bigr] \bigl[I+\rho(\{0\})\bigr] = I\]
and hence $\det\bigl[I - \mu(\{0\})\bigr] \not= 0$.

Next assume that $\det\bigl[I - \mu(\{0\})\bigr] \not= 0$. We first show that if $\rho$ exists such that \sef{eq:resol-eqn-rho} holds, then it is unique. Indeed, if there exist $\what\rho$ such that $\what\rho - \mu \ast \what\rho = \mu$ , then
\begin{align*}
\rho = \mu + \rho \ast \mu &= \mu + \rho \ast \gb{ \what\rho - \mu \ast \what\rho }\\
&= \mu + \rho \ast \what\rho - \gb{\rho \ast \mu}*\what\rho\\
&= \mu + \gb{\rho - \rho \ast \mu}\ast\what\rho\\
&= \mu + \mu \ast \what\rho = \what\rho.
\end{align*}

Because of the uniqueness of the solution $\rho \in M_{loc}\gb{[0,\infty), \BC^{n \times n}}$, it suffices to show that for each $T \in (0,\infty)$ there is a measure $\hat\rho \in M\gb{[0,T]}$ satisfying the resolvent equation on $[0,T]$:
\begin{equation}\label{eq:resol-eqn-rho-T}
\gb{\hat\rho - \mu \ast \hat\rho}_T = \mu_T = \gb{\hat\rho - \hat\rho \ast \mu}_T.
\end{equation}

Furthermore, if $\ga \in \BR$ and $\mu \in M_{loc}\gb{[0,\infty), \BC^{n \times n}}$ and $\nu \in M\gb{[0,T]}$ satisfies
\begin{equation}\label{eq:resol-eqn-rho-scaled}
\nu - \mu^\ga \ast \nu = \mu^\ga,
\end{equation}
then $\rho_T = \nu^{-\ga}$ satisfies \sef{eq:resol-eqn-rho-T}. Indeed
\begin{align*}
\mu^\ga \ast \nu\gb{[0,t]} &= \int_{[0,\infty)} e^{-\ga s}\mu(ds)\,\nu\gb{[0,t-s]}\\
&= e^{-\ga t}\int_{[0,\infty)}\mu(ds)\,\nu^{-\ga}\gb{[0,t-s]}\\
&= e^{-\ga t} \gb{\mu \ast \nu^{-\ga}}\gb{[0,t]}.
\end{align*}

\medskip

Fix $T > 0$ and assume at first that $\mu(\{0\}) = 0$. By replacing $\mu$ by $\mu^\ga$ with $\ga$ chosen appropriately, we can assume without loss of generality that
\begin{equation}\label{eq:contraction-estimate}
\babs{\mu}\gb{[0,T]} < 1.
\end{equation}
Using this fact, we have that the map 
\[\rho \mapsto \mu + \mu \ast \rho\]
defines a contraction on the Banach space $M\gb{[0,T]}$ for every $T > 0$. The Banach contraction principle implies that the restriction to $[0,T]$ of the solution $\rho$ of \sef{eq:resol-eqn-rho} is the unique fixed point of this map. Furthermore using the iteration method to approximate the fixed point, we have the following representation for $\rho$
\begin{equation}\label{eq:rep-resol-rho}
\rho = \sum_{j=1}^\infty \mu^{\ast j},
\end{equation}
where $\mu^{\ast j}$ denotes the $j$-times convolution of $\mu$ with itself.

Next assume that A=$\mu(\{0\}) \not= 0$. It follows from $ \det \bigl[I -A\bigr] \not = 0$ that we can rewrite the resolvent equation
\[\rho - \mu \ast \rho = \mu\]
as
\begin{equation}\label{eq:resol-eqn-rho2}
\rho = A(I-A)^{-1}\de + \nu + \nu \ast \rho,
\end{equation}
where
\begin{equation}\label{eq:loc-nu}
\nu = (I-A)^{-1}\gb{\mu - A\de}
\end{equation}
satisfies $\nu(\{0\}) \not= 0$. Also note from \sef{eq:resol-eqn-rho2} that $\rho\gb{\{0\}} = A(I-A)^{-1}$. Therefore it follows from representation \sef{eq:rep-resol-rho} with $\mu = \nu$ that in case  $A = \mu\gb{\{0\}} \not= 0$, we have the following representation for $\rho$
\begin{equation}\label{eq:rep-resol-rho3}
\rho = A(I-A)^{-1}\de + \sum_{j=1}^\infty \nu^{\ast j},
\end{equation}
where $\nu$ is given by \sef{eq:loc-nu}. This completes the proof of the first part of the theorem.

Finally, we prove the exponential estimate for the resolvent by modifying the above contraction argument.  If there exists a positive real $\ga$ such that the measure $\mu^\ga$ is a bounded Borel measure, then we can modify  \sef{eq:contraction-estimate} and replace $\mu$ by $\mu^\al$ with $\al \ge \ga$ chosen such that
\begin{equation}\label{eq:contraction-estimate-large}
\babs{\mu^\al}\gb{[0,\infty)} < 1.
\end{equation}
so that the map $\nu \mapsto \mu^\al + \mu^\al \ast \nu$ is a contraction in $M\gb{[0,\infty)}$. This proves that $\nu \in M\gb{[0,\infty)}$. Since $\rho^\al = \nu$ the proof of the theorem is complete.
\end{proof}

\medskip

To give the precise asymptotic behaviour of the resolvent $\rho$, i.e., the case that $\al = \ga$ in Theorem \ref{thm:resolvent}, we have to impose additional conditions on $\mu$, see Theorem \ref{thm:Gelfand}. We first need some preparations.

\medskip

The Laplace transform $\what\mu : \BC \to \BC^{n \times n}$ of a matrix-valued Borel measure $\mu$ on $[0,\infty)$ is given by 
\begin{equation}\label{eq:Laplace-m}
\what\mu(\la) = \int_{[0,\infty)} e^{-\la t}\,\mu(dt)
\end{equation}
and defined for those values of $\la \in \BC$ for which the integral converges absolutely.

The Laplace transform $\bar f : \BC \to \BC^{n \times n}$ of a vector-valued Borel function $f : [0,\infty) \to \BC^n$ is given by
\begin{equation}\label{eq:Laplace-f}
\bar f(\la) = \int_{[0,\infty)} e^{-\la t}f(t)\,dt
\end{equation}
and defined for those values of $\la \in \BC$ for which the integral converges absolutely.

If $\mu \in M_{loc}\gb{[0,\infty),\BC^{n \times n}}$ and $\what\mu(\la_0)$ exists for some $\la_0 \in \BC$, then $\what\mu(\la)$ is defined in the closed half plane $\Re \la \ge \Re \la_0$. Furthermore, if $f \in B\gb{[0,\infty),\BC^n}$, then
\[\gb{\clo{\mu \bast f}}\, (\la) = \what\mu(\la)\bar f(\la)\]
for all $\la \in \BC$ for which both $\what\mu(\la)$ and $\bar f(\la)$ are defined.

The following result, the so-called half-line Gel'fand theorem (see \cite[Theorem 4.4.3 and Corollary 4.4.7]{GLS90}), gives a precise estimate for the growth of the resolvent of $\mu$.

\begin{theorem}\label{thm:Gelfand} 
Suppose $\mu \in M_{loc}\gb{[0,\infty), \BC^{n \times n}}$ has no singular part and is such that $\mu^\ga$ is a bounded Borel measure. Let $\rho \in M_{loc}\gb{[0,\infty), \BC^{n \times n}}$ denote the unique solution of \sef{eq:resol-eqn-rho}. If
\begin{equation}\label{eq:Gelfand1}
\det \gb{I - \what\mu(z)} \not= 0\qquad\hbox{for } \Re z \ge \ga
\end{equation}
and
\begin{equation}\label{eq:Gelfand2}
\inf_{\Re z \ge \ga} \Big| \det\gb{I - \what\mu_d(z)} \Big| > 0,
\end{equation}
or combined in one condition
\begin{equation}\label{eq:Gelfand3}
\inf_{\Re z \ge \ga} \Big| \det\gb{I - \what\mu(z)} \Big| > 0,
\end{equation}
then $\rho^\ga$ is a bounded Borel measure.
\end{theorem}

\medskip

Let $NBV_{loc}\gb{[0,\infty);\BC^{n}}$ denote the vector space of complex Borel functions $f : [0,\infty) \to \BC^{n}$ such that for every $T > 0$ the function $f_T : [0,\infty) \to \BC^{n}$ defined by
\[f_T(t) := 
\begin{cases}
f(t), &\mbox{when}\quad  0\le t \le T;\\
f(T), &\mbox{when}\quad  t \ge T.
\end{cases}
\]
belongs to $NBV\gb{[0,\infty)}$.

\medskip

We conclude this appendix summarizing the results developed in this section when applied to the renewal equation
\begin{equation}\label{eq:renewal1}
x(t) = \int_{[0,t]} \mu(ds)x(t-s) + f(t),\quad\hbox{for } t \ge 0,
\end{equation}
for various classes of forcing functions $f$.

The following theorem summarizes some relevant results \cite[Theorem 4.1.7]{GLS90}.

\begin{theorem}\label{thm:renewal}
Let $\mu \in M_{loc}\gb{[0,\infty),\BC^{n \times n}}$ with $\det\bigl[I - \mu(\{0\}\bigr] \not= 0$.
\begin{itemize}
\item[(i)] For every $f \in B_{loc}\gb{[0,\infty),\BC^n}$, the renewal equation \sef{eq:renewal1} has a unique solution  $x \in B_{loc}\gb{[0,\infty),\BC^n}$ given by
\[x = f + \rho \bast f,\]
where $\rho$ satisfies \sef{eq:resol-eqn-rho} and is given by \sef{eq:rep-resol-rho}. Furthermore, if $f$ is locally absolutely continuous, then the solution $x$ is locally absolutely continuous as well.
\item[(ii)] If $f \in NBV_{loc}\gb{[0,\infty),\BC^n}$,  then $x \in NBV_{loc}\gb{[0,\infty),\BC^n}$.
\item[(iii)] If $f \in C\gb{[0,\infty),\BC^n}$ and $f(0) = 0$, then $x \in C\gb{[0,\infty), \BC^n}$.
\item[(iv)] If the kernel $\mu$ has no discrete part and if $f \in C\gb{[0,\infty),\BC^n}$, then $x \in C\gb{[0,\infty),\BC^n}$.
\end{itemize}
\end{theorem}

\begin{proof}
Standard arguments show that the solution of the renewal equation \sef{eq:renewal1} is given by $x = f + \rho \bast f$, where $\rho$ denotes the resolvent of $\mu$ given by 
Theorem \ref{thm:resolvent}. So (i) follows from Theorem \ref{thm:convo-meas-Borel-fun}. To prove (ii), first note that it follows from Theorem \ref{thm:convo-meas-fun2} that $x$ is locally of bounded variation. If $f$ is locally absolutely continuous, then $f$ is the integral of a locally $L^1$-function. Using the representation $x = f + \rho \bast f$ and Fubini's Theorem, we derive that $x$ is the integral of a locally $L^1$-function as well. Therefore it follows that $x$ is locally absolutely continuous. Furthermore, (iii) follows from Theorem \ref{thm:basic-est} (i). Finally, if $\mu, \nu \in M_{loc}\gb{[0,\infty)}$, then the discrete part of $\mu \ast \nu$ is given by the sum
\begin{equation}\label{eq:discrete-part}
\gb{\mu \ast \nu}_d = \sum_{k=1}^\infty \sum_{l=1}^\infty p_kq_l\de_{t_k+t_l}.
\end{equation}
In particular, we conclude that if either $\mu$ or $\nu$ has no discrete part, then the convolution $\mu \ast \nu$ also has no discrete part. In particular, if $\mu$ has no discrete part, then it follows from \sef{eq:resol-eqn-rho} that the resolvent $\rho$ has no discrete part. Thus (iv) follows from Theorem \ref{thm:basic-est} (ii)
\end{proof}

\medskip

If the measure $\mu$ has no singular part, then an application of Theorem \ref{thm:Gelfand} yields the following corollary.

\begin{corollary} \label{col:Gelfand} 
Suppose that $\mu \in M\gb{[0,\infty);\BC^{n \times n}}$ has no singular part and satisfies
\begin{equation}\label{eq:Gelfand2a}
\inf_{\Re z \ge 0} \Big| \det\gb{I - \hat\mu(z)} \Big| > 0.
\end{equation}
\begin{itemize}
\item[(i)] For every $f \in B\gb{[0,\infty),\BC^n}$, the renewal equation \sef{eq:renewal1} has a unique solution  $x \in B\gb{[0,\infty),\BC^n}$ given by
\[x = f + \rho \bast f,\]
where $\rho$ satisfies \sef{eq:resol-eqn-rho}. Furthermore, if $f$ is absolutely continuous, then the solution $x$ is absolutely continuous as well.
\item[(ii)] If $f \in NBV\gb{[0,\infty),\BC^n}$,  then $x \in NBV\gb{[0,\infty),\BC^n}$. 
\end{itemize}
\end{corollary}

\section{The norming dual pair $(B,NBV)$}
\setcounter{equation}{0}

In the study of delay differential equations, the natural dual pair is given by
\begin{equation}\label{eq:mot-ex0}
Y = B\gb{[-1,0],\BR^n}\quad\mbox{and}\quad\DY = NBV\gb{[0,1],\BR^n}
\end{equation}
with the pairing
\begin{equation}\label{eq:mot-ex1}
\pa{\dy}{y} = \int_{[0,1]} \dy(d\si) \cdot y(-\si).
\end{equation}
Here $Y$ is provided with the supremum norm and $\DY$ with the total variation norm (see \sef{eq:bv2}).

In the study of renewal equations, the natural dual pair is given by 
\begin{equation}\label{eq:mot-ex0r}
Y = NBV\gb{[-1,0],\BR^n}\quad\mbox{and}\quad\DY = B\gb{[0,1],\BR^n}
\end{equation}
with the pairing
\begin{equation}\label{eq:mot-ex1r}
\pa{\dy}{y} = \int_{[-1,0]} y(d\si) \cdot \dy(-\si).
\end{equation}

Returning to \sef{eq:mot-ex0}--\sef{eq:mot-ex1}, we first make two trivial, yet useful, observations: fix $1 \le i \le n$ and $-1 \le \th \le 0$,
\medskip
\begin{equation}
 \int_{[0,1]} \dy(d\si) \cdot y(-\si) = y_i(\th),
 \end{equation}
if $\dy_j(\si) = 0$, $0 \le \si \le 1$, $j \not= i$, and $\dy_i(\si) = 0$ for $0 \le \si < -\th$ and $\dy_i(\si) = 1$ for $\si \ge -\th$, and similarly
 \begin{equation}
 \int_{[0,1]} \dy(d\si) \cdot y(-\si) = \dy_i(-\th),
 \end{equation}
if $y_j(-\si) = 0$, $0 \le \si \le 1$, $j \not= i$, and $y_i(-\si) = 1$ for $0 \le \si \le -\th$ and $y_i(-\si) = 0$ for $\si > -\th$. The point is that, consequently, in case of \sef{eq:mot-ex0}--\sef{eq:mot-ex1}, convergence in both $\gb{Y,\si(Y,\DY)}$ and $\gb{\DY,\si(\DY,Y)}$ entails pointwise convergence (in, respectively, $B\gb{[-1,0],\BR^n}$ and $NBV\gb{[0,1],\BR^n}$). 

In the first case, the dominated convergence theorem implies that, conversely, a bounded pointwise convergent sequence in $B\gb{[-1,0],\BR^n}$ converges in $\gb{Y,\si(Y,\DY)}$. For $NBV\gb{[0,1],\BR^n}$, this is not so clear. It is true that the pointwise limit of a sequence of functions of bounded variation is again of bounded variation (Helly's theorem), but there is no dominated convergence theorem for measures.

\medskip

The purpose of this appendix is to show that the dual pairs given, respectively, by \sef{eq:mot-ex0} and \sef{eq:mot-ex1} and by \sef{eq:mot-ex0r} and \sef{eq:mot-ex1r} are norming dual pairs suitable for twin perturbation, cf. Definition \ref{def:4.3}.

\begin{theorem}\label{thm:suit-twin-perI}
The dual pair given by \sef{eq:mot-ex0} and \sef{eq:mot-ex1} is a norming dual pair, i.e.,
\begin{align*}
\|y\| &= \sup \Bigl\{|\pa{\dy}{y}| \mid \dy \in \DY,\ \|\dy\| \le 1\,\Bigr\}\\
\|\dy\| &= \sup \Bigl\{|\pa{\dy}{y}| \mid y \in Y,\ \|y\| \le 1\,\Bigr\}
\end{align*}
such that \sef{eq:4.1} and \sef{eq:4.9} are satisfied, i.e.,
\begin{itemize}
\item[(i)] $(Y,\si(Y,\DY))$ is sequentially complete;
\item[(ii)] a linear map $(Y,\si(Y,\DY)) \to \BR$ is continuous if it is sequentially continuous.
\end{itemize}
\end{theorem}

Before we can prove the theorem we need to present some notions from the theory of Riesz spaces.

A Riesz space $Y$ is a real vector space equipped with a lattice structure, i.e., a partial ordering compatible with the vector space structure such that each pair of vectors $x,y \in Y$ has a supremum or least upper bound denoted by $\sup\{x,y\} \in Y$. For a given vector $y$ in a Riesz space, the absolute value $|y| \in Y$ is defined by $|y| = \sup\{y, -y\}$.

The Banach spaces  $Y = B\gb{[-1,0],\BR^n}$ and $Y = NBV\gb{[-1,0],\BR^n}$ are Riesz Banach spaces when the ordering is defined pointwise and componentwise, i.e.,  $f \le g$ whenever 
\[P_jf(\th) \le P_jg(\th)\quad\mbox{for each } \th \in [-1,0] \mbox{ and } 1 \le j \le n,\]
where $P_j : \BR^n \to \BR$ denotes the projection onto the $j^{th}$-coordinate of a $n$-vector. The corresponding absolute value function $|f| : [-1,0] \to \BR^n$ is defined componentwise by
\[P_j |f|(\th) := \sup \{f_j(\th), -f_j(\th)\}\quad\hbox{for } \th \in [-1,0]\hbox{ and } 1 \le j \le n.\]

A sequence $\{f_n\}$ in a Riesz space $Y$ is \textit{order bounded} from above if there is a $g \in Y$ such that $f_n \le g$. A sequence $\{f_n\}$ is called \textit{decreasing to zero} if $\inf_{n\ge 1} \{f_n\} = 0$ and $n \ge m$ implies $0 \le f_n \le f_m$. Furthermore, a sequence $\{f_n\}$ in a Riesz space $Y$ \textit{converges in order} to $f \in Y$ if there is a sequence $\{g_n\}$ in $Y$ that is decreasing to zero and such that
\begin{equation}\label{eq:order-convergence-def}
|f - f_n| \le g_n,\qquad\mbox{for all } n \ge 0.
\end{equation}

A linear functional $\La : Y \to \BR$ on a Riesz space $Y$ is \textit{$\si$-order continuous} if $\La(f_n) \to 0$ in $\BR$ for every sequence $\{f_n\}$ in $Y$ that converges to zero in order. The vector space of all $\si$-order continuous linear functionals is called the $\si$-order continuous dual of $Y$, cf. \cite[Definition 8.26]{AliBor06}.

The following result \cite[Theorem 14.5]{AliBor06} is an essential ingredient of the proof of Theorem \ref{thm:suit-twin-perI}.

\begin{theorem} \label{thm:sigma-order-dual}
The $\si$--order continuous dual of $B\gb{[-1,0],\BR^n}$ is represented by $NBV\gb{[0,1],\BR^n}$.
\end{theorem}

\begin{proof}
In the proof we use the fact that the norm dual of a Riesz Banach space is again a Riesz Banach space (cf, \cite[Theorem 9.27 and Theorem 14.2]{AliBor06}). So, in particular, if $\La$ is a bounded linear functional on  $B\gb{[-1,0],\BR^n}$, then it has an absolute value $\abs{\La}$ in the norm dual of $B\gb{[-1,0],\BR^n}$. Let ${\bf 1} \in B\gb{[-1,0],\BR^n}$ denote the function which is constant one in all components. Since the unit ball in $B\gb{[-1,0],\BR^n}$ coincides with the order interval $[-{\bf 1}, {\bf 1}]$, i.e.,
\[[-{\bf 1}, {\bf 1}] = \bigl\{ f \in B\gb{[-1,0],\BR^n} \mid -{\bf 1} \le f \le {\bf 1} \bigr\},\]
we have that if $\La$ is a bounded linear functional on  $B\gb{[-1,0],\BR^n}$, then
\begin{equation}\label{eq:riesz-norm1}
\|\La\| = \bigl \| |\La| \bigr\| = \sup_{f \in [-{\bf 1}, {\bf 1}]} \bigl| |\La|(f)\bigr| = |\La|({\bf 1}).
\end{equation}

Furthermore an order bounded sequence $f_n$ in $B\gb{[-1,0],\BR^n}$ converges in order to $f$ if and only if 
\begin{equation}\label{eq:order-convergence1}
f_n(x) \to f(x),\qquad\mbox{for all } x \in [-1,0].
\end{equation}
Indeed if for some $\ep > 0$ and $x \in [-1,0]$ we have that $|f(x) - f_n(x)| > \ep$, then $g_n \ge \ep \chi_{\{x\}}$, but $g_n$ is a sequence decreasing to zero and this is a contradiction.

\smallskip\noindent
Step 1. We first show that if $\La$ is a bounded linear functional on  $B\gb{[-1,0],\BR^n}$, then the set function $\mu_\La$ defined by
\begin{equation}\label{eq:def-meas1}
\mu_\La(A) = \La(\chi_A)\quad\hbox{ for any Borel set } A
\end{equation}
is a finitely additive signed measure of bounded variation. 

Indeed from the linearity of $\La$ it is clear that $\mu_\La$ is a finitely additive real-valued set function. To see that $\mu_\La$ is of bounded variation, let $\{E_1,\ldots,E_n\}$ be a partition of $[-1,0]$, then it follows from \sef{eq:riesz-norm1}
\begin{align*}
\sum_{i=1}^n |\mu_\La(E_i)| &= \sum_{i=1}^n \babs{\La(\chi_{E_i})} \le \sum_{i=1}^n |\La| (\chi_{E_i})\\
&= \babs{\La}\gb{\sum_{i=1}^n \chi_{E_i}} = \babs{\La}({\bf 1}) = \|\La\|
\end{align*}
which implies that $\mu_\La$ is of bounded variation. 

As a side remark we mention that the norm dual of  $B\gb{[-1,0],\BR^n}$ is actually represented by the Riesz Banach space of all finitely additive signed measures of bounded variation (cf. \cite[Theorem 14.4]{AliBor06}).

\smallskip\noindent
Step 2.
We next show that $\mu_\La$ is a Borel measure if and only if $\La$ is a $\si$-order continuous linear functional. Assume first that $\La$ is $\si$-order continuous and let $\{E_i\}$ be a pairwise disjoint sequence of Borel measurable sets. Put 
\[E= \Mcup_{i=1}^\infty E_i\quad\hbox{and}\quad F_n = \Mcup_{i=1}^n E_i\] 
and note from \sef{eq:order-convergence1} that $\chi_{F_n}$ converges in order to $\chi_E$. Since $\La$ is $\si$-order continuous, it follows that
\[\sum_{i=1}^n \mu_{\La}(E_i) = \La(\chi_{F_n}) \to \La(\chi_E) = \mu_{\La}(E),\]
which shows that $\mu_{\La}$ is $\si$-additive.

Conversely, assume that $\mu_\La$ is a complex Borel measure. Let $f_n$ be a sequence that converges to zero in order in $B\gb{[-1,0],\BR^n}$ . This implies that $f_n$ is order bounded and it follows from \sef{eq:order-convergence1} that  $f_n \to 0$ pointwise. Thus the Lebesgue dominated convergence theorem implies that
\[\La(f_n) = \int_{[-1,0]} f_n\,d\mu_\La \to 0,\]
proving that $\La$ is $\si$-order continuous.
\end{proof}

\medskip\noindent
{\sl Proof of Theorem $\ref{thm:suit-twin-perI}$.} 
The proof consists of three parts.

\smallskip
\textsc{Part I}.
In this part we prove that $(Y,\DY)$ is a norming dual pair. From Theorem \ref{thm:convo-meas-Borel-fun} that it follows that for every $\dy \in \DY$ and $y \in Y$
\begin{equation}\label{eq:mot-ex2}
\babs{ \int_{[0,1]} \dy(d\si) \cdot y(-\si) } \le \|\dy\|\,\|y\|.
\end{equation}
By considering step functions for $\dy$, i.e., Dirac point measures by Theorem \ref{thm:bv2}, we obtain
\begin{equation}\label{eq:mot-ex5}
\|y\| = \sup_{x \in [-1,0]} |y(x)| = \sup\bigl\{\abs{\pa{\dy}{y}} \mid \dy \in \DY,\ \|\dy\| \le 1 \bigr\}.
\end{equation}
On the other hand, fix $\dy \in \DY$ and let $\mu = \mu_{\dy}$ be the corresponding Borel measure according to Theorem \ref{thm:bv2}.

If $\SP = \{E_j\}_{j=1}^n$ is a partition of $[-1,0]$ into finitely many, pairwise disjoint, measurable sets $E_j$, then
\begin{equation}\label{eq:mot-ex6}
y_\SP = \sum_{j=1}^n {\rm sgn}\, \mu(-E_j)\chi_{E_j}
\end{equation}
is a bounded Borel function on $[-1,0]$ with norm $\| y_\SP \| \le 1$. Furthermore,
\begin{equation}\label{eq:mot-ex7}
\pa{\dy}{y_\SP} = \sum_{j = 1}^n \abs{\mu(-E_j)},
\end{equation}
and taking the supremum over all such finite partitions $\SP$ of $[-1,0]$ we arrive at
\begin{equation}\label{eq:mot-ex8}
\|\dy\| = \sup\bigl\{\abs{\pa{\dy}{y_\SP}} \mid \SP \hbox{ a finite partition of } [-1,0] \bigr\}.
\end{equation}
This shows that the pair $(Y,\DY)$ is a norming dual pair.

\smallskip
\textsc{Part II}.
In this part we prove that $(Y,\si(Y,\DY))$ is sequentially complete. Let $\{y_n\}$ be a Cauchy sequence in $(Y,\si(Y,\DY))$. Since step functions belong to $\DY$ it follows that $\{y_n(x)\}$ is, for every $x \in [-1,0]$, a Cauchy sequence in $\BR$. Since $\BR$ is complete, we have that
\[\lim_{n \to \infty} y_n(x) \ \hbox{ exists pointwise for } x \in [-1,0].\]
The pointwise limit of measurable functions is measurable, so it only remains to check the uniform boundedness of the sequence.
From the Cauchy property, it follows that the sequence $\{y_n\}$ is bounded in $(Y,\si(Y,\DY))$, i.e., 
\[\sup_n \abs{\pa{\dy}{y_n}} < \infty\quad\hbox{for any } \dy \in \DY\] 
and by considering the sequence $\{y_n\}$ in $Y$ as a sequence in $Y^{\diamond\ast}$, the uniform boundedness principle implies that
\[\sup_{n \ge 1} \|y_n\| \ \hbox{ is bounded}.\]
Therefore the sequence $\{y_n\}$ is bounded in the supremum norm and hence the pointwise limit defines a bounded Borel function. 

This shows that $(Y,\si(Y,\DY))$ is sequentially complete.

\smallskip
\textsc{Part III}. In this part we prove that a linear map $(Y,\si(Y,\DY)) \to \BR$ is continuous if it is sequentially continuous.
Let $\La : (Y,\si(Y,\DY)) \to \BR$ be a sequentially continuous linear map. An application of Theorem \ref{thm:sigma-order-dual} shows that in order to prove that $\La$ belongs to $\DY$ it suffices to prove that $\La$ is $\si$-order continuous. 

Let $\{y_n\}$ a sequence in $Y$ that converges to zero in order. To prove that $\La(y_n) \to 0$ we first observe that if $y_n \to 0$ in order then because of \sef{eq:order-convergence1} $y_n$ converges pointwise to zero. Hence $y_n \to 0$ in $\gb{Y,\si(Y,\DY)}$ (see the discussion in the paragraph before Theorem \ref{thm:suit-twin-perI}). Since $\La$ is sequentially continuous it follows that $\La(y_n) \to 0$. This proves that $\La$ is $\si$-order continuous. Thus it follows from the characterization of $\DY$ in Theorem \ref{thm:sigma-order-dual} that $\La$ belongs to $\DY$. This completes the proof that $\La$ is continuous if it is sequentially continuous in $(Y,\si(Y,\DY))$.
\ \QED

\medskip\noindent

Since reflection $[0,1] \ni t \mapsto -t \in [-1,0]$ induces an isometric isomorphism, it follows from Theorem \ref{thm:suit-twin-perI} that  $B\gb{[0,1],\BR^n}$ and $NBV\gb{[-1,0],\BR^n}$ form a norming dual pair as well. Furthermore note that, according to the definition, $(Y,\DY)$ is a norming dual pair if and only if $(\DY,Y)$ is a norming dual pair. Therefore, we also have the following corollary to Theorem \ref{thm:suit-twin-perI}.

\begin{theorem}\label{thm:suit-twin-perII}
The dual pair given by \sef{eq:mot-ex0r} and \sef{eq:mot-ex1r} is a norming dual pair such that \sef{eq:4.5} and \sef{eq:4.8} hold, i.e.,
\begin{itemize}
\item[(i)] a linear map $(\DY,\si(\DY,Y)) \to \BR$ is continuous if it is sequentially continuous.
\item[(ii)] $(\DY,\si(\DY,Y))$ is sequentially complete;
\end{itemize}
\end{theorem}

\medskip

Note that if the dual pair is given by \sef{eq:mot-ex0r} and \sef{eq:mot-ex1r}, then the weak topology $\si(Y,\DY)$ on $Y$ is strictly stronger than the weak$^\ast$ topology on $Y$ as can be seen from the fact that for every $f \in C\gb{[0,1];\BR^n}$
\begin{equation}\label{eq:weak*-top-prop}
\pa{f}{\de_{x_n}} = f(x_n) \to f(x) = \pa{f}{\de_{x}}\qquad \mbox{ as}\quad n \to \infty
\end{equation}
and hence $\de_{x_n} \to \de_{x}$ in the weak$^\ast$ topology on $Y$ if $x_n \to x$ in $[-1,0]$, whereas $\de_{x_n} \not\to \de_x$ in $\si(Y,\DY)$ since \sef{eq:weak*-top-prop} does not hold for every $f \in B\gb{[0,1];\BR^n}$.

\medskip
We end this appendix with some more detailed information about norming dual pairs and their topologies. Given a norming dual pair $(Y,\DY)$, we call a topology $\tau$ on $Y$ {\it consistent} (with the duality) if $\DY$ is the dual space of $(Y,\tau)$. By the Mackey-Arens theorem \cite[Theorem 5.112]{AliBor06}, a consistent topology $\tau$ is finer than the weak topology $\si(Y,\DY)$ and coarser than the Mackey topology $\tau(Y,\DY)$, the finest topology on $Y$ that preserves the continuous dual. Note that the Mackey topology  $\tau(Y,\DY)$ allows the largest collection of continuous functions on $Y$ and all consistent topologies have the same bounded sets \cite[Theorem 6.30]{AliBor06}. 

Furthermore, if $\DY = Y^\ast$, then the Mackey topology $\tau(Y,Y^\ast)$ on $Y$ corresponds to the norm topology on $Y$, cf. \cite[Corollary 6.23]{AliBor06}.

\medskip

For the dual pair given by \sef{eq:mot-ex0} and \sef{eq:mot-ex1}, the topological space $(Y,\tau(Y,\DY))$ has been studied in \cite{BS96,Schaefer89} and plays an important role in the theory of Markov processes, cf. \cite[Chapter 19]{AliBor06} and \cite{Kraaij16}. 
 
A topological space $(Y,\tau(Y,\DY))$ is called {\sl semi-bornological} whenever full and sequential continuity of its linear forms is equivalent, cf. \cite[IV.3, p. 131]{Schaefer86}.

The following result \cite[Proposition E.2.4]{BS96} shows that \sef{eq:4.1} and \sef{eq:4.9} hold as well with respect to the Mackey topology. 

\begin{theorem}\label{thm:P12} 
Let 
\[Y = B\gb{[-1,0],\BR^n}\quad\mbox{and}\quad\DY = NBV\gb{[0,1],\BR^n}.\]
The topological space $(Y,\tau(Y,\DY))$ is semi-bornological and $\tau(Y,\DY)$-sequentially complete.
\end{theorem}

\medskip

In this paper we have formulated our assumptions with respect to the weak topology $\si(Y,\DY)$, but we could have formulated  \sef{eq:4.1} and \sef{eq:4.9} or  \sef{eq:4.5} and \sef{eq:4.8} with respect to any consistent topology and hence, in particular, with respect to the Mackey topology $\tau(Y,\DY)$. This yields, strictly speaking, stronger results. But we feel that the formulation in terms of the weak topology is easier to digest by people working with delay equations.

\end{appendices}

\end{document}